\documentclass[11pt]{amsart}

\usepackage{rlepsf, graphicx, amsthm, amsfonts, amscd,epstopdf}
\usepackage[all]{xy}
\usepackage{hyperref}
\usepackage[usenames,dvipsnames]{color}

\newcommand{\e}{\ensuremath{\eta}} 
\newcommand{\te}{\tilde{\e}}

\newcommand{\pd}[2]{\ensuremath{{\partial_{#2} #1}}}

\DeclareMathOperator{\image}{im}

\DeclareMathOperator{\ind}{Ind}

\def\TM+{T^*(\rr_+ \times M)}
\def\rp{\mathbb R_+}

\newcommand{\rr}{\ensuremath{\mathbb{R}}}
\newcommand{\zz}{\ensuremath{\mathbb{Z}}}

\newcommand{\ff}{\ensuremath{\mathbb{F}}}

\theoremstyle{plain}
\newtheorem{thm}{Theorem}[section]
\newtheorem{cor}[thm]{Corollary}
\newtheorem{lem}[thm]{Lemma}
\newtheorem{slem}[thm]{Sublemma}

\newtheorem{prop}[thm]{Proposition}

\theoremstyle{definition}
\newtheorem{defn}[thm]{Definition}

\newtheorem*{oques}{Open Question}

\theoremstyle{remark}
\newtheorem{rem}[thm]{Remark}
\newtheorem{ex}[thm]{Example}

\numberwithin{equation}{section} 

\def\dfn#1{{\textbf {#1}}}
\newcommand{\gh}[2]{\ensuremath{\widetilde{GH}\vphantom{H}^{#1}(#2)}}
\newcommand{\rgh}[2]{\ensuremath{{GH}^{#1}(#2)}}

\newcommand{\wgh}[2]{\ensuremath{\widetilde{WGH}\vphantom{H}^{#1}(#2)}}
\newcommand{\rwgh}[2]{\ensuremath{{WGH}^{#1}(#2)}}

\newcommand{\lmax}{\ensuremath{\overline{\ell}}}
\newcommand{\lmin}{\ensuremath{\underline{\ell}}}
\newcommand{\lmaxp}{\ensuremath{\lmax_+}}
\newcommand{\lmaxpm}{\ensuremath{\lmax_\pm}}
\newcommand{\lmaxm}{\ensuremath{\lmax_-}}
\newcommand{\lminp}{\ensuremath{\lmin_+}}
\newcommand{\lminpm}{\ensuremath{\lmin_\pm}}
\newcommand{\lminm}{\ensuremath{\lmin_-}}

\newcommand{\lingf} {\mathcal F^{\operatorname{lin}}}
\newcommand{\grad}{\operatorname{grad}}
\newcommand{\id}{\operatorname{id}}

\def\infin{\omega}
\def\Infin{\Omega}
\def\leg{\Lambda}
\def\sclag{\overline{L}}
\def\lag{\mathcal L}
\def\clag{\overline{\mathcal L}}
\def\slag{L}
\def\lag{\mathcal L}
\def\leg{\Lambda}

\usepackage{pdfsync} 

\begin{document}

\title[Lagrangian Cobordisms between Legendrians]
{Obstructions to  Lagrangian Cobordisms between 
Legendrians via Generating Families }

\date{\today}

\author[J. Sabloff]{Joshua M. Sabloff} \address{Haverford College,
Haverford, PA 19041} \email{jsabloff@haverford.edu} \thanks{JS is
partially supported by NSF grant DMS-0909273.}

\author[L. Traynor]{Lisa Traynor} \address{Bryn Mawr College, Bryn
Mawr, PA 19010} \email{ltraynor@brynmawr.edu} \thanks{LT is
partially supported by NSF grant DMS-0909021.}

\begin{abstract} The technique of generating families produces obstructions to the existence of embedded Lagrangian cobordisms between Legendrian submanifolds in the symplectizations of $1$-jet bundles. In fact, generating families may be used to construct a TQFT-like theory that, in addition to giving the aforementioned obstructions, yield structural information about invariants of Legendrian submanifolds.  For example, the obstructions devised in this paper show that there is no generating family compatible Lagrangian cobordism between the Chekanov-Eliashberg Legendrian $m(5_2)$ knots.  Further, the generating family cohomology groups of a Legendrian submanifold restrict the topology of a Lagrangian filling.  Structurally, the generating family cohomology of a Legendrian submanifold satisfies a type of Alexander duality that, when the Legendrian is null-cobordant, can be seen as Poincar\'e duality of the associated Lagrangian filling.  This duality implies the Arnold Conjecture for Legendrian submanifolds with linear-at-infinity generating families.  The results are obtained by developing a generating family version of wrapped Floer cohomology and establishing long exact sequences that arise from viewing the the spaces underlying these cohomology groups as mapping cones.  \end{abstract}

\maketitle


\section{Introduction}
\label{sec:intro}

\subsection{Motivation and Questions}

While the notions of cobordism and concordance of submanifolds have been influential in topology since the 1950s, their introduction into contact and symplectic geometry is more recent.  Arnold made the first steps in his study of geometric optics \cite{arnold:cobordism-1, arnold:cobordism-2} by introducing immersed Lagrangian cobordisms between Lagrangian submanifolds.  Over the last two decades, a greater understanding of Legendrian submanifolds --- and even the underlying smooth submanifolds --- was obtained by studying Lagrangian submanifolds of a symplectic manifold with a contact boundary. Rudolph \cite{rudolph}, for example, employed the gauge theoretic techniques of Kronheimer and Mrowka \cite{km:thom} to show that the Thurston-Bennequin invariant of a Legendrian knot gives a bound on the $4$-ball genus of a knot.  More recently, the advent of Eliashberg-Givental-Hofer's Symplectic Field Theory framework \cite{egh} has shifted attention to the use of Lagrangian cobordisms in the symplectization of a contact manifold $X$ to study, via pseudo-holomorphic curves, the geometry and geography of Legendrian submanifolds of $X$, and eventually the underlying smooth submanifolds.

The questions of interest in this paper fall into two families:

\begin{quote}{\bf Existence of Obstructions}: What are the obstructions to one Legendrian submanifold being Lagrangian cobordant to another? In particular, what are the obstructions to a Legendrian submanifold having a Lagrangian filling, i.e., being null-cobordant?
\end{quote}

These questions were first approached by Chantraine \cite{chantraine}: employing an adjunction inequality obtained through gauge theory, he showed that classical invariants of Legendrian submanifolds can provide obstructions to the existence of cobordisms.  One goal of this paper is to strengthen Chantraine's results, as well as Golovko's higher dimensional version of Chantraine's work \cite{golovko:tb}, using non-classical invariants derived from generating families.

Reversing the flow of information from cobordism to invariant, we may also ask: 

\begin{quote}
 {\bf Structure of Invariants:} What can a Lagrangian cobordism tell us about the meaning of Legendrian invariants?  How can Lagrangian cobordism be used to explain deeper structure in the invariants?
\end{quote}

As described more precisely below, the existence of a certain type of Lagrangian filling will, for example, impose conditions on the generating family invariants of a Legendrian submanifold; furthermore, a duality present in these invariants comes from Poincar\'e duality of the Lagrangian filling.

In this paper, we study embedded Lagrangian cobordisms between Legendrian submanifolds in the symplectizations of $1$-jet bundles, which are classical examples of contact manifolds; throughout, we assume that the Legendrians and the Lagrangian have compatible generating families, as defined below in Section~\ref{sec:cobord-gf}.  Generating families have previously been used to define non-classical invariants of certain Legendrian submanifolds of $1$-jet bundles and Lagrangian submanifolds of cotangent bundles; see \cite{f-r, lisa-jill, lisa:links} and Sections~\ref{sec:gh-legendrian} and \ref{sec:cobord-gf}, below.  Not every Legendrian submanifold has a generating family, though in $J^1\rr$ at least, the existence of a generating family is equivalent to the existence of an augmentation of the Chekanov-Eliashberg DGA \cite{fuchs:augmentations, fuchs-ishk, rulings}.  The condition of a Lagrangian cobordism having a compatible generating family is not yet well-understood. There do exist several methods for constructing generating family-compatible Lagrangian cobordisms, however: Legendrian isotopy, spinning, and, most importantly, attaching Lagrangian handles; see \cite{bst:construct}. In addition, any Legendrian submanifold with a generating family has a (potentially immersed) generating family-compatible Lagrangian filling \cite{bst:construct}.  At this time, it would not be unreasonable to conjecture that the study of generating family-compatible Lagrangian cobordisms between Legendrian submanifolds is tantamount to the study of zero-Maslov Lagrangian cobordisms between Legendrian submanifolds with DGAs admitting augmentations.

There are several reasons to study Lagrangian cobordisms through the technique of generating families.  The generating family technique, which employs classic analysis and Morse-theoretic techniques, is analytically simpler than holomorphic curve techniques.  Further, as mentioned above, results of \cite{bst:construct} show that generating family-compatible Lagrangian cobordisms are plentiful and easily constructed.  Given the potential complexity of the theory of Lagrangian cobordisms, it seems reasonable to begin its study in this more tractable setting.  

It is also interesting to compare results obtained though generating families and holomorphic curves.  For a number of results in this paper, there is a parallel story to be told for invariants defined through the theory of holomorphic curves, due to Ekholm, Honda, and K\'alm\'an \cite{ehk:leg-knot-lagr-cob}, Ekholm \cite{ekholm:rsft, ekholm:lagr-cob, ekholm:rsft-survey}, Golovko \cite{golovko:tb}, the first author \cite{duality}, and Ekholm, Etnyre, and the first author \cite{high-d-duality}.  The holomorphic projects of \cite{ehk:leg-knot-lagr-cob} and \cite{ekholm:lagr-cob} were initiated before we began this paper and deeply inspired us; the work of Golovko \cite{golovko:tb} which is based on these projects appeared as we were putting this finishing touches on this paper.  Holomorphic techniques have the advantage that they apply in more general settings and to more Legendrian and Lagrangian submanifolds.  
At the time of this writing, however, a number of the holomorphic curve based results are at the stage in which the shape of the theory is known in detail, and there is a fairly complete sketch of the difficult analysis required; see especially \cite{ekholm:rsft}. Generating family techniques can be seen as a way to more easily establish some results in standard settings and potentially provide intuition for phenomena that may occur in a more general setting.

\subsection{Main Results}
\label{ssec:results}

Let $M$ be a compact manifold (or $\rr^n$), and denote by $J^1M$ its $1$-jet space.  Let $\sclag$ be an embedded Lagrangian submanifold in the symplectization $\rr \times J^1M$ that is cylindrical over Legendrian submanifolds $\leg_\pm$ of $J^1M$ outside a compact interval of $\rr$; we denote such a cobordism by
$\leg_- \prec_{\sclag} \leg_+ $.
We assume that, in a sense to be described precisely in Section~\ref{sec:cobord-gf}, the cobordism $\sclag$ \footnote{more precisely, the image of the cobordism $\theta(\sclag) = \clag  \subset T^*(\rp \times M)$}
 has a generic slicewise-linear-at-infinity generating family $F$  that is compatible with the generic linear-at-infinity generating families $f_\pm$ for $\leg_\pm$ at the ends;  
 we refer to such a cobordism as a gf-compatible Lagrangian cobordism and
 denote it by   $(\leg_-, f_-) \prec_{(\sclag, F)} (\leg_+, f_+)$.  
The order is important: as we shall see below, Lagrangian cobordism is not a symmetric relation.

The relative (resp. total) generating family cohomology groups $\rgh{*}{f}$ (resp.\  $\gh{*}{f}$) for a linear-at-infinity generating family $f$ of a Legendrian submanifold $\leg \subset J^1M$ are defined to be the relative cohomology groups of pairs of sublevels sets of a difference function associated to $f$.  This idea has been explored in \cite{f-r, lisa-jill, lisa:links}; see Section~\ref{sec:gh-legendrian} for details.  
While the generating family cohomology may depend on the specific generating family $f$, the set of all cohomology groups for all linear-at-infinity generating families for a Legendrian submanifold $\leg$ is an invariant of the isotopy class of $\leg$.  The set of generating families of a fixed Legendrian submanifold may be simplified using a notion of equivalence to be defined in Section~\ref{sec:gf}: the generating family cohomologies all descend to equivalence classes.

The key idea in this paper is that introducing Lagrangian cobordisms into the theory of generating family cohomology gives rise to a TQFT-like structure.  In the following theorem, we let $\slag$ denote the compact portion of $\sclag$, noting that the boundary of $\slag$ is $\leg_- \cup \leg_+$.

\begin{thm} \label{thm:cobord-les}
  If $(\leg_-, f_-) \prec_{(\sclag, F)} (\leg_+, f_+)$ and $\sclag$ is orientable, then there exists a homomorphism $\Psi_F: \rgh{k}{f_-} \to \rgh{k}{f_+}$ that fits into the following long exact sequence:
  \begin{equation} \label{eqn:cobord-les}
    \xymatrix@1{
      \cdots \ar[r] & \rgh{k}{f_-} \ar[r]^{\Psi_F} & \rgh{k}{f_+} \ar[r] & H^{k+1}(\slag, \leg_+) \ar[r] & \cdots.
    }
  \end{equation}
\end{thm}

In this and all theorems below, the hypothesis of orientability may be dropped if $\zz_2$ coefficients are used.  The cobordism map $\Psi_F$ satisfies some of the typical properties of a TQFT such as non-triviality, naturality, and functoriality; see Section~\ref{sec:tqft}.  The existence of a cobordism map in the holomorphic curved based theory of Legendrian contact homology setting is explored  in \cite{ekholm:rsft, ehk:leg-knot-lagr-cob}, and the existence of a long exact sequence involving the cobordism map is examined in \cite{golovko:tb}.

Taking Euler characteristics of the long exact sequence (\ref{eqn:cobord-les}) yields a generalization of Chantraine's $3$-dimensional result about the relationships between the Thurston-Bennequin invariants of the Legendrian knots at the ends of a Lagrangian cobordism \cite{chantraine} as well as of Golovko's generalization \cite{golovko:tb}.

\begin{cor}\label{cor:chantraine}
If  $(\leg_-, f_-) \prec_{(\sclag, F)} (\leg_+, f_+)$,  $\sclag$ is orientable, 
 and the $n$-dimensional Legendrians $\leg_\pm \subset J^1\rr^n$ are generic, then
  $$tb(\leg_+) - tb(\leg_-) = (-1)^{\frac{1}{2}(n^2 - 3n)} \chi(\slag, \leg_+).$$
\end{cor}

The asymmetry of the Lagrangian cobordism relation is evident from this corollary.  If we only consider cobordisms that are actually concordances, then we get:

\begin{cor} \label{cor:cyl-iso}
 If    $(\leg_-, f_-) \prec_{(\sclag, F)} (\leg_+, f_+)$, and $\sclag$ is orientable and  
  diffeomorphic to $\rr \times \leg$, then $\Psi_F$ is an isomorphism.
\end{cor}

\begin{figure}
\centerline{\includegraphics[width=4in]{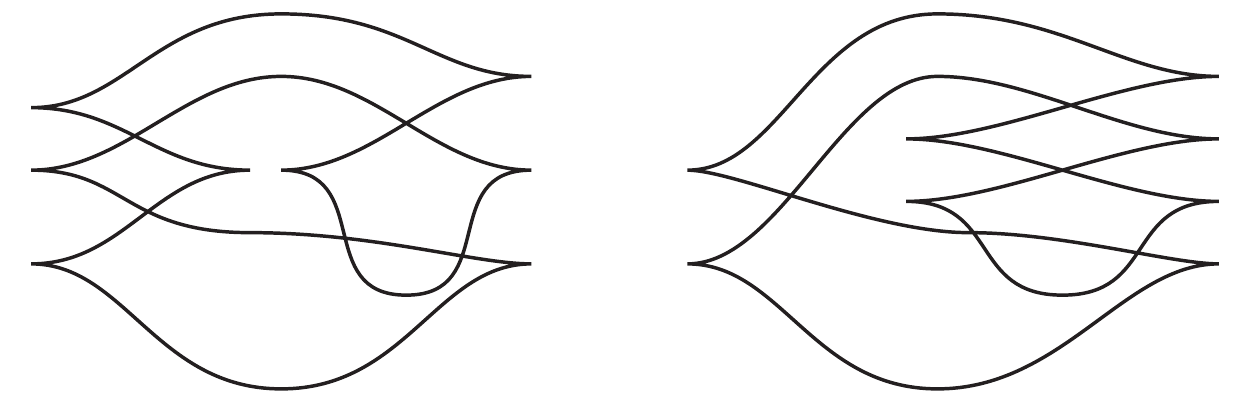}}
\caption{Front projections of the Chekanov-Eliashberg examples $K_1$ and $K_2$.}
\label{fig:chv-ex}
\end{figure}

\begin{ex} \label{ex:ch-e} There is no gf-compatible Lagrangian cobordism between
the  Chekanov-Eliashberg Legendrian $m(5_2)$ knots $K_1$ and $K_2$  pictured in Figure~\ref{fig:chv-ex} in either order. To see this, first notice that by using the established connection between the existence of a ruling and the existence of a generating family \cite{chv-pushkar, f-r},
it is easy to see that $K_1$ and $K_2$  have linear-at-infinity generating families.  Moreover, 
using Fuchs and Rutherford's connection between generating family homology and linearized Legendrian contact homology \cite{f-r}, it is straightforward to compute that, for any linear-at-infinity generating families $f_i$ of $K_i$, the Poincar\'e polynomials of the generating family cohomologies with coefficients in $\zz_2$ are given by:
\begin{equation} \label{eqn:chv-el}
P_{f_1}(t) = 2+t \quad \text{and} \quad P_{f_2}(t) = t^{-1} + t + t^2.
\end{equation}
Thus, $K_1$ and $K_2$ 
are not Legendrian isotopic.  Since their Thurston-Bennequin invariants agree, Corollary~\ref{cor:chantraine} implies that a gf-compatible cobordism between them would necessarily be a concordance, but Corollary~\ref{cor:cyl-iso} forbids this.
\end{ex}
Additional examples are given in Theorem~\ref{thm:ex}.

If we restrict to the case where $\leg_- = \emptyset$, i.e.\ to when $\leg_+$ has a Lagrangian filling, then Theorem~\ref{thm:cobord-les} yields a strong restriction on the topology of the filling.  In fact, we show that both the relative and total generating family cohomology detect the topology of the filling:

\begin{thm} \label{thm:filling-iso}  
If $(\leg_-, f_{-}) \prec_{(\sclag, F)} (\leg_{+}, f_{+})$, then
 $$\rgh{k}{f_+} \simeq H^{k+1}(\slag, \leg_+) \quad \text{and} \quad \gh{k}{f_+} \simeq H^{k+1}(\slag).$$
\end{thm}

The geometric framework for a parallel result involving the
holomorphic-curve-based Legendrian contact homology appears in \cite{ekholm:lagr-cob}.

\begin{ex} Returning to the $m(5_2)$ knots of Example~\ref{ex:ch-e}, we see that for any generating family, the knot $K_2$ cannot have a compatible Lagrangian filling.  Further, any gf-compatible Lagrangian filling for $K_1$ must be homeomorphic to a punctured torus.
In \cite{bst:construct}, it is shown that such a punctured torus filling for $K_{1}$ indeed
exists. \end{ex}

Theorems~\ref{thm:cobord-les} and \ref{thm:filling-iso} arise from a study of an adaptation of wrapped Floer homology --- as introduced by Abbondondolo and Schwarz \cite{abb-schw:cotangent} and developed further by Fukaya, Seidel and Smith \cite{fss:wrapped-floer} and by Abouzaid and Seidel \cite{as:wrapped-floer} --- to the generating family setting.  The \dfn{relative (resp.\ total) wrapped generating family cohomology} $\rwgh{*}{F}$ (resp.\ $\wgh{*}{F}$) of a Lagrangian cobordism 
$(\leg_-, f_-) \prec_{(\sclag, F)} (\leg_+, f_+)$   is defined as the relative cohomology of a pair of sublevel sets of a ``sheared'' difference function associated to $F$.  This cohomology has a cochain complex with generators that can be identified with the self-intersections of $\sclag$ and the Reeb chords of the Legendrians at the ends; see Section~\ref{sec:cobord-gf}.  The proofs of the aforementioned theorems have similar outlines: the spaces underlying the relative and total wrapped generating family cohomology can be viewed as mapping cones and thus give rise to long exact sequences. In fact, the total wrapped generating family cohomology vanishes, and thus the corresponding long exact sequence gives rise to the isomorphism in Theorem~\ref{thm:filling-iso}.  These proofs have parallels to the work of Ekholm \cite{ekholm:lagr-cob} for Legendrian contact homology and Bourgeois and Oancea \cite{bo:exact-seq} for closed contact homology.

The theorems above can be applied to analyze the following examples: Legendrian negative twist knots in $\rr^3$, which have recently been classified by Etnyre, Ng, and Vertesi, \cite{env:twist}; the higher dimensional non-isotopic Legendrians studied by Ekholm, Etnyre, and Sullivan in \cite{ees:high-d-geometry}; and the Legendrian knot studied by Melvin and Shrestha \cite{melvin-shrestha} which has augmentations that lead to different linearized contact homologies.

\begin{thm} \label{thm:ex}  
\begin{enumerate}
\item  For all $n \geq 1$, each of the $n$ Legendrian representatives of the odd, negative twist knot $K_{-2n-1}$ with maximal Thurston-Bennequin invariant, as described in Figure~\ref{fig:twist-knot},
has a linear-at-infinity generating family, but none have a gf-compatible Lagrangian filling.  
Moreover, there is no  gf-compatible Lagrangian cobordism between any two of these  
  Legendrian versions
of $K_{-2n-1}$.
\item  The smoothly equivalent but non-Legendrian-isotopic surfaces $\leg_0$ and $\leg_1$ pictured in Figure~\ref{fig:high-d-ex} have linear-at-infinity generating families and the same classical invariants
but are not gf-compatibly Lagrangian cobordant.  Furthermore, $\leg_1$ does not have a gf-compatible Lagrangian filling.  
\item  The Legendrian $m(8_{21})$ knot shown in Figure~\ref{fig:ms},
 has two linear-at-infinity generating families $f_0$ and $f_1$ that do not have a
gf-compatible Lagrangian cobordism compatible with the pair $\{f_0, f_1\}$ at the ends.
\end{enumerate}
\end{thm}

 Similar arguments produce examples  of topologically equivalent Legendrian negative, even twist knots that are not gf-compatibly Lagrangian cobordant and examples of higher dimensional Legendrians with the same classical invariants that are not
gf-compatibly Lagrangian cobordant; see Section~\ref{sec:ex}.

\begin{figure}
  \centerline{\includegraphics[width=3in]{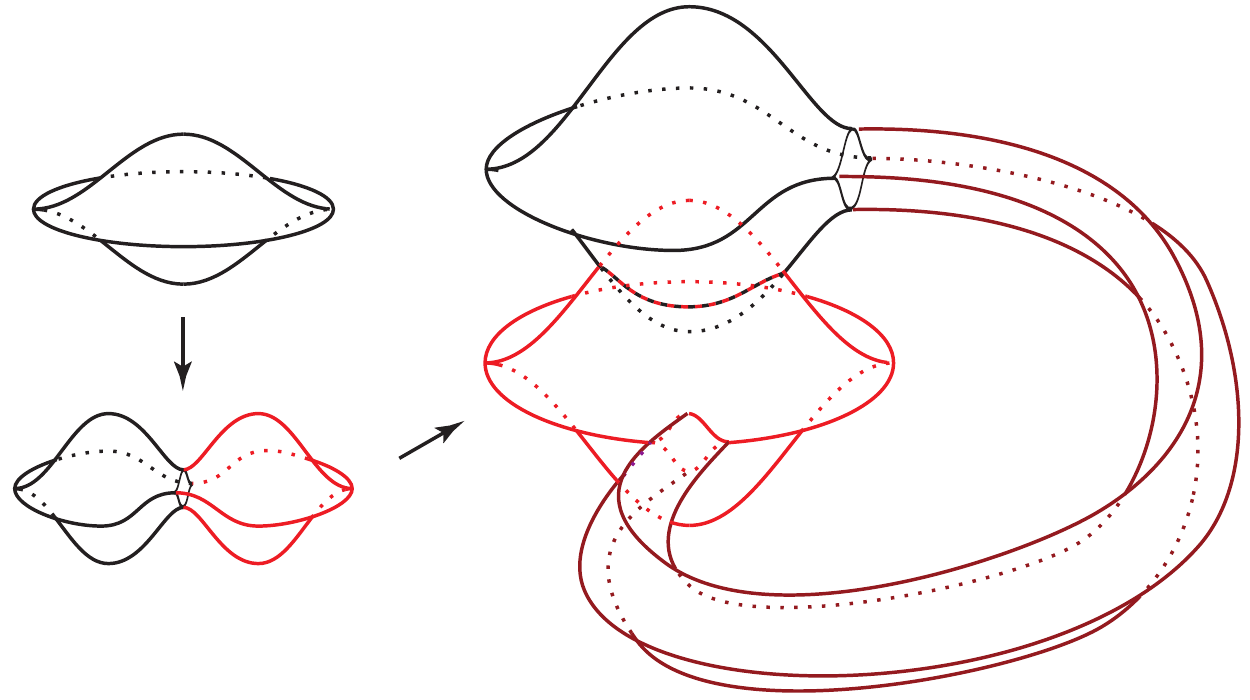}}
 \caption{ $\leg_0$ is the Legendrian with the ``flying saucer" front projection shown in the upper left portion of this diagram;
 $\Lambda_1$ is constructed by squeezing the front of $\Lambda_0$ into a dumbbell shape  and then doing a helical rotation
 of the connecting tube so that  that dumbbell ends are overlapping.}
 \label{fig:high-d-ex}
\end{figure}

\subsection{Interactions with Duality}
\label{ssec:duality-intro}

Given the isomorphisms in Theorem~\ref{thm:filling-iso} and the fact that the Lagrangian fillings will satisfy Poincar\'e duality, one sees that, for Legendrians that are null-cobordant, there is a duality in the generating family cohomology groups.  It is natural to ask whether there is generating family duality for more general Legendrians and whether this duality for Legendrians which are null-cobordant can be geometrically interpreted as Poincar\'e duality.

There is a well-developed theory of duality for the holomorphic curved based linearized contact homology of a horizontally-displaceable Legendrian submanifold in dimensions three \cite{duality} and higher \cite{high-d-duality}; here, horizontally displaceable means that the Lagrangian projection of the Legendrian submanifold is Hamiltonian displaceable.  In dimension three, the duality says that there is an isomorphism between the linearized contact homology groups in degrees $\pm k$ when $k \neq 1$; when $k = 1$, the isomorphism is offset by the presence of a fundamental class of index $1$.  In higher dimensions, the duality statement takes the form of a long exact sequence showing that up to a fixed ``error term'', which depends only on the topology of the $n$-dimensional Legendrian submanifold, there is an isomorphism between $k$-dimensional homology classes and $(-k + (n-1))$-dimensional cohomology classes.  Fuchs and Rutherford \cite{f-r} showed that in the isomorphism between generating family and linearized contact homology groups, the three-dimensional duality for the linearized contact homology corresponds to a version of Alexander duality for the generating family homology.  The following theorem refines their statement of duality and generalizes it to higher dimensions in parallel to the results of \cite{high-d-duality}.

\begin{thm} \label{thm:duality} If $\leg$ is an Legendrian submanifold of $J^1M$ with linear-at-infinity generating family $f$, then there is a long exact sequence:
  \begin{equation*}
    \xymatrix@R=5pt@C=12pt{
     \cdots \ar[r] & GH^{k-1}(f) \ar[r]^\phi & GH_{n-k}(f)
      \ar[r] & H^{k}(\Lambda) \ar[r]
      & \cdots   }
  \end{equation*}
\end{thm}

\begin{rem}
  Having a linear-at-infinity generating family and being horizontally displaceable are not equivalent conditions on a Legendrian submanifold.  For example, it is straightforward to construct a linear-at-infinity generating family for the Legendrian knot $K$ in $J^1S^1$ pictured in Figure~\ref{fig:non-disp}, 
  however, this knot is not horizontally-displaceable. To see why, suppose for the sake of contradiction that $K$ were, indeed, horizontally-displaceable.  Then there exists a Hamiltonian displaceable neighborhood $U$ of the Lagrangian projection of $K$ in $T^*S^1$.  It it easy to draw a section of $T^*S^1$ inside $U$, so our assumption would imply that a section of $T^*S^1$ is Hamiltonian displaceable, which is a impossible. Thus, Theorem~\ref{thm:duality} above and Corollary~\ref{cor:arnold} below capture different Legendrian submanifolds than does the theory of \cite{high-d-duality}.  
\end{rem}

\begin{figure}
  \centerline{\includegraphics[height=.8in]{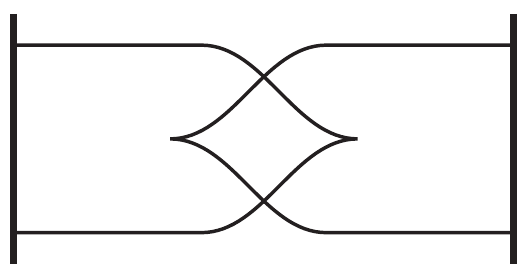}}
  \caption{The Legendrian submanifold of $J^1S^1$ represented by this front diagram has a linear-at-infinity generating family, but is not horizontally displaceable.}
  \label{fig:non-disp}
\end{figure}

The Arnold conjecture for Legendrian submanifolds, which states that the number of Reeb chords of a generic Legendrian submanifold $\leg \subset J^1\rr^n$ with respect to the standard contact form is bounded from below by half the sum of the Betti numbers of $\leg$, was proven for horizontally displaceable Legendrian submanifolds with linearizable contact homology in \cite{ees:ori} and refined in \cite{high-d-duality}. Analogously, Theorem~\ref{thm:duality} easily leads to a refined version of the Arnold Conjecture for Legendrian submanifolds with linear-at-infinity generating families:

\begin{cor} \label{cor:arnold} Let $r_i(\leg)$ denote the number of Reeb chords of $\leg$ of Conley-Zehnder index $i$.  If $\leg$ is a
generic, $n$-dimensional Legendrian submanifold of $J^1M$ with linear-at-infinity generating family  $f$, then:  
  $$r_i(\leg) + r_{n-i}(\leg) \geq b_i(\leg; \ff)$$
  for $0 \leq i \leq n$, where $b_i(\leg; \ff)$ is the $i^{th}$ Betti number of $K$ over a field $\ff$. 
\end{cor}

Results from Theorems~\ref{thm:filling-iso} and \ref{thm:duality} can be combined to obtain a relationship between the ``Alexander duality'' map $\phi$ in Theorem~\ref{thm:duality} and the Poincar\'e duality map for a Lagrangian filling $L$.

\begin{thm} \label{thm:duality-filling}  Suppose $(\emptyset, f_-) \prec_{(\sclag, F)} (\leg_+, f_+)$, and $\sclag$ is orientable. 
  Then the following diagram commutes, where the bottom sequence comes from Theorem~\ref{thm:duality}, the top sequence comes from the long exact sequence of the pair $(\slag, \leg_+)$ and Poincar\'e-Lefschetz duality, and the vertical maps arise from Theorem~\ref{thm:filling-iso}.
 \begin{equation*}
    \xymatrix@R=3pt@C=15pt{
      & & H^k(L,\leg_+) \ar[r]^-e \ar[dd]^{\simeq} & H_{n+1-k}(L,\leg_+) 
      \ar[dr] \ar[dd]^\simeq & & \\
      \cdots \ar[r] & H^{k-1}(\leg_+) \ar[ur] \ar[dr] & & & H^{k}(\leg_+) \ar[r]
      & \cdots \\
      & & GH^{k-1}(f) \ar[r]^{\phi} & GH_{n-k}(f) \ar[ur] & &
    }
  \end{equation*}
\end{thm}

Ekholm discusses a parallel statement for Legendrian contact homology in \cite{ekholm:lagr-cob}.

\subsection{Open Questions}  The results above open a number of questions for future research.
  \begin{enumerate}
\item   Every Lagrangian cobordism of $\rr \times J^1M$ with a generating family will
be an exact Lagrangian with Maslov index $0$. When does an exact, Maslov  $0$, Lagrangian cobordism of $\rr \times J^1M$ 
have a generating family?  
 \item More specifically, through augmentations and rulings, we have algebraic and combinatorial ways to detect the existence of a
tame generating family for a Legendrian knot in $\rr^3$.  Are there algebraic and/or
combinatorial ways to detect the existence of a tame generating family for a higher dimensional Legendrian or for a 
Lagrangian cobordism in $\rr \times J^1M$?
 \item It can be easily shown that a Legendrian with a tame generating family always
has an immersed Lagrangian filling with a tame, compatible generating family.  What are obstructions to
removing  double points of a given index?
\item It is known that a Lagrangian filling will minimize the smooth $4$-ball genus.  Does an exact, Maslov 0,  Lagrangian cobordism
between two Legendrian knots minimize the smooth $4$-genus of such a cobordism?  
\end{enumerate}

\subsection{Plan of the Paper}

In the next section, we will briefly review some background on the theory of generating families.  In Sections~\ref{sec:gh-legendrian} and \ref{sec:cobord-gf}, we give precise definitions and prove basic properties for various flavors of generating family (co)homology groups for Legendrian submanifolds of $1$-jet bundles and Lagrangian cobordisms in the symplectization.  Section~\ref{sec:map-cones} lays out the technical tools necessary to prove our main theorems.  These theorems are then proved in the next four sections: we examine the special case of Lagrangian fillings in Section~\ref{sec:filling-iso}; discuss duality in Section~\ref{sec:duality}; prove the main theorems about the cobordism map and its associated long exact sequence in Section~\ref{sec:cobord-les}; and finally study the TQFT-like properties of the cobordism map in Section~\ref{sec:tqft}.  In Section ~\ref{sec:ex}, we end by giving additional examples of 
Legendrian submanifolds that are not Lagrangian cobordant, and pose some open questions.

\subsection*{Acknowledgements}

The work in this paper was launched in response to a conjecture of Paul Seidel that predicted a holomorphic analogue of Theorem~\ref{thm:filling-iso} based on preliminary results of Ekholm, Honda, and K\'alm\'an; we are also indebted to Ekholm, Honda, and K\'alm\'an's pioneering work.  We thank Mohammed Abouzaid, Fr\'ed\'eric Bourgeois, 
 Baptiste Chantraine, 
Roman Golovko, Paul Melvin, and Dan Rutherford  for several stimulating discussions.  Finally, we thank the participants in the Thematic Week on Generating Families during the Special Trimester on Contact and Symplectic Topology at the Universit\'e de Nantes for their feedback on the material contained in this paper.

\section{Generating Family Background}
\label{sec:gf}

In this section, we discuss the background necessary for working with generating families for Lagrangian and Legendrian submanifolds.  The germ of the idea comes from the following simple observation: given a function $f: B \to \rr$,
the graph of $df$ in $T^*B$ is a Lagrangian submanifold and the $1$-jet of $f$ is a Legendrian submanifold of $J^1B$.  Generating families extend this construction to ``non-graphical'' Lagrangians and Legendrians by expanding the domain to, for example, the trivial vector bundle $B \times \rr^N$ for some potentially large $N$.  We will denote the fiber coordinates by $\e = (\e_1, \ldots, \e_N)$. What follows are bare-bones definitions so as to set notation; see \cite{theret:viterbo, lisa:links, viterbo:generating} for more details.

Suppose that we have a smooth function $f: B^b \times \rr^N \to \rr$ such that $\mathbf{0}$ is a regular value of the map $\partial_{\e} f: B \times \rr^N \to \rr^N$.  We define the \dfn{fiber critical set} of $f$ to be the $b$-dimensional submanifold $\Sigma_f = (\partial_\eta f)^{-1}(\mathbf{0})$.  Define immersions $\partial_f: \Sigma_f \to T^*B$ and $j_f: \Sigma_f \to J^1B$ in local coordinates by: \begin{align*}
  \partial_f(x,\e) &= (x,\partial_xf(x,\e)), \\
  j_f(x,\e) &= (x, \partial_x f(x,\e), f(x,\e)).
\end{align*}
The image $L$ of $\partial_f$ is an immersed Lagrangian submanifold, while the
image $\leg$ of $j_f$ is an immersed Legendrian submanifold.  We say
that $f$ \dfn{generates} $L$ and $\leg$, or that $f$ is a
\dfn{generating family (of functions)} for $L$ and $\leg$.

Two functions $f_i: B \times \rr^{N_i} \to
\rr$, $i= 0, 1$, are \dfn{equivalent} (denoted $f_0 \sim f_1$) if they can be made equal after the operations fiber-preserving diffeomorphism and
stabilization, which are defined as follows:
\begin{enumerate}
\item Given a function $f: B \times \rr^N \to \rr$, let $Q:
  \rr^K \to \rr$ be a non-degenerate quadratic function.  Define $f
  \oplus Q: B \times \rr^N \times \rr^K \to \rr$ by $f \oplus Q(x,
  \e, \e') = f(x, \e) + Q(\e')$.  Then $f \oplus Q$ is a (dimension $K$)
  \dfn{stabilization} of $f$.
\item Given a function $f: B \times \rr^N \to \rr$, suppose
  $\Phi: B \times \rr^N \to B \times \rr^N$ is a fiber-preserving
  diffeomorphism, i.e., $\Phi(x,\e) = (x, \phi_x(\e))$ for a smooth family of 
  diffeomorphisms $\phi_x$. Then $f \circ \Phi$ is said to be
  obtained from $f$ by a \dfn{fiber-preserving diffeomorphism}.
\end{enumerate}
Given a function $f$, denote by $[f]$ its equivalence class with
respect to these two operations.

It is easy to see that if $f: B \times \rr^N \to \rr$ is a generating family for a Lagrangian (Legendrian) $L$ ($\leg$), then any $f' \in [f]$ will also be a generating family for $L$ ($\leg$).  While a Lagrangian or Legendrian submanifold with a generating family will always have an infinite number of generating families, the set of equivalence classes may be more tractable.

\begin{rem}  When dealing with generating families for Lagrangians, it is also common to include the addition of a constant
in the notion of equivalence.  Our definition of equivalence comes from the fact that   the Lagrangians we consider will be ``cylindrical'' over Legendrians. 
\end{rem}

In the next two sections, we will define invariants of Legendrian and Lagrangian submanifolds by applying Morse-theoretic constructions to ``difference functions'' associated to generating families.  The following three Morse-theoretic lemmas will be essential in defining and working with the invariants.  In order to apply these three lemmas, we will often work in the setting where our generating families are \dfn{tame}, meaning ``linear-at-infinity'' as defined in Definition~\ref{defn:li} or ``slicewise-linear-at-infinity'' as defined in Definition~\ref{defn:sli}. Equivalence and tameness of generating families will imply equivalence and tameness of the associated difference functions.

The first Morse-theoretic lemma tells us that the relative cohomology and homology of sublevel sets of a function remain unchanged under equivalence, perhaps up to a shift in degree.

\begin{lem} \label{lem:equiv-he} If $g_0 \sim g_1$, then there exists $q \in \zz$ so that for all $a < b$, the relative (co)homologies of the pairs of sublevel sets $(g_0^b, g_0^a)$ and $(g_1^b, g_1^a)$ are isomorphic up to a shift in degree by $q$.
\end{lem} 

A proof of this statement can be found in \cite[Lemma 4.7]{josh-lisa:cap}.  In fact, if the $g_i$ differ only by a fiber-preserving diffeomorphism, then there is no shift in degree. If $g_0$ differs from $g_1$ by a stabilization, however, the shift is precisely the index of the quadratic $Q$.

The second Morse-theoretic lemma is an extension of the key deformation lemma in Morse theory to some functions with non-closed domains.

\begin{lem} \label{lem:retract} 
 If there
is an integrable, gradient-like vector field for $g: B \times \rr^N \to \rr$ that is bounded away from $\mathbf 0$ on the set $g^{-1}[a, b]$, then the sublevel set $g^a$ is a deformation retract of the sublevel set $g^b$.  
\end{lem}

The proof of this lemma is a straightforward extension of the usual key lemma where the domain of $g$ is closed; see, for example, \cite{milnor:morse}. 
The main idea is that since the vector field is integrable and bounded away from $\mathbf 0$, the flow of the
normalized vector field is defined and  will give a deformation of one sublevel set to another.  
The final Morse-theoretic lemma will prove useful in proving invariance and independence properties in subsequent sections.

\begin{lem}[Critical Non-Crossing; cf., \cite{lisa:links}] \label{lem:crit-non-crossing}  If a continuous $1$-parameter family of functions
  $g_s: B \times \rr^{N} \to \rr$, $s \in [0, 1]$, and continuous paths $\alpha, \beta: [0,1] \to \rr$, with $\alpha(s) \leq \beta(s)$, satisfy:
\begin{quote}
There exists $\epsilon > 0$ so that, for all $s$,  there exists an integrable, gradient-like vector field
$X_s$ for $g_s$ which is bounded away from $\mathbf 0$ on 
\begin{equation} \label{eqn:cnc}
g_s^{-1}\left( [\alpha(s) - \epsilon, \alpha(s) + \epsilon] \cup [\beta(s) - \epsilon, \beta(s) + \epsilon] \right),
\end{equation}
\end{quote}
then  $(g_0^{\beta(0)}, g_0^{\alpha(0)}) $ is homotopy equivalent to $(g_1^{\beta(1)}, g_1^{\alpha(1)})$. 
\end{lem}

\begin{proof} It suffices to show that, for all $t \in [0, 1]$, there exists a neighborhood $U(t)$ of $t$ so that, for all $s \in U(t)$, $(g_s^{\beta(s)}, g_s^{\alpha(s)})$ is homotopy equivalent to $(g_{t}^{\beta(t)}, g_{t}^{\alpha(t)} )$.  Choose $\epsilon$ as in the hypotheses, and choose a neighborhood $U(t)$ of $t$ so that, for $s \in U(t)$:
\begin{enumerate}
\item $\| g_s - g_{t} \|_\infty < \epsilon$ by the continuity of the path $g_s$, 
\item $|\beta(t) - \beta(s)| < \epsilon$, and $ |\alpha(t) - \alpha(s)| < \epsilon.$
\end{enumerate}
By (1), we have the inclusions
$$
(
g_{t}^{\beta(t) - \epsilon}, g_{t}^{\alpha(t) - \epsilon}
) \subset
(
g_{s}^{\beta(t)}, g_{s}^{\alpha(t)}
) \subset
(
g_{t}^{\beta(t) + \epsilon}, g_{t}^{\alpha(t) + \epsilon}
).
$$
Applying Lemma~\ref{lem:retract} with the vector field $X_s$, we see that $( g_{t}^{\beta(t) - \epsilon}, g_{t}^{\alpha(t) - \epsilon} )$ is a deformation retract of $( g_{t}^{\beta(t)}, g_{t}^{\alpha(t) } ).$ Similarly, we see that $( g_{t}^{\beta(t) - \epsilon}, g_{t}^{\alpha(t) - \epsilon} )$ is a deformation retract of $( g_{t}^{\beta(t) + \epsilon}, g_{t}^{\alpha(t) + \epsilon} )$, and thus also of $( g_{s}^{\beta(t)}, g_{s}^{\alpha(t)} )$. On the other hand, by (2), Equation~\ref{eqn:cnc}, and Lemma~\ref{lem:retract}, we see that
$(
g_{s}^{\beta(t)}, g_{s}^{\alpha(t)}
)$
is homotopy equivalent to 
$(
g_{s}^{\beta(s)}, g_{s}^{\alpha(s)}
).$  Thus $(
g_{s}^{\beta(s)}, g_{s}^{\alpha(s)}
)$ is homotopy equivalent to $(
g_{t}^{\beta(t)}, g_{t}^{\alpha(t) }
)$ for all $s \in U(t)$, as desired.
\end{proof}

\section{Generating Family Cohomology Groups for Legendrian Submanifolds}
\label{sec:gh-legendrian}

In this section, we use sublevel sets of a ``difference function'' associated to a generating family to define the generating family (co)homology invariants of Legendrian submanifolds; see also \cite{f-r, lisa-jill, lisa:links}.

\subsection{Basic Definitions and Properties}

Suppose that $f: M \times \rr^N \to \rr$ is a generating family for a Legendrian $\leg \subset J^1M$.  The \dfn{ difference function}, $\delta: M \times \rr^N \times \rr^N \to \rr$, is defined to be: \begin{equation} \delta(x,\e,\te) = f(x,\te) - f(x,\e).  \end{equation}
 
The reason to work with the difference function is that its critical points
capture information about the \dfn{Reeb chords} of $\leg$ (with respect to the standard contact form), which in this context are vertical 
 segments $\gamma: [a,b] \to J^1M$ whose endpoints lie on $\leg$. Note that Reeb chords are in one-to-one correspondence with double points of the projection of $\leg$ to an immersed Lagrangian submanifold of $T^*M$.  Let $\ell(\gamma) > 0$ be the
 \dfn{ length of the Reeb chord} $\gamma$, and let $\lmax$ (resp.\ $\lmin$) denote the maximum (resp.\ minimum) length of all Reeb chords of $\leg$.

\begin{prop}[\cite{f-r, josh-lisa:cap}] 
  \label{prop:leg-crit-point} 
  The critical points of the difference function $\delta$ are of two types:
  \begin{enumerate}
  \item For each Reeb chord $\gamma$ of $\leg$, there are two
    critical points $(x,\e,\te)$ and $(x,\te,\e)$ of $\delta$ with
    nonzero critical values $\pm \ell(\gamma)$.
  \item The set
    $$\left\{ (x,\e,\e) \,:\, (x,\e) \in \Sigma_f\right\}$$
    is a critical submanifold of $\delta$ with critical value $0$.
  \end{enumerate}
  For generic $f$, these critical points and submanifolds are
  non-degenerate, and the critical submanifold has index $N$.
\end{prop}

We may calculate the Morse index of the non-degenerate critical points using the Conley-Zehnder index of the associated Reeb chord, as defined in \cite{ees:high-d-geometry}.  Given a Reeb chord $\gamma$, let $c_\gamma$ be a ``capping path'' in $\leg$ from the top of the Reeb chord to the bottom.  The Lagrangian projections of the tangents to $\leg$ along the capping path induce a path of Lagrangian subspaces $C_\gamma(t)$. We create a path from $C_\gamma(1)$ to $C_\gamma(0)$ as in \cite{ees:high-d-geometry}: choose a complex structure $I$ on $\rr^{2n}$ such that $I C_\gamma(1) = C_\gamma(0)$ and let $\lambda_\gamma(t) = e^{tI}C_\gamma(1)$.  The loop $C_\gamma \ast \lambda_\gamma$ will be denoted by $\bar{C}_\gamma$, and the Conley-Zehnder index $CZ(\gamma)$ is defined to be the Maslov index $\mu(\bar{C}_\gamma)$.

\begin{prop} \label{prop:index}
  Given a non-degenerate critical point $(x_0,\e_0, \te_0)$ of $\delta$ and its corresponding Reeb chord $\gamma$, we have:
  $$\ind_{(x_0,\e_0,\te_0)} d^2 \delta = CZ(\gamma) + N.$$
\end{prop}

\begin{proof}
  On one hand, after a fiber-preserving diffeomorphism, we may assume that in neighborhoods of $(x_0,\e_0)$ and $(x_0,\te_0)$, the generating family $f$ has the form
\begin{align*}
  f(x,\e) &= a(x) + b(\e) & \text{and}\\
  f(x,\te) &= \tilde{a}(x) + \tilde{b}(\te).
\end{align*}
It follows that: 
\begin{equation} \label{eqn:gf-ind}
  \ind_{(x_0, \e_0, \te_0)} d^2 \delta = \ind_{x_0} (d^2\tilde{a} - d^2 a) + \ind_{\te_0} d^2 \tilde{b} - \ind_{\e_0} d^2 b + N.
\end{equation}

On the other hand, we may use the Conley-Zehnder index of the path $\bar{C}_\gamma$ relative to the vertical Lagrangian $V = \{0\} \times \rr^n$, as defined in \cite{robbin-salamon:maslov}, to compute $CZ(\gamma)$.  We compute using the definitions above and \cite[Lemma 3.4]{ees:high-d-geometry}:
\begin{equation} \label{eqn:cz-ind}
  \begin{split}
    CZ(\gamma) = \mu(\bar{C}_\gamma) &= \mu(C_\gamma,V) + \mu(\lambda_\gamma,V) \\
    &=  \mu(C_\gamma,V) + \ind_{x_0} (d^2\tilde{a} - d^2 a).
  \end{split}
\end{equation}
To compute $\mu(C_\gamma, V)$ in terms of the generating family $f$, we let $H = \rr^n \times \{0\}$ and appeal to \cite[Theorem B.5]{theret:camel} (after an overall sign correction):
\begin{equation} \label{eqn:theret}
  \begin{split}
    \mu(C_\gamma, H) &= \ind_{(x_0,\te_0)}d^2 f - \ind_{(x_0,\e_0)} d^2 f \\
    &= \ind_{x_0} d^2 \tilde{a} - \ind_{x_0} d^2 a+ \ind_{\te_0} d^2 \tilde{b} - \ind_{\e_0} d^2 b.
  \end{split}
\end{equation}

It remains to understand the difference $\mu(C_\gamma,V) - \mu(C_\gamma,H)$. By \cite[Theorem 3.5]{robbin-salamon:maslov}, this difference is independent of the path $C_\gamma$ (with fixed endpoints) and is, in fact, equal to the difference
\begin{equation} \label{eqn:mas-diff}
  \mu(L,C_\gamma(0)) - \mu(L,C_\gamma(1)),
\end{equation}
where $L$ is a clockwise rotation from $H$ to $V$.  By a similar argument to that in \cite[Lemma 3.4]{ees:high-d-geometry}, the difference in (\ref{eqn:mas-diff}) is equal to $\ind_{x_0} d^2 \tilde{a} - \ind_{x_0} d^2 a$.  

Thus, combining Equations (\ref{eqn:gf-ind}), (\ref{eqn:cz-ind}), (\ref{eqn:theret}), and the correction between $V$ and $H$ calculated above, we obtain the desired index computation.
\end{proof}

Given  the geometric importance of the  critical points of $\delta$ and the 
philosophy of Morse theory, it is natural to study sublevel sets of $\delta$.  Choose 
$\epsilon$ and $ \infin $ so that   
\begin{equation} \label{ineq:epsilon-infin}
  0 < \epsilon < \lmin \leq \lmax < \infin.
\end{equation}

\begin{defn} 
  \label{defn:gh}  
  The \dfn{total (resp.\ relative) generating family cohomology} $\gh{*}{f}$ (resp.\ 
  $\rgh{*}{f}$)  of the  generating family $f$ 
  is defined to be:
  $$\gh{k}{f} = H^{k+N+1}(\delta^\infin, \delta^{-\epsilon}) \quad \text{and} \quad
  \rgh{k}{f} = H^{k+N+1}(\delta^\infin, \delta^{\epsilon}).$$
\end{defn}
  
\begin{rem}
\begin{enumerate}
\item There are also analogous definitions of the \dfn{total generating family homology}, $\widetilde{GH}_k(f)$, and \dfn{relative generating family homology}, $GH_k(f)$, using the same degree shift as above.

\item Caveat lector: the generating family homology in \cite{f-r} (and hence all comparisons to Legendrian contact homology) coincides with the \emph{relative} generating family homology in this paper.

\item We may think of the relative and total generating family cohomology as the total generating family cohomology taken relative to an expanded set.  This, along with the statement of Theorem~\ref{thm:filling-iso}, explains our naming convention.  \end{enumerate} \end{rem}

There is a simple relationship between the total and relative generating family cohomologies:

\begin{prop} \label{prop:full-rel-dual} Let $\Lambda^n$ be an  orientable, Legendrian submanifold of $J^1M$ with linear-at-infinity generating family $f$.   
  There is a long exact sequence:
  $$\dots \to H^k(\Lambda) \to GH^{k}(f) \to \widetilde {GH}^{k}(f) \to H^{k+1}(\Lambda) \to \dots.$$
  If the groups are calculated with $\mathbb Z_2$ coefficients, the result holds without the orientability condition on the Legendrian.
\end{prop}

\begin{proof} Fix $\epsilon$ and $\infin$ satisfying (\ref{ineq:epsilon-infin}).
From the triple $(\delta^\infin, \delta^\epsilon, \delta^{-\epsilon})$, we obtain the long exact sequence:
$$\dots \to H^{k+N}\left( \delta^\epsilon, \delta^{-\epsilon} \right) \to
H^{k+N+1}\left( \delta^\infin, \delta^{\epsilon} \right)  \to
H^{k+N+1}\left( \delta^\infin, \delta^{-\epsilon} \right)  \to \dots.
$$
When $\Lambda$ is orientable or when $\mathbb Z_2$ coefficients are used, standard constructions in Morse-Bott theory and the Thom isomorphism imply that 
$H^{k+N}\left( \delta^{\epsilon}, \delta^{-\epsilon} \right) \simeq H^k(\Lambda)$. The proposition now follows from the definitions of total and relative generating family cohomology.
\end{proof}

The generating family (co)homology descends to equivalence classes of generating families:

\begin{lem}  \label{lem:cohom-equiv}  If $f_0 \sim f_1$, then $\rgh{*}{f_0} \simeq \rgh{*}{f_1}$ 
and $\gh{*}{f_0} \simeq \gh{*}{f_0}$.
\end{lem}

\begin{proof} It is easy to verify that if $f_0 \sim f_1$, then their associated difference functions $\delta_0$ and $\delta_1$ will also be equivalent.  Moreover, if $f_0$ and $f_1$ differ by a dimension $K$ stabilization, then $\delta_0$ and $\delta_1$ will differ by an 
\emph{index} $K$ stabilization.  The lemma now follows immediately from Lemma~\ref{lem:equiv-he}.
\end{proof}

This lemma partially justifies the shift in index in the definition of generating family (co)homology; choosing to shift by $N+1$ rather than $N$ produces an isomorphism with linearized contact homology \cite{f-r}.

We next show that, under some ``tameness'' conditions on the generating family, the generating family (co)homology does not depend on the choices of  $\epsilon$ and $\infin$.  Since $f$ is defined on the non-compact space $M \times \rr^N$, its behavior outside a compact set must be sufficiently well-behaved in order to apply the Morse-theoretic lemmas enumerated in the previous section.  An important class of such generating families satisfies the following condition:

\begin{defn} \label{defn:li} A function $g: M \times \rr^N \to \rr$ is \dfn{linear-at-infinity} if $g$ can be written as
$$g(x, \e) = g^c(x, \e) + A(\e),$$
where $g^c$ has compact support and $A$ is a non-zero linear function. 
\end{defn}

This convention is particularly convenient for producing compact Legendrians when $M= \rr^n$, as seen in \cite{f-r, lisa-jill}.  There are two problems with the definition of linear-at-infinity:  first, it is not preserved under stabilization.  Second, it is easy to check that if $f$ is linear-at-infinity, then the associated difference function $\delta$ is no longer linear-at-infinity.  All is not lost, however, as the following lemmas show:

\begin{lem} \label{lem:lq-l} If $f$ is the stabilization of a
  linear-at-infinity generating family, then $f$ is equivalent to a
  linear-at-infinity generating family.
\end{lem}

\begin{proof} Let $f: M \times \rr^K \times \rr^{N-K} \to \rr$ be the
  stabilization of a linear at infinity generating family.  Thus we
  can assume that there is a non-zero linear function $A: \rr^K \to
  \rr$ and a non-degenerate quadratic function $Q: \rr^{N-K} \to \rr$
  so that outside a compact set of $M \times \rr^K \times \rr^{N-K}$,
  $f(x, \ell, \eta) = A(\ell) + Q(\eta)$.  Furthermore, after applying
  a fiber-preserving diffeomorphism of $M \times \rr^K \times
  \rr^{N-K}$ (the product of a linear transformation of $\rr^K$ and
  the identity transformation of $\rr^{N-K}$), we can assume that
  outside a compact set, $f(x, \ell, \eta) = A_1(\ell) + Q(\eta)$ with
  $A_1(\ell) = A_1(\ell_1, \dots, \ell_K) = \ell_1$.  To complete the
  proof, it suffices to show that the linear-quadratic function
  $g(\ell, \eta) = A_1(\ell) + Q(\eta)$ is equivalent to the linear
  function $A_1$.  Namely, we will construct a diffeomorphism $\Phi$
  of $\rr^K \times \rr^{N-K}$ so that $A_1 \circ \Phi ( \ell, \eta) =
  g(\ell, \eta)$.

  Notice that there are no critical points of $g$, and with respect to
  the standard metric on $\rr^K \times \rr^{N-K}$, each gradient
  trajectory of $g$ will intersect the hyperplane $\{ \ell_1 = 0 \}$
  transversally at precisely one point.  Suppose that $(\ell_1,
  \ell_2, \dots, \ell_K, \eta)$ and $(0, \ell_2, \dots, \ell_K,
  \eta')$ are on the same gradient trajectory of $g$, and $g(\ell_1,
  \ell_2, \dots, \ell_K, \eta) =t$.  Consider the diffeomorphism
  $\Phi$ of $\rr^K \times \rr^{N-K}$ given by $\Phi(\ell_1, \ell_2,
  \dots, \ell_K, \eta) = (t, \ell_2, \dots, \ell_K, \eta')$; this
  diffeomorphism is constructed by appropriate time flows along the
  gradient trajectories of $g$ and $A_1$.  By construction, $A_1 \circ
  \Phi ( \ell, \eta) = g(\ell, \eta)$, as desired.
\end{proof}

\begin{lem}[\cite{f-r}] \label{lem:lin-diff} If $f$ is linear-at-infinity, then the associated difference function $\delta$ is equivalent to a linear-at-infinity function.  \end{lem}

\begin{cor} \label{lem:e-l-indep} If $f: M \times \rr^N \to \rr$ is a linear-at-infinity generating family for the Legendrian $\leg$, then the isomorphism classes of 
$\rgh{*}{f}$ and $\gh{*}{f}$ are independent of the choice of $\epsilon$ and $\infin$.  \end{cor}

\begin{proof} 
If $f$ is linear-at-infinity, then by Lemma~\ref{lem:lin-diff}, we may assume that the associated difference function is linear-at-infinity as well.  Independence from the choice of $\epsilon$ and $\infin$ now follows from Proposition~\ref{prop:leg-crit-point} and Lemma~\ref{lem:retract}.  
\end{proof}

\begin{rem} \label{rem:e-l-indep}
  In fact, we can say something stronger:  not only are $H^*(\delta^{\infin_0}, \delta^{\epsilon_0})$ and $H^*(\delta^{\infin_1}, \delta^{\epsilon_1})$ isomorphic for different choices of $\infin_i$ and $\epsilon_i$, but because the isomorphisms come from following the negative gradient flow of $\delta$, we can even identify the underlying chain complexes.
\end{rem}

The results of Proposition~\ref{prop:leg-crit-point} led us to consider the levels $\pm \epsilon$ and $\infin$ in the definition of generating family cohomology.  Other combinations of levels are also natural to examine, but these do not lead to new invariants.

\begin{lem}  \label{lem:total-vanish} For sufficiently large $\infin$, the pair $( \delta^\infin, \delta^{-\infin})$ is acyclic.
\end{lem}

\begin{proof} By Lemma~\ref{lem:lin-diff}, we may assume that $\delta$ is linear-at-infinity, i.e. that we may write $\delta(x,\eta, \tilde{\eta}) = \delta^c(x,\eta, \tilde{\eta}) + A(\eta, \tilde{\eta})$ for a compactly supported function $\delta^c$ and non-zero linear function $A$.  Define a $1$-parameter family of functions $\delta_s$ by $\delta_s(x, \eta, \tilde{\eta}) = s \delta^c(x,\eta, \tilde{\eta}) + A(\eta, \tilde{\eta})$.  Denoting the support of $\delta^c$ by $U$, take $\infin$ greater than $\|\delta|_U\|_\infty$.  Applying Lemma~\ref{lem:crit-non-crossing} to $\delta_s$, the constant paths $\alpha(s) = -\infin$ and $\beta(s) = \infin$, and the gradient field $X = \nabla \delta$ for some choice of metric on $M \times \rr^{2N}$, we see that the pair $(\delta^\infin, \delta^{-\infin})$ is homotopy equivalent to $(A^\infin, A^{-\infin})$, which is obviously acyclic.  \end{proof}

\begin{cor}  \label{cor:alt-gf-hom} If $f: M \times \rr^N \to \rr$ is a linear-at-infinity generating family for the Legendrian submanifold $\leg \subset J^1(M)$, then for $\epsilon, \infin$ satisfying Inequalities~(\ref{ineq:epsilon-infin}), we have:
$$
\begin{aligned} 
H_{k + N } \left( \delta^\epsilon, \delta^{-\infin} \right) &\simeq GH_{k} (f), \text{ and } \\
H_{k + N } \left( \delta^{-\epsilon}, \delta^{-\infin} \right) &\simeq \widetilde{GH}_k(f),
\end{aligned}
$$
for all degrees $k$.
\end{cor}

\begin{proof} The corollary follows from Lemma~\ref{lem:total-vanish} and the long exact sequences of the triples $(\delta^\infin, \delta^{\epsilon}, \delta^{-\infin})$ and $(\delta^\infin, \delta^{-\epsilon}, \delta^{-\infin})$.
\end{proof}

\subsection{Invariance}

Given a Legendrian submanifold $\leg \subset J^1M$, let:
$$\lingf(\leg) = \{ f:   f \text{ is a linear-at-infinity generating family for } \leg \}.$$
 On the level of equivalence, we will be interested in equivalence classes of generating families that contain linear-at-infinity representatives.  

When $\leg$ is a Legendrian unknot in the standard contact $\rr^3$ with maximal Thurston-Bennequin invariant, all elements of $\lingf(\leg)$ are equivalent; see \cite{lisa-jill}.\footnote{In \cite{lisa-jill}, the focus was on generating families that are linear-quadratic-at-infinity.  Lemma~\ref{lem:lq-l}, however, can be used to show that linear-quadratic-at-infinity functions are equivalent to linear-at-infinity ones.}  In general, the set $\lingf(\leg)$ is not well understood, though see \cite{chv-pushkar, f-r,henry:mcs} for some recent progress.

To form an invariant of a Legendrian submanifold $\leg$ with a generating family, it is important to know that the existence of a linear-at-infinity generating family persists under Legendrian isotopy.  A proof of the following proposition can be given using
Chekanov's ``composition formula'' \cite{chv:quasi-fns}; see, for example, \cite{lisa-jill} and Section~\ref{sec:tqft}.
 
\begin{prop} [Persistence of Legendrian Generating
  Families] \label{prop:leg-persist} Suppose $M$ is compact.  For $t \in [0,1]$, let $\leg_t \subset J^1M$ be an isotopy of Legendrian submanifolds.  If $\leg_0$ has a linear-at-infinity generating family $f$, then there exists a smooth path of generating families $f_t: M \times \rr^N \to \rr$ for $\leg_t $ so that $f_0$ is a stabilization of $f$ and $f_t = f_0$ outside a compact set.
\end{prop}
  
\begin{rem} We will often be considering generating families for compact Legendrians in $J^1(\rr^n)$.  The above persistence will still apply since these Legendrians can be thought of as living in $J^1S^n$, and the linear-at-infinity condition allows the generating families to be defined on $S^n \times \rr^N$.  \end{rem}

\begin{cor} \label{cor:end-naturality} If $\leg \subset J^1M$ is a Legendrian submanifold and $\kappa_t$ is a contact isotopy of $J^1M$, then for every $f \in \lingf(K)$ there exists $f_t \in \lingf(\kappa_t(K))$ and an isomorphism $\kappa_t^\#: GH^*([f_t]) \to GH^*([f])$.
\end{cor}

The isomorphisms $(\kappa^s_\pm)^\#$ are constructed by applying the Critical Non-Crossing Lemma~\ref{lem:crit-non-crossing} to the difference functions $\delta^s_\pm$.

In general, since it many not be the case that all elements in $\lingf(\leg)$ are equivalent, the generating family homology of a linear-at-infinity generating family $f$, is not itself an invariant of the generated Legendrian $\leg$. By Lemma~\ref{lem:crit-non-crossing}, however, we do have:

\begin{prop} [\cite{lisa-jill, lisa:links}] For a compact Legendrian submanifold $\leg \subset J^1M$, the set of all generating family cohomology groups $$\mathcal{GH}^k(\leg) = \{ GH^k([f])\,:\, f \in \lingf (\leg)\},$$ is invariant under Legendrian isotopy.  \end{prop}

\subsection{Additivity}
\label{ssec:additivity}

Generating family cohomology not only produces an invariant of Legendrian submanifolds, but also behaves well under disjoint union; see \cite{arnold-paper} for a similar phenomenon.  The technical conditions involve the ``support'' of a linear-at-infinity generating function. We say that two linear-at-infinity generating families $f_1, f_2$ have disjoint supports if, possibly after stabilizing to match the domains, the functions $f_1^c, f_2^c : M \times \rr^N \to \rr$ have disjoint supports (see Definition~\ref{defn:li}). If two linear-at-infinity generating families have disjoint supports, we may assume that, up to equivalence, they agree with the same linear function $A$ outside a compact set.

\begin{defn} The \dfn{sum of two linear-at-infinity generating families with disjoint supports} is the linear-at-infinity function:
$$f_1 + f_2 = \begin{cases}
f_1, &\text{ on the support of $f_1$}\\
f_2, &\text{ on the support of $f_2$}\\
A, &\text{ otherwise}.
\end{cases}
$$
\end{defn}

We may extend the ideas of disjoint supports and sums to pairs of equivalence classes if there are (linear-at-infinity) representatives of each equivalence class with disjoint supports.  Using these notions, we may specify the behavior of generating family cohomology for a disjoint union of Legendrians:

\begin{prop} Suppose $\leg_1, \leg_2$ are Legendrian submanifolds of $J^1M$ so that $\pi_M(\leg_1) \cap \pi_M(\leg_2) = \emptyset$, where $\pi_M: J^1M \to M$ is the projection. For every $f_i \in \lingf(\leg_i)$, $i=1,2$, $[f_1]$ and $[f_2]$ have disjoint supports and $f_1 + f_2 \in \lingf (\leg_1 \cup \leg_2)$ satisfies
$$ GH^k([f_1 + f_2]) \simeq GH^k([f_1]) \oplus GH^k([f_2]).$$
\end{prop}

\begin{proof} By the hypothesis on the projections of $\leg_1$ and $\leg_2$, 
$f_1$ and $f_2$ can be made to have disjoint supports.  The result now follows from choosing linear-at-infinity representatives with disjoint support and 
applying a Mayer-Vietoris argument; see \cite{arnold-paper} for a similar argument.  
\end{proof}

\begin{rem}
  The usual axiom for a TQFT is \emph{multiplicativity} rather than \emph{additivity}; that is, the vector spaces associated to the components of disjoint union should be combined using a tensor product rather than a direct sum.  The fact that we obtain a direct sum is not completely surprising, however, if we think of the generating family homology as the linear term in a differential graded algebra, a structure that is well-known in the pseudo-holomorphic world of linearized contact homology; see \cite{products, egh, henry-rutherford}. What we are detecting is that the linear part of the tensor product of two tensor algebras is a direct sum of the respective linear parts.
\end{rem}

\section{Generating Families for Lagrangian Cobordisms}
\label{sec:cobord-gf}

We now shift our attention from individual Legendrian submanifolds to Lagrangian cobordisms between Legendrian submanifolds.  We will extend the technique of generating families to this new setting, eventually using them to define a TQFT-like structure on the generating family cohomology.

\subsection{Lagrangian Cobordisms and Compatible Generating Families}

A Legendrian submanifold $\leg \subset J^1M$ gives rise to a Lagrangian
cylinder $Z_\leg = \rr \times \leg$ in the symplectization $\rr\times J^1M$. 

\begin{defn} \label{defn:symplect-cobord} A \dfn{Lagrangian cobordism of $\rr \times J^1M$  between two Legendrian submanifolds  $\leg_-, \leg_+ \subset J^1M$},   
 is an embedded Lagrangian submanifold $\sclag  \subset \rr \times J^1M$ so that, for some $s_- < s_+$,
$\sclag$ agrees with the cylinder $Z_{\leg_-}$ for $s \leq s_-$ and $\overline L$ agrees with the cylinder $Z_{\leg_+}$ for $s \geq s_+$; such a cobordism will
be denoted by $\leg_{-} \prec_{\sclag} \leg_{+}$,
\end{defn}
  
To study Lagrangian cobordisms using generating families, we identify $\rr \times J^1M$ with  $T^*(\rr_+ \times M)$ by the symplectomorphism
\begin{equation} \label{eqn:id}
  \begin{split}
    \theta: \rr \times J^1M &\to T^*(\rr_+ \times M) \\
    (s,x,y,z) &\mapsto (e^s, x, z, e^sy).
  \end{split}
\end{equation}
 
Given a Lagrangian cobordism of $\rr \times J^1M$, we refer to its image $\clag = \theta(\sclag)$ as a \dfn{Lagrangian cobordism in $T^*(\rp \times M)$}.  We relabel the values $e^{s_\pm}$ by $t_\pm$.

For a Lagrangian cobordism $\leg_{-} \prec_{\sclag} \leg_{+}$, 
we will be interested in the situation where $\theta(\sclag) = \clag \subset T^*(\rp \times M)$ and $\leg_\pm \subset J^1M$ have
``compatible'' generating families.

\begin{defn} \label{defn:compatible}   Let $f_\pm : M \times \rr^N \to \rr$ and $F:  (\rp \times M) \times \rr^N \to \rr$ be functions.  
The triple of functions $(F, f_-, f_+)$ is \dfn{compatible} 
if for some $t_- < t_+$, we have
$$F(t, x, \eta)  = 
\begin{cases}
  t  f_- (x, \eta), &t \leq t_- \\
  t f_+ (x, \eta), &t \geq t_+.  \end{cases}$$ 
  Moreover,  a \dfn{gf-compatible Lagrangian cobordism} consists of a 
  Lagrangian cobordism $\leg_- \prec_{\sclag} \leg_+$ together with
  compatible triple of generating families $(F, f_{-}, f_{+})$ for, respectively,
  $\theta(\sclag) = \clag \subset  T^*(\rp \times M)$, $\leg_{-}, \leg_{+} \subset J^1M$.  A gf-compatible
  Lagrangian cobordism will be denoted by:
 $$ (\leg_-, f_-) \prec_{(\clag, F)} (\leg_+, f_+).$$
\end{defn}

The notion of compatibility descends to the level of equivalence classes of functions, so long as we require the fiber-preserving diffeomorphisms to be independent of $t$ outside of $[t_-,t_+]$.

As described in Section~\ref{sec:gh-legendrian}, it is useful to impose some ``nice'' behavior on the generating families $f_\pm$ outside a compact set.  The behavior of $F$ outside of $[t_-, t_+]$ prevents $F$ from being linear outside a compact set.  Instead, we will impose the following condition:

\begin{defn} \label{defn:sli}
  A function $F: (\rp \times M) \times \rr^N \to \rr$ is \dfn{slicewise-linear-at-infinity}  if    
   for each $t \in \rp$, there exists a non-zero linear function $A_t : \rr^N \to \rr$ so that $F(t, x, \eta) = A_t(\eta)$ outside a compact set of $M \times \rr^N$.
   A triple of compatible functions $(F, f_-, f_+)$ is \dfn{tame} if $F$ is slicewise-linear-at-infinity and $f_\pm$ are linear-at-infinity.  \end{defn}

\subsection{Wrapped Generating Family Cohomology}
\label{ssec:wgh}

The Lagrangian Floer cohomology groups of a pair of {\it closed} Lagrangian submanifolds $L^0$ and $L^1$ in a symplectic manifold $(W,\omega)$ has a cochain complex generated by elements of $L^0 \cap L^1$, \cite{floer:lagr}.  To generalize to Lagrangians $L^i \subset \rr_+ \times J^1M$, $i=0,1$, that are compact and have {\it boundaries} consisting of Legendrian components $\leg^i_\pm \subset \{ s_\pm\} \times J^1M$, we use the notion of wrapped Floer homology; see \cite{abb-schw:cotangent, as:wrapped-floer, fss:wrapped-floer}.  Wrapped Floer homology uses a chain complex generated by the intersections between compact pieces of two Lagrangian submanifolds {\it and} the Reeb chords between the Legendrian ends $\leg^i_\pm$.  When the compact Lagrangians are extended to Lagrangians $\overline L^i$ with cylindrical ends, then the Reeb chords of interest correspond to intersections between $\overline L^0$ and the image of $\overline L^1$ under an appropriately defined Hamiltonian diffeomorphism. 

In the following, we will define wrapped cohomology using the theory of generating families for Lagrangians $\lag^0 = \lag^1 \subset T^*(\rp \times M)$; the definitions may easily be extended to the case where $\lag^0 \neq \lag^1$.

We begin by specifying the Hamiltonian functions that will be used to convert Reeb chords of Legendrian submanifolds at the boundary to intersections of Lagrangians.  

\begin{defn} \label{defn:shear-H} Given $\leg_{-} \prec_{\sclag}  \leg_{+}$,
for $\clag = \theta(\sclag)$, the set $\mathcal{H}\left(\clag \right)$ of \dfn{Hamiltonian shearing functions} consists of decreasing, smooth functions $H: \rr_+ \to \rr$ that depend on a choice of $t_\pm$ from Definition~\ref{defn:symplect-cobord} and additional parameters $r_\pm$ and $u_\pm$, where $u_-< t_- \leq t_+ < u_+$.  The functions $H$ must satisfy $H^{''}(t) \geq 0$ on $(0, t_-]$ and $H^{''}(t) \leq 0$ on $[t_+, \infty)$ with
  $$H(t) = 
  \begin{cases}
     \frac{r_-}{2}( t - t_-)^2, & t \leq u_-\\
    0, & t \in [t_-, t_+]\\
    - \frac{r_+}{2}( t - t_+)^2, & t \geq u_+;
  \end{cases}$$
see Figure~\ref{fig:H-dH}.  The parameters must satisfy the following technical conditions:
\begin{enumerate}
\item $r_+$ is chosen to be sufficiently large so that $r_+ > \frac{\lmaxp}{t_+}$;
\item $r_-$ is chosen to be sufficiently large so that 
$$\frac{r_-t_-^2}{2} > \text{max} \left\{ 2t_- \lmaxm, t_- \lmaxp - \frac{\lmaxp^2}{2r_-}, 3t_+\lmaxp + \frac{\lmaxp^2}{2r_+} \right\};$$
\item $u_\pm$ are chosen sufficiently close to $t_\pm$ so that $ |u_\pm - t_\pm| < \frac{\lminpm}{2r_\pm}$.
\end{enumerate}
\end{defn}

As a consequence of these choices, we have the following inequalities:

\begin{equation} \label{eqn:ru} \frac{\lmaxpm}{t_\pm} < r_\pm < \frac{\lminpm}{2|u_\pm - t_\pm|}.  
\end{equation}

\begin{figure}
  \centerline{\includegraphics{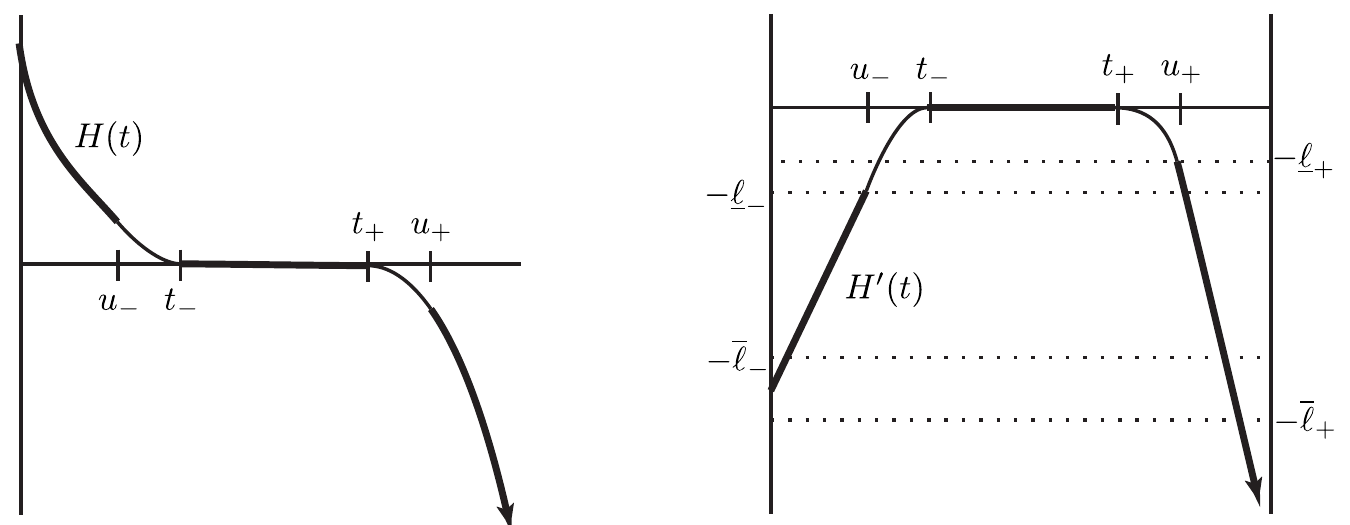}}
  \caption{A schematic picture of $H$ and $H'(t)$ for $H \in \mathcal H(\clag)$.}
  \label{fig:H-dH}
\end{figure}

For $H \in \mathcal{H}\left( \clag\right)$, the quadratic growth condition guarantees that the associated Hamiltonian vector field $X_H$ will be integrable.  If $\phi_H^1$ denotes the time-1 flow of this vector field and $F$ generates $\clag$, then it is easy to verify that $F(t,x,\e) + H(t)$ generates $\phi_H^1(\clag)$.   It is also straightforward to check that $\mathcal H(\clag)$ is path connected.

In parallel to the definition of the difference function $\delta$ in the previous section, a shearing function $H \in \mathcal{H}\left( \clag \right)$ may be used to define the \dfn{sheared difference function} $\Delta: \rp \times M \times \rr^{N} \times \rr^{N'} \to \rr$ is: \begin{equation} \label{eqn:pert-diff}
  \Delta(t,x,\e,\te) = F(t,x,\te) + H(t) - F(t,x,\e).
\end{equation}
Notice that when $t \leq t_-$, the sheared difference function satisfies the identity
$$\Delta(t,x,\e,\te) = t \delta_- (x, \e, \te) + H(t),$$
where $\delta_-$ is the difference function of $f_-$.  A similar statement holds when $t \geq t_+$.

In parallel to Proposition~\ref{prop:leg-crit-point},  the critical points of $\Delta$ detect information about the intersection points of
$\clag$ and $\phi^1_H({\clag})$:

\begin{prop} \label{prop:crit-pt-wgh}  Let $(\leg_-, f_-) \prec_{(\clag, F)} (\leg_+, f_+)$.
There is a one-to-one correspondence between intersection points in ${\clag} \cap \phi^1_H\left({\clag}\right)$ and critical points of $\Delta$.  Moreover, there is a one-to-one correspondence between Reeb chords $\gamma_\pm$ of $\leg_\pm$  and points in 
${\clag} \cap \phi^1_H\left({\clag}\right) \cap \{ t \in(0, u_-) \cup (u_+, \infty)\}$; the critical value of the point corresponding
to the Reeb chord $\gamma_\pm$ is:
$$ t_\pm \ell(\gamma_\pm) \pm \frac1{2r_\pm}  \left( \ell(\gamma_\pm) \right)^2 > 0.$$
All other critical points 
lie in the critical submanifold $$C=\left\{(t,x,\e,\e) \;:\; (t,x,\e) \in \Sigma_F \text{ with }
    t \in [t_-, t_+] \right\}.$$ The critical submanifold is diffeomorphic to $\lag = \clag \cap \{ t \in [t_-, t_+] \}$ and has value $0$; for generic $F$, the submanifold $C$ is non-degenerate of index $N$.  
\end{prop}	
 
\begin{proof}
  A straightforward calculation shows that critical points of $\Delta$ correspond to points in ${\clag} \cap \phi^1_H\left({\clag}\right)$.  By construction of $H$, we see that critical points of $\Delta$ with $t \in [t_-, t_+]$ correspond to self-intersection points of $\clag \cap \{ t \in [t_-,t_+]\}$.  Since ${\clag}$ is embedded, every critical point with $t \in [t_-,t_+]$ has critical value $0$.
  
For $t < t_- $, we have:
  \begin{align*}
    {\clag}&= \bigl\{  (t, x_0, z_0, ty_0) \,:\, (x_0, y_0 ,z_0) \in \leg_- \bigr\}, \\
    \phi_H^1\left({\clag}\right) &= \bigl\{ (t, x_1, z_1 + H'(t), ty_1) \,:\, (x_1, y_1, z_1) \in \leg_-
    \bigr\}.
  \end{align*}
  On this region, $H'(t) \leq 0$ and so the intersection points of these two Lagrangians occur when there is a Reeb chord $\gamma_-$ of $\leg_-$ with $ \ell(\gamma_-) = -H'(t)$.  The special form of $F$ and $H$ when $t < u_-$ then implies that critical points occur when $t = t_- - \frac{\ell(\gamma_-)}{r_-}$, and Inequality~(\ref{eqn:ru}) implies that $t < u_-$. The Inequalities~(\ref{eqn:ru}) also imply that such $t$ values are positive, and hence the critical points of $\Delta$ capture all of the Reeb chords of $\leg_-$.  The critical value is then a simple calculation, with its positivity following once again from Inequalities~(\ref{eqn:ru}).
  
  The arguments for $t > t_+$ are similar (and, in fact, slightly easier).
\end{proof}
 
The following lemma is essentially a $1$-parameter version Lemma~\ref{lem:lin-diff}.

\begin{lem} \label{lem:slin-diff} 
If $(\leg_-, f_-) \prec_{(\clag, F)} (\leg_+, f_+)$, 
 then for any $H \in \mathcal H\left( \clag \right)$, the associated compatible triple $(\Delta, \delta_-, \delta_+)$ is equivalent to a tame triple of functions.  \end{lem}

We are now ready to define generating family cohomology groups for Lagrangian cobordisms.  

\begin{defn} \label{defn:wgh} Let $(\leg_-, f_-) \prec_{(\clag, F)} (\leg_+, f_+)$.
For $H \in \mathcal H\left( \clag \right)$, choose $\Infin$ and  $\mu$ so that
\begin{equation} \label{ineq:wgh}
  \begin{split}
    \frac{r_-t_-^2}{2} &> \Infin > \max\left\{ t_\pm \lmaxp \pm \frac{\lmaxp^2}{2r_{\pm}}, 2t_- \lmaxm,  3t_+ \lmaxp + \frac{\lmaxp^2}{2r_+}  \right\},\\
    0 &<  \mu < \min\left\{ t_\pm \lminp \pm \frac{\lminp^2}{2r_{\pm}}, \frac{r_\pm}2|u_\pm^2 - t_\pm^2|, t_- \lminm, \ \frac{u_+ \lminp}{2},  \frac{r_+}2(u_+ - t_+)^2\right\}.
 \end{split}
\end{equation}
The \dfn{total (resp. relative) wrapped generating family cohomology of $F$}, $\wgh{k}{F}$ (resp. \rwgh{k}{F}), are defined to be:
 $$\wgh{k}{F} = H^{k+N}\left(\Delta^\Infin, \Delta^{-\mu}\right) \quad \text{and} \quad
  \rwgh{k}{F} = H^{k+N}\left(\Delta^\Infin, \Delta^{\mu}\right).$$ 
\end{defn} 

As shown in Proposition~\ref{prop:crit-pt-wgh}, all critical values of $\Delta$ lie in $[-\mu, \Infin]$, and all critical values of $\Delta$ arising from the Reeb chords of the ends lie in $[\mu, \Infin]$; the other restrictions on $\Infin$ and  $\mu$ will be useful later in this section and when examining the pairs $(\Delta^\Infin, \Delta^{\pm \mu})$ in Section~\ref{sec:map-cones} and beyond.

We have not included $H$, $\Infin$, or $\mu$ in the notation for the total and relative wrapped generating families cohomologies since, as we will show below, they are independent of these choices.  The Critical Non-Crossing Lemma~\ref{lem:crit-non-crossing} will play a key role in these proofs, so it will be necessary to work with the sheared difference function over the compact base $[v_-, v_+] \times M$ rather than $\rp \times M$
  We begin by introducing some convenient notation and two lemmas that will allow us to apply the Critical Non-Crossing Lemma~\ref{lem:crit-non-crossing}.
  
  For $J = [t_0, t_1] \subset \rp$, we use the shorthand
  \begin{gather*} \label{eqn:Delta-restrict}
    \Delta|_J = \Delta|_{\{(t,x,\e, \te): t \in J\}}, \\
 \Delta_J^a = \Delta^a \cap \{t \in J\} = \left\{(t,x,\e,\te)\;:\; t \in J, \Delta(t,x,\e,\te) \leq a\right\}.
\end{gather*}
Notice that if $J \subset [t_+, \infty)$, then we have:
\begin{equation} \label{eqn:end-level}
  \Delta_J^a = \left\{ (t,x,\e,\te) \;:\; t \in J,\  \delta_+(x,\e,\te) \leq \frac{1}{t}(a - H(t)) \right\};
\end{equation}
a similar fact holds for $J \subset (0,t_-]$.  The function $\frac{1}{t}(a-H(t))$ is sufficiently important that we assign to it the name $\lambda_a(t)$.

\begin{lem} \label{lem:cobord-cpt-base}
  For constants $\sigma < \tau < \frac{r_-t_-^2}{2}$,
  there exist $v_\pm \in \rp$ so that:
  \begin{equation} \label{eqn:cobord-cut-off} 
  H^*(\Delta^\tau, \Delta^\sigma) \simeq H^*\left(\Delta^\tau_{[v_-,v_+]}, \Delta^\sigma_{[v_-,v_+]}\right).
  \end{equation}
\end{lem}

\begin{proof} The points $v_\pm$ will be constructed in the region where $\Delta = t \delta_\pm + H(t)$.  First notice 
   that if $a < \frac{r_-t_-^2}{2}$, then the following limits hold for the quantity in Equation~(\ref{eqn:end-level}):
  \begin{equation} \label{eqn:level-limits}
    \lim_{t \to 0} \lambda_a(t) = -\infty \quad \text{and} \quad \lim_{t \to \infty} \lambda_a(t) = \infty.
  \end{equation}
  Let $\underline{c}_-$ denote the minimum critical value of $\delta_-$, and let $\overline{c}_+$ denote the maximum critical
  value of $\delta_+$.  
  Choose $v_- < t_-$ so that $\lambda_\tau(t)< \underline{c}_-$ for all $t \leq v_-$, and 
 choose $v_+> t_+$ so that $\lambda_\sigma(t) > \overline{c}_+$ for all $t \geq v_+$.
 Note that there are no critical values of $\delta_-$ in $[ \lambda_\sigma(t),\lambda_\tau(t)]$ for all $t \in (0,v_-]$, and similarly at the positive end.  Equation~\ref{eqn:cobord-cut-off} now follows from a Mayer-Vietoris argument: for a sufficiently small $\epsilon$, split the domain into $U = (v_- - \epsilon, v_++ \epsilon) \times M \times \rr^{2n}$ and $V = \bigl[ (0, v_-+\epsilon) \cup (v_+-\epsilon) \bigr] \times M \times \rr^{2N}$.  The pair $(\Delta^\tau_{(0, v_-+\epsilon)}, \Delta^\sigma_{(0, v_-+\epsilon)})$ is acyclic since we may follow the slicewise negative gradient flow of $\delta_-$ on each $\{t\} \times M \times \rr^{2N}$ to retract $\Delta^\tau_{(0, v_-+\epsilon)}$ down to $\Delta^\sigma_{(0, v_-+\epsilon)}$.  A similar argument applies at the positive end and on the overlaps, so the desired isomorphism follows.  \end{proof}

\begin{lem} \label{lem:cobord-cpt-deform}
  For the values $v_\pm$ in the previous lemma, if  $\tau' < \tau <  \frac{r_-t_-^2}{2}$
  and there are no critical values of $\Delta|_{[v_-, v_+]}$ in
  $[\tau', \tau]$, then
  $\Delta_{[v_-, v_+]}^{\tau'}$ is a deformation retract of
  $\Delta_{[v_-, v_+]}^\tau$.
\end{lem}

\begin{proof} 
  The claimed deformation retract will follow from Lemma~\ref{lem:retract} if we can construct an integrable, gradient-like vector field $X$ on $[v_-, v_+] \times M \times \rr^N$ for $\Delta|_{[v_-, v_+]}$ which is bounded away from $\mathbf 0$ on $\left(\Delta|_{[v_-, v_+]}\right)^{-1}([\tau', \tau])$.  

  Fix a metric on $[v_-, v_+] \times M \times \rr^N$ and let $X$ be the vector field
$$ X (t, x, \e, \te)  = \begin{cases}
\grad \Delta(t, x, \e, \te), &t \in [t_-, t_+]\\
\rho(t) \left( \delta_- + H'(t) \right) \pd{}{t} + t \grad \delta_-, & t \leq t_- \\
\rho(t) \left( \delta_+ + H'(t) \right) \pd{}{t} + t \grad \delta_+, & t \geq t_+,
\end{cases}$$
 where $\rho(t): [v_-, t_-] \cup [t_+, v_+] \to [0,1]$ is a smooth function with $\rho(t_\pm) = 1$
 and $\rho^{-1}\{ 0 \} = \{v_\pm\}$.  It is clear that $X$ is a gradient-like vector field for $\Delta|_{[v_-, v_+]}$ when
  $t \in [t_-, t_+]$.  
 When $t \leq t_-$,
 $$\langle X, \grad \Delta \rangle = 
 \rho(t) \left( \delta_- + H'(t) \right)^2 + t^2 \| \grad \delta_- \|^2 \geq 0.$$
 The quantity $\langle X, \grad \Delta \rangle$ vanishes only  when either $(t, x, \e, \te)$ is a critical 
 point of $\Delta$
 or when 
 $t = v_-$, and $\grad \delta_-(x, \e, \te) = 0$.  Neither of these two cases can occur
 on $\left(\Delta|_{\{v_-\}}\right)^{-1}[\tau', \tau]$
  by
 construction of $v_-$.
 A similar argument works for $t \geq t_+$.

 By construction, $X$ is parallel to the boundary of $[v_-, v_+] \times M \times \rr^N$, and hence the tameness of $(\Delta, \delta_-, \delta_+)$ (see Lemma~\ref{lem:slin-diff}) implies that $X$ is integrable.
 Furthermore, if $\tau' < \tau <  \frac{r_-t_-^2}{2}$
   and there are no critical values of $\Delta_{[v_-, v_+]}$  in $[\tau', \tau]$,
   then   
    $X$ is bounded away from $\mathbf 0$ on
   $\left(\Delta|_{[v_-, v_+]}\right)^{-1}([\tau', \tau])$.  Thus, by Lemma~\ref{lem:retract}, $\Delta_{[v_-, v_+]}^{\tau'}$ is a deformation
   retract of  $\Delta_{[v_-, v_+]}^{\tau}$.
   \end{proof}

We will be particularly interested in Lemma~\ref{lem:cobord-cpt-base} where
$\tau = \Infin$ and $\sigma = \pm \mu$, where
$\Infin$ and $\mu$ satisfy Inequalities~(\ref{ineq:wgh}).
In this case, we want to choose  $v_- < t_-$ and $v_+ > t_+$ so that 
\begin{equation} \label{ineq:vpm}
\lambda_\Infin(t) < -\lmaxm \quad \forall t \leq v_-, \quad\text{ and } \quad  \lambda_{\pm \mu}(t) > \lmaxp 
\quad \forall t \geq v_+.
\end{equation} 

\begin{cor} \label{cor:wgh-cpct} Given $(\leg_-, f_-) \prec_{(\clag, F)} (\leg_+, f_+)$,
choose $H \in \mathcal H(\clag)$.  For $\Infin$ and $\mu$ satisfying Inequalities~(\ref{ineq:wgh}) and $v_\pm$ satisfying Inequalities~(\ref{ineq:vpm}), we have
$$
\rwgh{k}{F}
\simeq H^{k+N}\left( \Delta_{[v_-, v_+]}^\Infin, \Delta_{[v_-, v_+]}^{\mu} \right),$$
and
$$
\wgh{k}{F}
\simeq H^{k+N}\left( \Delta_{[v_-, v_+]}^\Infin, \Delta_{[v_-, v_+]}^{-\mu} \right).$$
\end{cor}

With these preliminary constructions established, we are ready to prove the independence of the wrapped generating family cohomology from the choices of $\Infin$, $\mu$, and $H$.

\begin{prop} \label{prop:wgh-independent} 
Given $(\leg_-, f_-) \prec_{(\clag, F)} (\leg_+, f_+)$,
the isomorphism types of $\wgh{k}{F}$ and $\rwgh{k}{F}$ do not depend on the choice of $H$, $\Infin$, or $\mu$.  \end{prop}

\begin{proof}  
Since $\mathcal{H}(\clag)$ is path connected, we may choose continuous paths $H_s$ in $\mathcal{H}(\clag)$, $\Infin_s$, and $\mu_s$ with $s \in [0,1]$ joining any two triples of choices of $H$, $\Infin$, and $\mu$ that all satisfy the appropriate inequalities.
Let $\Delta_s$ be the path of associated sheared difference functions.  Choose $v_\pm$ that satisfy the Inequalities~(\ref{ineq:vpm}) for all $s \in [0,1]$.
After applying a fiber-preserving diffeomorphism, we can assume that 
when $t \in [v_-, v_+]$, the sheared difference functions may be written as
\begin{equation} \label{eqn:nice-Delta}
\Delta_s(t, x, \e, \te) = \Delta_s^c(t, x, \e, \te) + A_t(\e, \te) + H_s(t),
\end{equation}
where $\Delta_s^c(t, x, \e, \te)$ is compactly supported and, for each $t$,  $A_t( \e, \te)$ is a non-zero linear function.  

We finish the proof by applying the Critical Non-Crossing Lemma~\ref{lem:crit-non-crossing}.  Since we can assume that $(\Delta_s, \delta_-, \delta_+)$ is tame and that $\Infin(s)$ and  $\pm \mu(s)$ are always regular values of $\Delta_s$, to apply Lemma~\ref{lem:crit-non-crossing} it suffices to construct an appropriate gradient-like vector field $X_s$ on $[v_-, v_+] \times M \times \rr^N$ for $\Delta_s|_{[v_-, v_+]}$; this vector field may be constructed exactly as in Lemma~\ref{lem:cobord-cpt-deform}.
\end{proof}   

Though it was simple enough to simultaneously prove that both the relative and total wrapped generating family cohomologies do not depend on the choices involved in their definitions, it only matters for the relative case, as we have:

\begin{prop} \label{prop:wgh-vanish}
  The total wrapped generating family cohomology vanishes.
\end{prop}

\begin{proof}
Using notation as in Equation~(\ref{eqn:nice-Delta}), consider 
$$\Delta_s(t, x, \e, \te) = (1-s) \Delta^c(t, x, \e, \te) + A_t(\e, \te) + H(t)$$
for $t$ in some compact interval $J$. Choose paths $\Infin_s$ and $\mu(s)$ so that $\Infin_0 = \Infin$, $\mu_0 = \mu$, and all critical values of $\Delta_s$ lie in $[-\mu_s, \Infin_s]$.  Notice that $\Delta_1(t, x, \e, \te) = A_t(\e, \te) + H(t)$, and hence has no critical values.  If we can show that there exists $J = [v_-, v_+]$ with $v_\pm$ satisfying Inequalities (\ref{ineq:vpm}), and an integrable, gradient-like vector field $X_s$ for $\Delta_s$ on $[v_-, v_+] \times M \times \rr^N$, then Corollary ~\ref{cor:wgh-cpct} and the Critical Non-Crossing Lemma~\ref{lem:crit-non-crossing} imply:
$$
\begin{aligned}
\wgh{k}{F} &\simeq 
H^{k+N}\left((\Delta_0)_{[v_-, v_+]}^{\Infin_0}, (\Delta_0)_{[v_-, v_+]}^{-\mu_0} \right) \\
&\simeq  H^{k+N}\left((\Delta_1)_{[v_-, v_+]}^{\Infin_1}, (\Delta_1)_{[v_-, v_+]}^{-\mu_1} \right) \\ &= 0.
\end{aligned}$$

To construct appropriate $v_\pm$, notice that 
when $t \leq t_-$ or $t \geq t_+$, we have:
$$\Delta(t, x, \e, \te) =  t\delta_\pm^c(x, \e, \te) + tD_\pm(\e, \te) + H(t).$$
Consider 
$$ (\delta_\pm)_s = (1-s)\delta_\pm^c(x, \e, \te) + D_\pm(\e, \te).$$
Let $\left( \overline{c}_+ \right)_s $ be greater than all critical values of $(\delta_+)_s$,
and let $\left( \underline{c}_-\right)_s $ be less than all critical values of $(\delta_-)_s$.
Then choose $v_\pm$ so that for all $s \in [0,1]$,
$$\lambda_{\Infin_s}(v_-) < (\underline{c}_-)_s  \quad \text{and} \quad
\lambda_{-\mu_s}(v_+) >   (\overline{c}_+)_s.$$
It follows that if $\Delta(v_\pm, x, \e, \te) \in [-\mu_s, \Infin_s]$, then $(x, \e, \te)$ is not a critical
point of $(\delta_\pm)_s$.  The construction of the integrable, gradient-like vector field $X_s$ for
$\Delta_s$ on $[v_-, v_+] \times M \times \rr^N$ is as in the proof of Lemma~\ref{lem:cobord-cpt-deform}.
\end{proof}

Although the total wrapped generating family cohomology vanishes, the relative version can be non-trivial.  In fact, we will show in Proposition~\ref{prop:vert-les} that $\rwgh{k+1}{F} \simeq H^{k}(\lag, \partial \lag_+)$.

\section{Relative Mapping Cones}
\label{sec:map-cones}

In Sections~\ref{sec:filling-iso} and \ref{sec:cobord-les}, we will show that the pair of spaces used to define the total and relative wrapped generating family cohomology can be viewed as objects akin to mapping cones.  In this section, we develop the theory of \emph{relative} mapping cones.

\subsection{Long Exact Sequences from Mapping Cones} 

The key idea behind the desired constructions is the use of a relative version of the well-known mapping cone construction.  Let $I$ denote the unit interval $[0,1]$. Recall that the cone of a space $X$, $C(X)$, is defined to be $X \times I/ X \times \{1\}$.  Given a map $f: X \to Y$, the mapping cone $C(f)$ is defined to be $C(X) \cup_f Y$, where $\cup_f$ indicates an identification of $(x, 0)$ with $f(x)$.
 
\begin{defn}\label{def:rel-cone}
Given a pair $(X,A)$, define the \dfn{relative cone} $C(X,A)$ to be the pair $(X \times I, A \times I \cup X \times \{1\})$.  For a map $g: (X,A) \to (Y,B)$, let the \dfn{relative mapping cone} $C(g)$ be the pair $C(X,A) \cup_g (Y, B)$. 

We may similarly define the  \dfn{relative suspension} $\Sigma(X,A)$ to be the pair $(X \times I, A \times I \cup X \times \{0,1\})$.
\end{defn}

The following lemma, whose proof is an easy exercise, shows that the (co)homology of relative suspensions behaves in the same way as it does for the non-relative case:

\begin{lem} \label{lem:rel-susp}
  $H^k(\Sigma(X,A)) \simeq H^{k-1}(X,A)$.
\end{lem}

It is well-known that the classical mapping cone on $f:X \to Y$ induces a long exact sequence:
$$\dots \to H^k(Y) \stackrel{f^*}{\to} H^k(X) \to H^{k+1}(C(f)) \to H^{k+1}(Y) \stackrel{f^*}{\to} \dots .$$  
A similar sequence exists for a relative mapping cone:

\begin{lem} \label{lem:cone-exact} Given a map $g: (X, A) \to (Y, B)$, there is a long exact sequence in cohomology:
  $$ \cdots \to H^k(Y, B) \stackrel{g^*}{\to} H^k(X, A) \to H^{k+1}(C(g)) \to H^{k+1}(Y, B) \stackrel{g^*}{\to} \cdots.$$  
\end{lem}

\begin{proof} The desired result follows by examining the long exact sequence of the triple $$\bigl(X \times I \cup_g Y, (A \times I \cup X \times \{1\}) \cup_g Y,
  (A \times I \cup X \times \{1\}) \cup_g B \bigr).$$ By excision, the cohomology groups of the last pair in the triple agree with those of $(Y,B)$, the coholomogy groups of the pair made from the first and last terms are the cohomology groups of $C(g)$, and the first pair is a suspension of $(X,A)$ and thus Lemma~\ref{lem:rel-susp} applies.  \end{proof}

\subsection{Morse-Theoretic Lemmas to Realize Relative Mapping Cones}
\label{sec:sublevel-sets}

In Sections~\ref{sec:filling-iso} and \ref{sec:cobord-les}, we prove that pairs of sublevel sets of $\Delta$ may be identified with relative mapping cones. The lemmas developed in this section will play a critical role in this identification.
 
We begin by setting notation. Let $J = [t_0,t_1]$ be a closed interval.  Given continuous functions $a,b: J  \to \rr$, we define the following subsets of $J \times X$:
\begin{equation} \label{eqn:A-B}
  \begin{aligned}
    A_t &= \{ t \} \times \delta^{a(t)}, \qquad A_J = \bigcup_{t \in
      J} A_t,  \\
    B_t &= \{ t \} \times \delta^{b(t)}, \qquad B_J = \bigcup_{t \in
      J} B_t.
  \end{aligned}
\end{equation}
We will be interested in finding  homotopy equivalences of the pair $(B_J, A_J)$ under different conditions on the functions $a$ and $b$.  The reason for considering such a setup is that the pairs $(\Delta^\tau, \Delta^\sigma)$ have this form at the ends of $\rr_+$ for the difference functions $\delta_\pm$ and the levels $b(t) = \lambda_\tau(t)$ and $a(t) = \lambda_\sigma(t)$.

First, we analyze the pair $(B_J, A_J)$ in terms of the sublevel sets on the right side of $J$.
\begin{lem} \label{lem:def-retr-right}
  Let $\delta: X \to \rr$ be a continuous function and let $a, b : J \to \rr$ be continuous
  functions satisfying:
\begin{enumerate}
\item $a(t) \leq b(t)$ for all $t \in J$, and
\item $a$ and $b$ are strictly increasing.
\end{enumerate}
Then $\left(B_J, A_J\right)$  deformation retracts to 
$\left(A_J \cup B_{t_1}, A_J\right)$.
\end{lem}

\begin{proof}
  We define a retraction $\rho_+: B_J \to A_J \cup B_{t_1}$ as follows:
  $$\rho_+(t,x) = \begin{cases}
    (t,x), & \delta(x) \leq a(t) \\
    (a^{-1}(\delta(x)),x), & a(t) \leq \delta(x) \leq a(t_1)\\
    (t_1,x), & \delta(x) \geq a(t_1).
  \end{cases}$$
  See the left side of Figure~\ref{fig:morse-lemmas}.  To see this map as the end map of 
  a deformation retraction, simply follow the flow of the horizontal vector field $\partial_t$ for ever shorter time intervals; this flow lies inside $B_J $ since $b(t)$ is increasing.
\end{proof}

\begin{cor}  \label{cor:left-map-cone} Under the hypotheses of Lemma~\ref{lem:def-retr-right}, 
$(B_J, A_J)$ 
 deformation retracts to  $(B_{t_1}, A_{t_1})$
\end{cor}

\begin{proof}  Consider the retraction $\sigma_+: A_J \cup B_{t_1} \to (B_{t_1}, A_{t_1})$
given by $\sigma_+(t, x) = (t_1, x)$.  Since $a$ is increasing,  $\sigma_+$ can be seen
as the end map of a deformation retraction.  Composing $\rho_+$ and $\sigma_+$ gives the desired deformation retraction.   
\end{proof}

\begin{figure}
  \centerline{\includegraphics[width=5.5in]{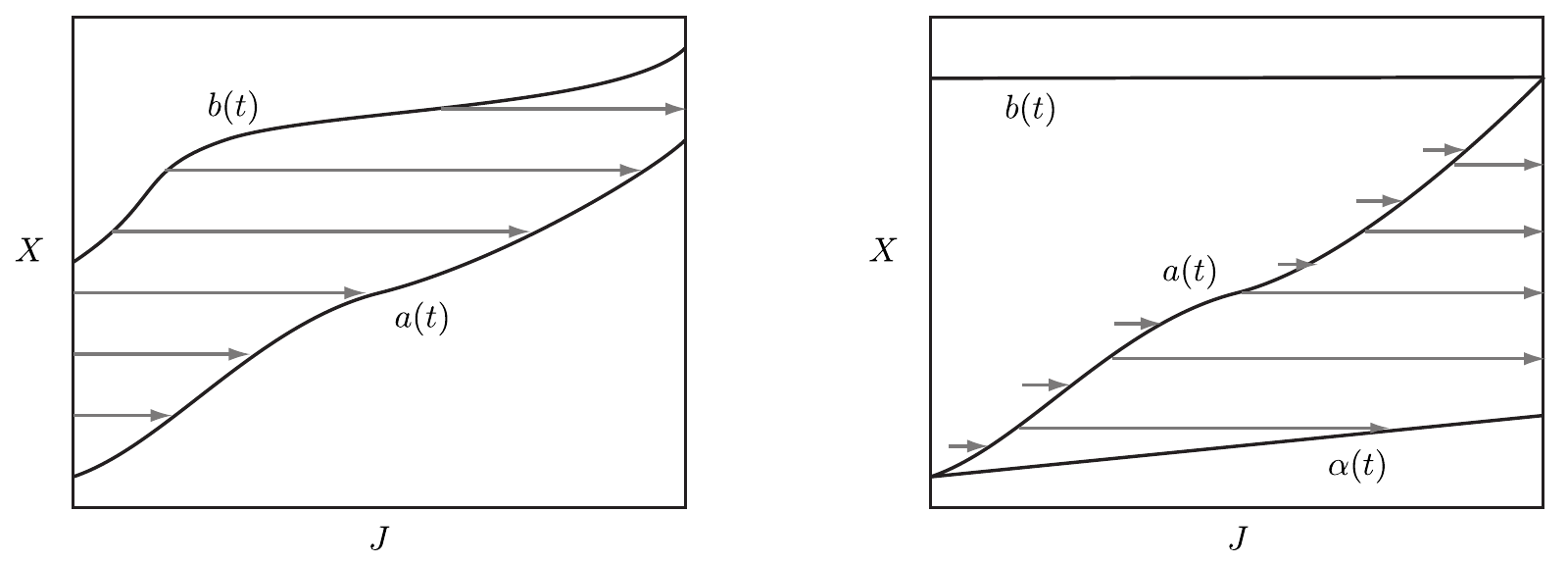}}
 \caption{Schematic diagrams for the maps in the proofs of Lemmas~\ref{lem:def-retr-right} and
  \ref{lem:def-retr-cone}. }
  \label{fig:morse-lemmas}
\end{figure}

Next, we seek to understand $\left(B_J, A_J\right)$ in terms of the sublevel sets of 
$\delta$ at the left side of $J$:

\begin{lem} \label{lem:def-retr-cone} Let $\delta: X \to \rr$ be a smooth
  function whose negative gradient flow exists for all time, and let 
 $a, b:J \to \rr$ be continuous functions satisfying:
\begin{enumerate}
\item $b(t) = b(t_0)$ for all $t$;
\item $a(t)$ is strictly increasing, and $a(t_1) = b(t_0)$;
\item $a(t_0)$ has a neighborhood of regular values of $\delta$.
\end{enumerate}
Then 
$(B_J, A_J)$ deformation retracts  to the relative cone
$C\left(\delta^{b(t_0)}, \delta^{a(t_0)}\right)$.
 \end{lem}

\begin{proof} 
   Choose $\epsilon>0$ so that $a(t_0) + \epsilon < b(t_0)$ and so that there are no critical values of $\delta$ in $[a(t_0), a(t_0)+\epsilon]$.  Let $\alpha(t)$ be a function that strictly increases from $a(t_0)$ to $a(t_0) + \epsilon$ over $J$, and satisfies $\alpha(t) < a(t)$ for all $t > t_0$.  Define a map $\sigma: B_J \to B_J$ that is equal to the map $\rho_+(x,t)$ from the proof of Lemma~\ref{lem:def-retr-right} on $A_J$ (with $a$ taking the place of $b$ and $\alpha$ taking the place of $a$), is equal to the identity when $\delta(x) \geq a(t)+\epsilon$, and interpolates between these two extremes in the $t$ direction for $a(t) < \delta(x) \leq a(t)+\epsilon$; see the right side of Figure~\ref{fig:morse-lemmas}.  As in the proof of Lemma~\ref{lem:def-retr-right}, this map is homotopic to the identity.    To finish the proof, the negative gradient flow of $\delta$ defines a map
   that takes $(\sigma(B_J), \sigma(A_J))$ to $\left(J \times \delta^{b(t_0)}, J \times \delta^{a(t_0)}  \cup 
   \{t_1\} \times \delta^{b(t_1)} \right)  = C\left(\delta^{b(t_0)}, \delta^{a(t_0)}\right)$, as desired.
\end{proof}

In practice, we will often need to expand and/or retract $(B_J, A_J)$ by deforming $(B_t, A_t)$ for $t \in J$ before we can apply Lemmas~\ref{lem:def-retr-right} and \ref{lem:def-retr-cone}.  We capture this maneuver in the following definition:
 
\begin{defn} Pairs $(B_J, A_J)$ and $(\widetilde B_J, \widetilde A_J)$ are \dfn{fiberwise homotopy equivalent} if there exists a
  homotopy equivalence $H: (B_J, A_J) \to (\widetilde B_J, \widetilde A_J)$ that, for all $t \in J$, restricts to a homotopy equivalence $H_t: (B_t, A_t) \to (\widetilde B_t, \widetilde A_t)$.  \end{defn}

\begin{lem}\label{lem:vary-a-b} Assume $\delta: X \to \rr$ is a function whose gradient flow exists for all time.  Given continuous functions $a, \tilde a, b, \tilde b : J \to \rr$ with $a(t) \leq b(t)$ and $\tilde a(t) \leq \tilde b(t)$ for all $t$, let $A_J, \widetilde A_J, B_J, \widetilde B_J$ be as defined in Equation~\ref{eqn:A-B}.  If for all $t \in J$, there exist no critical values of $\delta$ between $a(t)$ and $\tilde a(t)$ and between $b(t)$ and $\tilde b(t)$, then $(B_J, A_J)$ and $(\widetilde B_J, \widetilde A_J)$ are fiberwise homotopy equivalent.  \end{lem}

\begin{proof}  The desired homotopy equivalence arises from following the positive or negative gradient flow of $\delta$ in each $t$-slice.
\end{proof}

\subsection{Analysis of functions corresponding to Particular Sublevel Sets}
\label{ssec:sublevel-at-infty}

A crucial part of our analysis of sublevel sets of the sheared difference function in Sections~\ref{sec:filling-iso} and \ref{sec:cobord-les} will occur over intervals that lie outside of $[t_-,t_+]$ and at levels $\Infin$ and $\pm \mu$.  We will want to show that, after a fiberwise homotopy equivalence, the hypotheses of Lemma~\ref{lem:def-retr-right} and \ref{lem:def-retr-cone} can be applied to sublevel sets of $\delta$ with the functions of the form $\lambda_\tau(t)$.  To this end, we analyze the behavior of $\lambda_\Infin$ and $\lambda_{\pm \mu}$ in this section.  We suggest that the reader bypass this section on first reading, looking only at Figures~\ref{fig:Infin}, \ref{fig:mu}, and \ref{fig:-mu}.

For the remainder of this section, suppose that $(\leg_-, f_-) \prec_{(\clag, F)} (\leg_+, f_+)$ is a gf-compatible Lagrangian cobordism of $T^*(\rp \times M)$, and that $H \in \mathcal H(\clag)$.

The first lemma examines the levels $\lambda_\Infin$ and explains some of the lower bounds on $\Infin$ in Inequalities~(\ref{ineq:wgh}).

\begin{lem} \label{lem:Infin} For $\Infin$ meeting the requirements of (\ref{ineq:wgh}), the function $\lambda_\Infin$ satisfies the following:
  \begin{enumerate}
  \item For all $t \geq t_+$, $\lambda_\Infin(t)  > \lmaxp$;
  \item There exists $t_-^c < t_-$ so that $\lambda_\Infin(t)$ has a unique maximum on $(0,t_-]$ at $t_-^c$, and although
    $\lambda_\Infin(t)$ is decreasing on $[t_-^c, t_-]$, $\lambda_\Infin(t) > \lmaxm$ on that interval. 
  \item If $v_-$ satisfies
    Inequality~(\ref{ineq:vpm}), then $\lambda_\Infin'(v_-) > 0$.
  \end{enumerate}
 See Figure~\ref{fig:Infin}.
\end{lem}

We will use the following sublemma implicitly throughout the proofs in this subsection; the proof comes from direct calculations and the defining conditions of $H \in \mathcal{H}(\clag)$:

\begin{slem} \label{slem:lambda-sign}
  The sign of $\lambda_\sigma'(t)$ is governed by the sign of $l_\sigma(t) = -tH'(t) -\Infin + H(t)$. Since $l_\sigma'(t) = -tH''(t)$, $l_\sigma$ is increasing when $t \geq t_+$; similarly, $l_\sigma$ is decreasing when $t\leq t_-$.
\end{slem}

\begin{proof}[Proof of Lemma~\ref{lem:Infin}] 
  To prove (1), we will show that at the unique minimum $t^c_+$ of $\lambda_\Infin$ on $[t_+, \infty)$, with $\lambda_\Infin(t^c_+) > \lmaxp$.  Let $\overline{t}_+$ denote the unique point in $(t_+, 2t_+)$ with $-H'(\overline{t}_+) = \lmaxp$.  As shown in Proposition~\ref{prop:crit-pt-wgh}, $H(\overline{t}_+) = t_+\lmaxp + \frac{\lmaxp^2}{2r_+}$.  Since $\overline{t}_+ < 2t_+$ and $\Infin > 3t_+ \lmaxp + \frac{\lmaxp^2}{2r_+}$ from Inequality~(\ref{ineq:wgh}), we may compute that $l_\Infin(\overline{t}_+) < 0$.  On the other hand, $l_\Infin(t)>0$ for sufficiently large $t$.  Thus we see that the unique minimum of $\lambda_\Infin$ occurs at some $t_+^c > \overline{t}_+$.  Since $l_\Infin(t_+^c) = 0$, we obtain the relation $- H(t_+^c) = -t_+^c H'(t_+^c) - \Infin$, and thus
$$  
\lambda_\Infin(t_+^c)=-H'(t_+^c) > -H'(\overline{t}_+) = \lmaxp,
$$
as desired.

A similar argument when $t \leq t_-$ that uses the inequality $2t_- \lmaxm < \Infin$ from (\ref{ineq:wgh}) proves (2). Furthermore, if $v_-$ satisfies Inequality~(\ref{ineq:vpm}), then 
$\lambda_\Infin(t) < \lmaxm$ for all $t \leq v_-$.  It follows that $v_- < t_-^c$, and hence $v_-$ must be contained in the region where $\lambda_\Infin$ is increasing; this proves (3).
\end{proof}

 \begin{figure}
  \centerline{\includegraphics{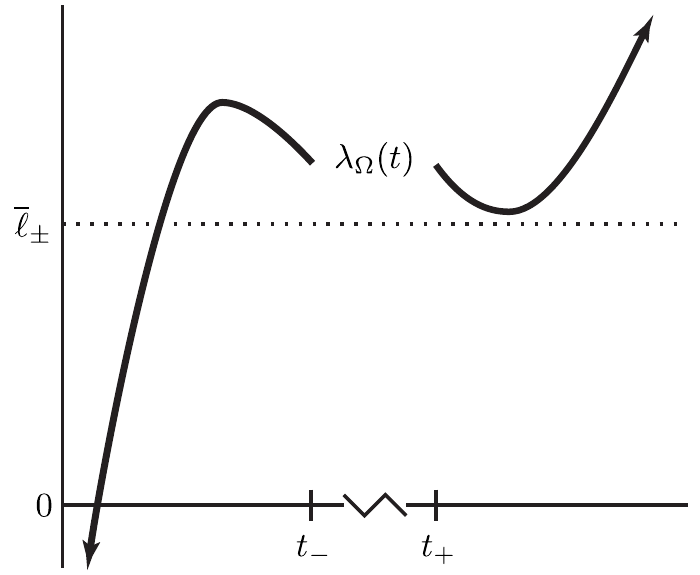}}
  \caption{A schematic picture of $\lambda_\Infin(t) = \frac{1}{t}(\Infin - H(t))$ for $H \in \mathcal H(\clag)$
  and $\Infin$ satisfying Inequalities~(\ref{ineq:wgh}).}
  \label{fig:Infin}
\end{figure}
 
The next lemma explains some of the upper bounds imposed on $\mu$ in Inequality~(\ref{ineq:wgh}).

\begin{lem} \label{lem:mu}
For $\mu$ meeting the requirements of (\ref{ineq:wgh}), the function $\lambda_\mu(t)$ satisfies the following:
\begin{enumerate}
\item There exists a unique minimum $t_+^c \in (t_+, u_+)$ for $\lambda_\mu$ on $[t_+,\infty)$ so that  $0 < \lambda_\mu(t) < \lminp$ for all $t \in [t_+, u_+]$.
\item  There exists a unique maximum $t_-^c \in (u_-, t_-)$ for $\lambda_\mu$ on $(0,t_-]$ so that $0 < \lambda_\mu(t) < \lminm$ for all $t \in [t_-^c, t_-]$.
\end{enumerate}
See Figure~\ref{fig:mu}.
\end{lem}

 \begin{figure}
  \centerline{\includegraphics{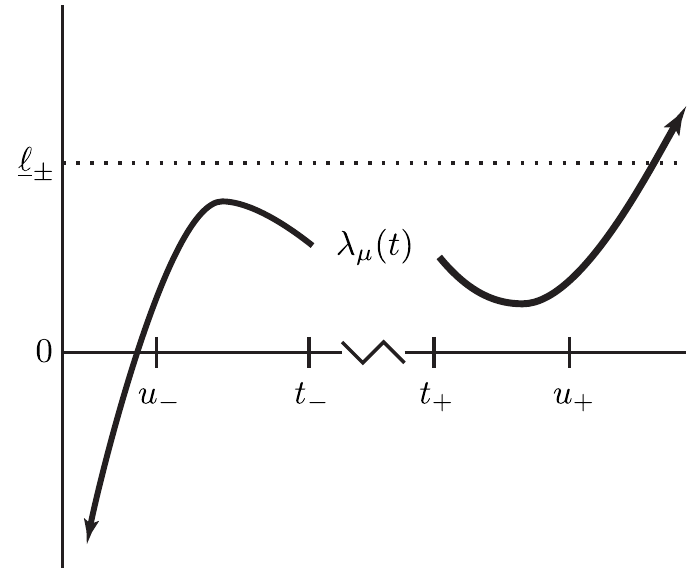}}
  \caption{A schematic picture of $\lambda_\mu$, for $H \in \mathcal H(\clag)$
  and $\mu$ satisfying Inequalities~(\ref{ineq:wgh}).}
  \label{fig:mu}
\end{figure}

\begin{proof} 
  We begin with the case of $t \geq t_+$.  To show that there exists a minimum (Sublemma~\ref{slem:lambda-sign} guarantees that such a minimum would be unique), we compute that $l_\mu(t_+) = -\mu < 0$ on one hand, and that $l_\mu(u_+) = \frac{r_+}{2}(u_+^2-t_+^2) - \mu > 0$, since $\mu < \frac{r_+}{2}(u_+^2- t_+^2)$ by Inequality~(\ref{ineq:wgh}).  Thus, at some point $t_+^c \in (t_+, u_+)$, $\lambda_\mu$ has a minimum. 

  Since $l_\mu(t_+^c) = 0$, we see that $ - H(t_+^c) = -t_+^c H'(t_+^c) - \mu$, and hence $\lambda_\mu(t_+^c)=-H'(t_+^c) > 0$.  Since $\mu < t_+ \lminp$ by Inequality~(\ref{ineq:wgh}), we have $\lambda_\mu(t_+) < \lminp$, and so, for $t \in [t_+, t_+^c]$, $p(t) \in (0,\lminp)$.  In addition, the inequalities $\mu < \frac{u_+}{2}\lminp < t_+ \lminp$ from Inequality~(\ref{ineq:wgh}), $r_+ < \frac{\lminp}{u_+ - t_+}$ from Inequality~(\ref{eqn:ru}), and $u_+ < 2t_+$ imply that $\lambda_\mu(u_+) < \lminp$.  This finishes the proof of (1).

  A similar argument when $t < t_-$ using the inequalities $\mu < \frac{r_-}{2}(t^2_--u^2_-)$ and $\mu < t_-\lminm$ from (\ref{ineq:wgh}) yields (2).
\end{proof}

\begin{rem}  \label{rem:sigmat+u+}
For later purposes, it will be useful to point out that the proof of Lemma~\ref{lem:mu} shows that
when $0< \mu < \frac{u_+}{2} \lminp$, we have $\lambda_\mu(t) \in (0,\lminp)$ for all $t \in [t_+,u_+]$.
 \end{rem}

\begin{lem} \label{lem:-mu} For $\mu$ meeting the requirements of (\ref{ineq:wgh}), the function $\lambda_{-\mu}$ satisfies the following:
  \begin{enumerate}
  \item $\lambda_{-\mu}$ is increasing on $(-\infty, t_-) \cup (t_+, \infty)$,
  \item $-\lminpm < \lambda_{-\mu}(t_\pm) < 0$,  and
  \item  $0< \lambda_{-\mu}(u_+) < \lminp$.
 \end{enumerate}
 See Figure~\ref{fig:-mu}.
\end{lem}

 \begin{figure}
  \centerline{\includegraphics{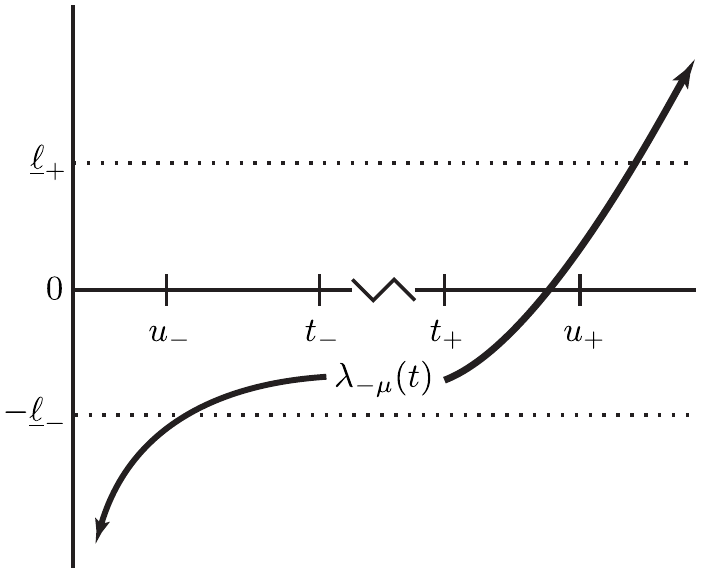}}
  \caption{A schematic picture of $\lambda_{-\mu}$, for $H \in \mathcal H(\clag)$
  and $\mu$ satisfying Inequalities~(\ref{ineq:wgh}).}
  \label{fig:-mu}
\end{figure}

\begin{proof}
  Since $l_{-\mu}(t_\pm) = \mu > 0$, Sublemma~\ref{slem:lambda-sign} implies that $\lambda_{-\mu}(t)$ is increasing when $t \in (-\infty, t_-) \cup (t_+, \infty)$. Parts (2) and (3) follow from direct calculations using Inequalities~(\ref{ineq:wgh}) and a calculation similar to the one in the proof of Lemma~\ref{lem:mu}.
 \end{proof}

\section{Filling Isomorphisms}
\label{sec:filling-iso}

In this section, we will prove Theorem~\ref{thm:filling-iso}.  Namely, we will show that if
$(\emptyset, f_-) \prec_{(\clag, F)} (\leg_+, f_+)$, 
then 
\begin{equation} \label{eqn:gh-iso}
  \rgh{k}{[f_+]} \simeq H^{k+1}(L, \leg_+)
  \quad \text{and} \quad   \gh{k}{[f_+]} \simeq H^{k+1}(L).
\end{equation}
This will follow easily from  Proposition~\ref{prop:wgh-vanish} and the following theorem.

\begin{thm} \label{thm:filling-les}  If $(\leg_-, f_-) \prec_{(\clag, F)} (\leg_+, f_+)$,
and $\clag$ is orientable, 
 then, for $\lag = \clag \cap \{ t \in (0, t_+]\}$, there are long exact sequences
  \begin{equation*}
    \xymatrix@R=5pt{
     \cdots \ar[r] & \wgh{k+1}{F} \ar[r] & H^{k+1}(\lag, \partial \lag)
   \ar[r]^{\phi_F^*} & \rgh{k}{[f_+]} \ar[r] & \cdots\\
      \cdots \ar[r] & \wgh{k+1}{F} \ar[r] & H^{k+1}(\lag)
    \ar[r]^{\widetilde \phi_F^*} & \gh{k}{[f_+]} \ar[r] & \cdots.
 }
  \end{equation*}
\end{thm}

 The idea of the proof of Theorem~\ref{thm:filling-les} is to realize the pair $( \Delta_{[v_-, v_+]}^\Infin, \Delta_{[v_-, v_+]}^{-\mu})$
as two different mapping cones:
 \begin{enumerate}
\item Over $[t_+, v_+]$ (resp. $[u_+, v_+]$), the pair may be associated with the relative cone on $(\delta_+^\infin, \delta_+^{-\epsilon})$ (resp. $(\delta_+^\infin, \delta_+^{\epsilon}))$;
\item Over $[v_-, t_+]$, the pair may be associated with a disk bundle over the Morse-Bott submanifold $C$, which is diffeomorphic to the
manifold-with-boundary $\lag$, relative to the boundary sphere bundle; over $[v_-, u_+]$, we obtain the same disk bundle, but now taken relative to the boundary sphere bundle {\it and} the disk bundle over $\partial C$. 
\end{enumerate}
 
In Subsection~\ref{ssec:cones}, we state a number of lemmas (which are proved in Subsection~\ref{ssec:lem-proofs}) that make the above outline more precise; we then prove Theorem~\ref{thm:filling-les}.  In Subsection~\ref{ssec:commute}, we show that there are natural vertical maps between the two long exact sequences in Theorem~\ref{thm:filling-les} that produce a commuting diagram that will be useful the discussion of duality in Section~\ref{sec:duality}.
  
\subsection{Realizing Pairs as Relative Mapping Cones} \label{ssec:cones}

Our first goal is to show that $\left( \Delta_{[v_-, v_+]}^\Infin, \Delta_{[v_-, v_+]}^{-\mu} \right)$ can be realized as a mapping cone in two different ways.  

The following Lemmas ~\ref{lem:filling-cone}, \ref{lem:filling-retract}, and \ref{lem:filling-cobord-cohom} will easily lead to the proof of Theorem~\ref{thm:filling-les}.  The proofs of these lemmas appear in Subsection~\ref{ssec:lem-proofs}.  Throughout this section, we will assume that 
$(\emptyset, f_-) \prec_{(\clag, F)} (\leg_+, f_+)$ is an orientable, gf-compatible Lagrangian cobordism of $T^*(\rp \times M)$. Further, we will fix a Hamiltonian shearing function $H \in H(\clag)$ and $\Infin, \mu>0$ satisfying Inequalities~(\ref{ineq:wgh}).

\begin{lem} \label{lem:filling-cone} There are diffeomorphisms of pairs
  \begin{align*}
    (\Delta_{\{u_+\}}^\Infin, \Delta_{\{u_+\}}^{- \mu}) &\simeq (\delta_+^\infin, \delta_+^\epsilon)  \text{ and} \\
    ( \Delta_{\{t_+\}}^\Infin, \Delta_{\{t_+\}}^{- \mu} ) & \simeq (\delta_+^\infin, \delta_+^{-\epsilon}), 
  \end{align*}
  where $\infin, \epsilon$ satisfy Inequalities~(\ref{ineq:epsilon-infin}).  Moreover, for any $v_+> u_+$ satisfying Inequality~(\ref{ineq:vpm}), after applying a fiberwise homotopy equivalence,
  there are deformation retractions:
  \begin{align*}
    {\rho}_+: &\left(\Delta^\Infin_{[u_+, v_+]}, \Delta^{-\mu}_{[u_+, v_+]}\right) \to 
    C\left(\delta_{+}^{\infin}, \delta_{+}^{\epsilon}\right)  \text{ and }\\
   \widetilde \rho_+: &\left(\Delta^\Infin_{[t_+, v_+]}, \Delta^{-\mu}_{[t_+, v_+]}\right) \to 
    C\left(\delta_{+}^{\infin}, \delta_{+}^{-\epsilon}\right),  \\
  \end{align*}
  with ${\rho}_+|_{\Delta^\Infin_{\{u_+\}}}  = \id$ and $\widetilde\rho_+|_{\Delta^\Infin_{\{t_+\}}}  = \id$.
\end{lem}

For the next lemmas, select $\sigma > 0$ so that 
\begin{equation} \label{ineq:sigma}
r_+u_+(u_+ - t_+) < \sigma < \frac{ u_+}{2}\lminp.
\end{equation}
Note that such a $\sigma$ always exists by condition (3) of Definition~\ref{defn:shear-H}.

\begin{lem} \label{lem:filling-retract} Suppose that $v_- < t_-$ 
 satisfies Inequality~(\ref{ineq:vpm}). After applying a fiberwise homotopy equivalence,
  there exist deformation retractions:
  \begin{align*}
    {\rho}_-: &\left(\Delta^\Infin_{[v_-, u_+]}, \Delta^{-\mu}_{[v_-, u_+]}\right)
    \to 
    \left(\Delta^{ \sigma}_{[v_-, u_+]},  \Delta_{ [v_-, u_+]}^{-\mu} \right),\\
    \widetilde\rho_-: &  \left(\Delta^\Infin_{[v_-, t_+]}, \Delta^{-\mu}_{[v_-, t_+]}\right)
    \to 
    \left(\Delta^{ \sigma}_{[v_-, t_+]},  \Delta_{ [v_-, t_+]}^{-\mu} \right).
  \end{align*}
\end{lem} 

From Lemmas~\ref{lem:filling-cone} and \ref{lem:filling-retract}, we get:

\begin{cor} \label{cor:2cones}  The pair $\left( \Delta_{[v_-, v_+]}^\Infin, \Delta_{[v_-, v_+]}^{- \mu} \right)$ can be viewed
  as two different mapping cones: $C(\phi_{F})$, where 
  $$\phi_{F}: \left( \Delta_{\{u_+\}}^\Infin, \Delta_{\{ u_+ \}}^{-\mu} \right) \to \left(\Delta^\sigma_{[v_-, u_+]}, \Delta^{-\mu}_{[v_-, u_+]}\right)$$ is given by the restriction of the map $\rho_-$ in Lemma~\ref{lem:filling-retract} to $\Delta_{\{u_+\}}^\Infin$, 
  and $C(\widetilde \phi_{F})$, where 
  $$\widetilde \phi_{F}: \left( \Delta_{\{t_+\}}^\Infin, \Delta_{\{ t_+ \}}^{-\mu} \right) \to \left(\Delta^\sigma_{[v_-, t_+]}, \Delta^{-\mu}_{[v_-, t_+]}\right)$$ 
  is given by the restriction of the map $\widetilde\rho_-$ in Lemma~\ref{lem:filling-retract} to $\Delta_{\{t_+\}}^\Infin$.  
  \end{cor}

  The following lemma will be useful in identifying terms that arise in the long exact sequences of the mapping cones of Corollary~\ref{cor:2cones}:
 
  \begin{lem} \label{lem:filling-cobord-cohom} There exist isomorphisms
\begin{align*}
H^{k+ N }\left( \Delta_{[v_-, u_+]}^{\sigma}, \Delta_{[v_-, u_+]}^{-\mu} \right) &\simeq H^{k}(\lag, \partial \lag), \text{ and} \\
H^{k+ N }\left( \Delta_{[v_-, t_+]}^{\sigma}, \Delta_{[v_-, t_+]}^{-\mu} \right) &\simeq H^{k}(\lag).
\end{align*}
 \end{lem}

\begin{proof} [Proof of Theorem~\ref{thm:filling-les}]
We prove the statement for the relative generating function cohomology $\rgh{*}{f_+}$; the proof for the total generating family cohomology is almost verbatim.
By Corollary~\ref{cor:2cones},  Lemma~\ref{lem:cone-exact} gives a long exact sequence:
\begin{equation}
\begin{aligned} \label{les:u+}
\cdots \to H^{*}\left(\Delta_{[v_-, v_+]}^\Infin, \Delta_{[v_-, v_+]}^{-\mu} \right)  & \to H^*\left( \Delta_{[v_-, u_+]}^{ \sigma}, \Delta_{[v_-, u_+]}^{-\mu}   \right) \\ &\stackrel{\phi_F^*}{\to} 
H^*\left( \Delta_{\{u_+\}}^\Infin, \Delta_{\{u_+\}}^{- \mu}  \right) \to \cdots.
\end{aligned}
\end{equation}
We now identify terms.  Corollary~\ref{cor:wgh-cpct} and Definition~\ref{defn:wgh} allow us to identify the first term:
$$H^{k+N + 1}\left(\Delta_{[v_-, v_+]}^\Infin, \Delta_{[v_-, v_+]}^{-\mu} \right) \simeq
H^{k+N + 1}\left(\Delta^\Infin, \Delta^{-\mu} \right)  =
 \wgh{k+1}{\clag}.$$
Lemma~\ref{lem:filling-cobord-cohom} identifies the second term as $H^{k+1}(\lag, \partial\lag).$
The identification of the last term as $\rgh{k}{[ f_+ ]}$ follows immediately from Lemma~\ref{lem:filling-cone}.
\end{proof} 
 
\begin{proof}[Proof of Theorem~\ref{thm:filling-iso}] The stated isomorphism follows from Theorem~\ref{thm:filling-les}, the fact that $(\lag, \partial \lag)$ is diffeomorphic to $(L, \leg_+)$, and Proposition~\ref{prop:wgh-vanish} which guarantees the vanishing of the total wrapped generating family cohomology of $F$.
\end{proof}

\subsection{Proofs of Lemmas~\ref{lem:filling-cone}, \ref{lem:filling-retract}, and \ref{lem:filling-cobord-cohom}}
\label{ssec:lem-proofs}

\begin{proof}[Proof of Lemma~\ref{lem:filling-cone}]  \label{proof:lem-filling-cone} Fix $H \in H(\clag)$ and  $\Infin, \mu>0$ satisfying Inequalities~(\ref{ineq:wgh}).
The proof relies on two applications of Lemma~\ref{lem:def-retr-cone}.
With Equation~(\ref{eqn:end-level}) in mind, we take $a(t) = \lambda_{-\mu}(t)$ and $b(t) = \lambda_\Infin(t)$.  By Lemma~\ref{lem:-mu},
 $a(t)$ is strictly increasing when $t \geq t_+$.  For $v_+$ satisfying Inequality~(\ref{ineq:vpm}),
$a(v_+)$ is greater than all critical values of $\delta_+$. 
By Lemma~\ref{lem:Infin}, we can assume that $b(t)$ is greater than all critical values of $\delta_+$ for all $t \in [t_+, v_+]$.  After applying Lemma~\ref{lem:vary-a-b}, we can assume $b(t) = a(v_+)$ on $[t_+,v_+]$.     

Lemma~\ref{lem:-mu} then tells us that:
$$-\lminp < a(t_+) < 0 < a(u_+) < \lminp.$$   
Hence, by Lemma~\ref{lem:def-retr-cone},  we have: 
\begin{enumerate} 
\item The pair $\left(\Delta^\Infin_{[u_+, v_+]}, \Delta^{-\mu}_{[u_+, v_+]}\right)$ deformation retracts to the relative cone 
$C\left(\delta_+^{b(u_+)}, \delta_+^{a(u_+)}\right) =    C\left(\delta_{+}^{\infin}, \delta_{+}^{\epsilon}\right)$, and
\item The pair $\left(\Delta^\Infin_{[t_+, v_+]}, \Delta^{-\mu}_{[t_+, v_+]}\right)$ deformation retracts to the relative cone $C\left(\delta_+^{b(t_+)}, \delta_+^{a(t_+)}\right) = C\left(\delta_{+}^{\infin}, \delta_{+}^{-\epsilon}\right)$,
\end{enumerate}
for
$\infin, \epsilon$ satisfying Inequalities~(\ref{ineq:epsilon-infin}).
\end{proof}

\begin{proof}[Proof of Lemma~\ref{lem:filling-retract}]  
  Fix $H \in H(\clag)$, $\Infin, \mu >0$ satisfying Inequalities~(\ref{ineq:wgh}) and $v_-$ satisfying Inequality~(\ref{ineq:vpm}).  We will construct the deformation retractions $\rho_-$ and $\widetilde{\rho}_-$ by flowing along the negative gradient vector field of $\Delta$ (here, we implicitly fix a metric) on the manifolds-with-boundary $[v_-, u_+] \times M \times \rr^{2N}$ and $[v_-, t_+] \times M \times \rr^{2N}$, stopping when the value of $\Delta$ reaches $\sigma$.  The embeddedness of $\clag$ implies that $0$ is the only critical value of $\Delta$ on these manifolds, and hence it suffices to show that the negative gradient vector field is inward-pointing at the boundaries.  In particular, we will show:
  \begin{enumerate}
  \item $\partial_t \Delta < 0$ on $\left( \{v_-\} \times M \times \rr^{2N} \right) \cap \{ \sigma < \Delta \leq \Infin\}$,  for all $\sigma > 0$,  by the choice of $v_-$;
  \item $\partial_t \Delta > 0$ on $\left( [t_+,u_+] \times M \times \rr^{2N} \right) \cap \{\sigma < \Delta \leq \Infin\}$, when $\sigma$ is chosen to satisfy the lower bound in Inequalities (\ref{ineq:sigma}).  This fact is stronger than what we need for the lemma, but will prove useful in the next section.
  \end{enumerate}
  
When $t \leq t_-$, recall that $\Delta(t,x,\e,\te) = t\delta_-(x,\e,\te) + H(t)$ and hence that $\partial_t \Delta = \delta_-(x, \e, \te) + H'(t)$.  Using the notation of Section~\ref{ssec:sublevel-at-infty}, since $\Delta < \Infin$, we have $\delta_- \leq \lambda_\Infin$, and hence that $\partial_t \Delta \leq \lambda_\Infin(t) + H'(t)$.  Rewriting this inequality yields:
$$t \partial_t \Delta(t,x,\e,\te) \leq l_\Infin(t),$$
so Lemma~\ref{lem:Infin} implies that  
$t \partial_t \Delta|_{t = v_-} < 0$, and hence  $\partial_t \Delta|_{t = v_-} < 0$, as desired.

For $t\in [t_+, u_+]$, recall that $\Delta(t,x,\e,\te) = t\,\delta_+(x,\e,\te) + H(t)$.  Since $H''(t) < 0$, we have: 
\begin{equation}
 \partial_t \Delta = \delta_+(x,\e,\te) + H'(t) > \delta_+(x,\e,\te) - r_+(u_+ - t_+). \label{eqn:deriv-Delta-2}
\end{equation}
The inequality $\Delta(t, x, \e,\te) > \sigma$ implies $t \delta_+(x,\e,\te) + H(t) > \sigma$ and hence, since $H(t) \leq 0$,
that 
\begin{equation*}
  \delta_+(x,\e,\te) > \frac{\sigma}{t} > \frac{\sigma}{u_+}.
\end{equation*}
Applying this inequality and the lower bound on $\sigma$ from
Inequality~(\ref{ineq:sigma}) to Equation~(\ref{eqn:deriv-Delta-2}) yields the desired positivity of the derivative $\partial_t \Delta$ when $t \in [t_+,u_+]$.
\end{proof}

\begin{proof}[Proof of Lemma~\ref{lem:filling-cobord-cohom}]

  We will first show that the cohomology groups of the pair $\left(\Delta^\sigma_{[v_-,u_+]}, \Delta^{-\mu}_{[v_-,u_+]}\right)$ agree with those of
  \begin{equation}
    \label{eqn:u+pair} \left(\Delta^{\sigma}_{[t_-, t_+]} , 
      \Delta_{[t_-, t_+]}^{-\mu} \cup \Delta_{\{t_+\}}^{\sigma}\right),
  \end{equation}
  and the cohomology groups of $\left(\Delta^\sigma_{[v_-, t_+]}, \Delta^{-\mu}_{[v_-, t_+]}\right)$ agree with those of
  \begin{equation}  \label{eqn:t+pair}  \left(\Delta^{\sigma}_{[t_-, t_+]} , 
      \Delta_{[t_-, t_+]}^{-\mu} \right).\
  \end{equation}
  We will then apply a Morse-Bott argument to identify the cohomology groups of the pair in (\ref{eqn:u+pair}) with those $(\lag, \partial \lag)$ and the cohomology groups of the pair in (\ref{eqn:t+pair}) with those of $\lag$.
 
  First, we consider the pair $\left(\Delta^\sigma_{[v_-,t_-]}, \Delta^{-\mu}_{[v_-,t_-]}\right)$.  Since $\Lambda_- = \emptyset$,
  after a fiberwise homotopy equivalence, Corollary~\ref{cor:left-map-cone}, Lemma~\ref{lem:mu}, and Lemma~\ref{lem:-mu} show this pair retracts to $\left(\Delta^\sigma_{\{t_-\}},\Delta^{-\mu}_{\{t_-\}}\right)$.  Thus, the cohomology groups of $\left(\Delta^\sigma_{[v_-, t_+]}, \Delta^{-\mu}_{[v_-, t_+]}\right)$ agree with those of the pair in (\ref{eqn:t+pair}), as desired.   

  To get the desired identification between $\left(\Delta^\sigma_{[v_-, u_+]}, \Delta^{-\mu}_{[v_-, u_+]}\right)$ and the pair in (\ref{eqn:u+pair}), we will first employ Lemma~\ref{lem:def-retr-cone} to analyze $\left(\Delta^\sigma_{[t_+, u_+]}, \Delta^{-\mu}_{[t_+, u_+]}\right)$.  Restrict attention to $t$ in the interval $[t_+,u_+]$. Consider $a(t) = \lambda_{-\mu}(t) $ and $b(t) = \lambda_\sigma(t)$.  As noted in Remark~\ref{rem:sigmat+u+}, we can assume $0< b(t) < \lminp$.  By Lemma~\ref{lem:-mu}, $a(u_+)>0$.  After applying a fiberwise homotopy equivalence, we can assume $b(t) = a(u_+)$ for all $t \in [t_+, u_+]$.   By Lemma~\ref{lem:def-retr-cone}, we find that
  $\left( \Delta_{[t_+, u_+]}^\sigma, \Delta_{[t_+, u_+]}^{-\mu} \right) $ deformation retracts
  to  
 $$\left( \Delta_{ \{t_+\}}^\sigma\times [t_+, u_+], \left( \Delta_{\{t_+ \}}^{-\mu} \times [t_+, u_+]  \right) \cup \Delta_{\{ u_+ \}}^\sigma \right).$$
 Thus we see that the cohomology groups of 
$ \left(\Delta^\sigma_{[v_-,u_+]}, \Delta^{-\mu}_{[v_-,u_+]}\right)$, agree with those of
$$\left( \Delta_{[t_-, t_+ ]}^\sigma \cup \left(\Delta_{\{ t_+ \}}^\sigma \times [t_+, u_+]\right), 
\Delta_{[t_-, t_+ ]}^{-\mu} \cup \left( \Delta_{\{ t_+ \}}^{-\mu} \times [t_+, u_+]  \right) \cup  \Delta_{\{ u_+ \}}^\sigma \right),$$
which, after applying a diffeomorphism,  agree with those of the pair in (\ref{eqn:u+pair}), as desired.

To determine the cohomology groups of the pairs in (\ref{eqn:u+pair}) and (\ref{eqn:t+pair}),
 we will  analyze the change in topology as we pass through the critical level $0$ on the way up from $\Delta_{[t_-, t_+]}^{-\mu}$ to $ \Delta_{[t_-, t_+]}^{\sigma}$. To analyze the change, we will employ a simple modification of the standard constructions of Morse-Bott theory to allow for critical submanifolds with boundary.   By Proposition~\ref{prop:crit-pt-wgh}, there is a properly embedded, non-degenerate critical submanifold  $(C, \partial C) \subset \left( \{ t \in [t_-, t_+] \}, \{ t = t_+ \} \right)$ of index $N$ having critical value $0$.  A careful examination of the proof of the Morse-Bott lemma (as in \cite{bh:morse-bott-lemma}, say) shows that since the critical submanifold $C$ is properly embedded in $[t_-, t_+] \times M \times \rr^{2N}$, there is a choice of metric that allows us to assume that the negative gradient flow of $\Delta$ is tangent to the boundary $\{t=t_+ \}$ in a neighborhood of the boundary of $C$.  Thus, the effect of passing through the critical level is to attach an $N$-disk 
   bundle over $C$ to $\Delta^{-\mu}_{[t_-, t_+]}$ along its unit sphere bundle. By Lemma~\ref{prop:crit-pt-wgh}, we know that $C$ is diffeomorphic to $\lag$.  Denote the $N$-disk bundle by $D\lag$ and its sphere bundle by $S\lag$. We obtain a homotopy equivalence between the pairs
  $\left(\Delta^{\sigma}_{[t_-, t_+]}, \Delta^{-\mu}_{[t_-, t_+]} \cup \Delta_{\{t_+\}}^\sigma \right)$  and 
   $(D\lag, S\lag \cup D(\partial \lag))$ and
   between the pairs
   $\left(\Delta^{\sigma}_{[t_-, t_+]}, \Delta^{-\mu}_{[t_-, t_+]}  \right)$  and 
   $(D\lag, S\lag)$.

 The claimed isomorphism between $H^{k+N}\left(\Delta^{\sigma}_{[v_-, t_+]}, \Delta^{-\mu}_{[v_-, t_+]}\right)$ and $H^k(\lag)$ now follows from the Thom isomorphism theorem.  To complete the proof in the other case, consider the long exact sequences of the triple $(D\lag, S\lag \cup D(\partial \lag), S\lag)$ and the pair $(\lag,\partial \lag)$, related by the Thom maps $\tau_\lag$ and $\tau_{\partial \lag}$, along with an induced map $\overline{\tau}$:
 \begin{equation} \label{eqn:iso2}
  \xymatrix@C=15pt{
    \cdots \ar[r] & H^{k} (D\lag, S\lag \cup D(\partial \lag)) \ar[r] \ar[d]^{\overline{\tau}}
    & H^{k}(D\lag, S\lag) \ar[r] \ar[d]^{\tau_\lag} 
    & H^{k}(D(\partial \lag), S(\partial \lag)) \ar[r] \ar[d]^{\tau_{\partial \lag}} & \cdots\\
    \cdots\ar[r] & H^{k-N}(\lag, \partial \lag) \ar[r] & H^{k-N}(\lag) \ar[r] & H^{k-N}(\partial \lag) \ar[r] & \cdots 
 }
\end{equation}
The Thom isomorphism theorem and the $5$-Lemma imply that the map $\overline{\tau}$ is an isomorphism, thus giving
 $$H^{k+N}(\Delta^{\sigma}_{[v_-, u_+]}, \Delta^{-\mu}_{[v_-,u_+]}) \simeq
  H^{k+N}(D\lag, S\lag \cup D\partial \lag) \simeq H^{k}(\lag, \partial \lag),$$
  as desired.
\end{proof}

\subsection{Commutativity of Filling Isomorphisms\label{ssec:commute}}

The following lemma shows that the isomorphisms $\phi_F^*$ and $\widetilde \phi_F^*$ constructed in the proof of Theorem~\ref{thm:filling-les} commute with natural inclusion maps.  In Section~\ref{sec:duality}, this lemma will be employed to prove Theorem~\ref{thm:duality-filling}.

\begin{lem}\label{lem:cone-maps-commute}
  Assume $(\leg_-, f_-) \prec_{(\clag, F)} (\leg_+, f_+)$ where $\clag$ is orientable.
  Let $H \in \mathcal H(\clag)$ and $\Infin, \sigma, \mu > 0$ be chosen to 
satisfy the Inequalities~(\ref{ineq:wgh}) and (\ref{ineq:sigma}).
The maps $\phi_{F}$ and $\widetilde\phi_{F}$ defined in the proof of Theorem~\ref{thm:filling-les} fit into the following commutative diagram:
$$\xymatrix{H^{k+N+1}\left( \Delta_{[v_-, u_+]}^{\sigma}, \Delta_{[v_-, u_+]}^{-\mu} \right) \ar[r]^-{\phi_{F}^*} \ar[d]^{I^*} & H^{k+N+1}\left(\delta_+^\infin, \delta_+^\epsilon\right) \ar[d]^{i^*} \\H^{k+N+1}\left( \Delta_{[v_-, t_+]}^{\sigma}, \Delta_{[v_-, t_+]}^{-\mu} \right)  \ar[r]^-{\widetilde\phi_{F}^*} &H^{k+N+1}\left(\delta_+^\infin, \delta_+^{-\epsilon}\right)
}$$
where $i$ and $I$ are the obvious inclusion maps.
\end{lem}
 
\begin{proof}
  The main idea is to show that the following diagram commutes up to homotopy, where we shall define the maps $r$ and $\overline{\phi}_F$ below.
\begin{equation} \label{eqn:cone-maps-comm}
\xymatrix{
  & (\Delta^\Infin_{\{u_+\}}, \Delta^{-\mu}_{[t_+,u_+]}), \ar[dl]_-{\phi_F} \ar@<1ex>[d]^j \\
   (\Delta^\Infin_{[v_-,u_+]}, \Delta^{-\mu}_{[v_-,u_+]}) & (\Delta^\Infin_{[t_+,u_+]},  \Delta^{-\mu}_{[t_+,u_+]}) \ar@<1ex>[u]^r \ar[l]_-{\overline{\phi}_F} \\
  (\Delta^\Infin_{[v_-,u_+]}, \Delta^{-\mu}_{[v_-,u_+]}) \ar[u]^I & (\Delta^\Infin_{\{t_+\}}, \Delta^{-\mu}_{\{t_+\}}) \ar[u]^i \ar[l]_-{\widetilde{\phi}_F}
}
\end{equation}

We define $\overline{\phi}_F$ in the same way as we defined $\phi_F$ and $\widetilde{\phi}_F$:  simply follow the negative gradient flow of $\Delta$ until the value of $\Delta$ reaches $\sigma$ or less.  The commutativity of the bottom square in the diagram (\ref{eqn:cone-maps-comm}) is then clear, as is the relation
\begin{equation} \label{eqn:cone-map-tri}
  \phi_F = \overline{\phi}_F \circ j.
\end{equation}  

The map $r$ is a deformation retract constructed using Lemma~\ref{lem:-mu} and then applying Corollary~\ref{cor:left-map-cone}.  By pre-applying $r$ to each side of Equation~\ref{eqn:cone-map-tri}, we obtain:
$$ \phi \circ r = \overline{\phi}_F \circ j \sim \overline{\phi}_F,$$
and hence the diagram (\ref{eqn:cone-maps-comm}) commutes up to homotopy.

Finally, we observe that, up to the identifications in Lemma~\ref{lem:filling-cone} and possibly some deformations near $\omega$ and $\epsilon$, $r \circ i$ is the natural inclusion of $(\delta_+^\infin, \delta_+^{-\epsilon})$ into $(\delta_+^\infin, \delta_+^\epsilon)$.
\end{proof}	

\section{Duality} \label{sec:duality} 
 
In this section, we will prove Theorems~\ref{thm:duality} and \ref{thm:duality-filling}.  Namely, in Subsection~\ref{ssec:leg-dual}, we show that if $f$ generates $\Lambda^n$, there is a ``duality map'' $\phi: \rgh{j}{f} \to GH_{k}(f)$ when $j+k = n-1$.  In Subsection~\ref{ssec:filling-dual}, we show that up to the isomorphism between $\rgh{*}{f}$ and $H^*(L, \partial L)$ in Theorem~\ref{thm:filling-iso}, the duality map is essentially the same as the Poincar\'e-Lefschetz duality map.
 
\subsection{Duality for Generating Family Homology} \label{ssec:leg-dual}

We will prove Theorem~\ref{thm:duality} by extending a version of Alexander duality used by Fuchs and Rutherford in \cite{f-r}.  First, we isolate the application of Poincar\'e-Alexander-Lefschetz duality to our situation in the following lemma:
	
\begin{lem} \label{lem:+-iso} Assume $f: M^n \times \rr^N \to \rr$ is linear-at-infinity; let $\delta: M^n \times \rr^{2N} \to \rr$ be its
associated difference function.  For $\infin$ satisfying Inequality~(\ref{ineq:epsilon-infin}) and 
for all $a \in \rr$, there is an isomorphism
$$H^j\left(\delta^\infin, \delta^{-a} \right) \simeq H_{2N+n-j}\left(\delta^{a}, \delta^{-\infin} \right).$$
\end{lem}

\begin{proof} Since $f$ is linear-at-infinity, Lemma~\ref{lem:lin-diff} allows us to assume that $\delta$ is also linear-at-infinity.  As in the last section of \cite{f-r}, we compactify the super- and sublevel sets of $\delta$ inside $M \times S^{2N}$ so that we can apply standard duality theorems; by abuse of notation, however, we will still refer to the original sublevel sets below.  We examine the compact embedded pair $(\delta^{\geq -\infin}, \delta^{\geq a})$ inside $M \times S^{2N}$; since the super- and sublevel sets are absolute neighborhood retracts, the pair is tautly embedded and hence that we may use ordinary (co)homology theories throughout.  Poincar\'e-Alexander-Lefschetz duality (as formulated in \cite[\S6.2]{spanier}, for example) then yields: \begin{equation} \label{eq:pal-duality}
  H^j(\delta^{\geq -\infin}, \delta^{\geq a}) \simeq H_{2N+n-j}(\delta^a, \delta^{-\infin}).
\end{equation}
The map that exchanges $\e$ and $\te$ induces a homeomorphism $s$ between $\delta^a$ and $\delta^{\geq -a}$, and hence Equation~(\ref{eq:pal-duality}) becomes:
\begin{equation} 
  H^j(\delta^{\infin}, \delta^{-a}) \simeq H_{2N+n-j}(\delta^a, \delta^{-\infin}),
\end{equation}
as desired.
\end{proof}

This lemma is the key to the proof of the duality theorem for generating family cohomology.

\begin{proof}[Proof of Theorem~\ref{thm:duality}] 
The desired long exact sequence follows from the long exact sequence of Proposition~\ref{prop:full-rel-dual}, the isomorphism $\gh{k}{[f]} \simeq  GH_{n-1-k}([f])$ given by  Lemma~\ref{lem:+-iso} with $a=\epsilon$ and $j=k+N+1$, Corollary~\ref{cor:alt-gf-hom}, and the definitions of (relative) generating family (co)homology.
\end{proof}

The proof of Corollary~\ref{cor:arnold} is completely analogous to that in
\cite{high-d-duality}; we repeat it here for the reader's convenience.

\begin{proof}[Proof of Corollary~\ref{cor:arnold}]  We label the maps in the long exact sequence of Theorem~\ref{thm:cobord-les} as follows:
$$ \cdots \to GH_{n-k}([f]) \buildrel{\rho_k}\over \longrightarrow H^{k}(\Lambda) \buildrel{\sigma_k}\over \longrightarrow GH^{k}([f]) \to \cdots.$$
Let $m_k$ denote the number of critical points of index $k+N+1$ of the difference function $\delta$ for $f$.  We work over a field and denote the $k^{th}$ Betti number by $b_k = \dim H^k(\Lambda)$.  Finally, we compute: \begin{align*}
  b_k &= \dim \ker \sigma_k + \dim \image \sigma_k &\\
  &= \dim \image \rho_k + \dim \image \sigma_k &\\
  &\leq \dim  GH_{n-k}(f) + \dim GH^{k}(f)& \\
  &\leq m_{n-k} + m_{k}, &\text{by Proposition~\ref{prop:leg-crit-point},}   \\
  &= r_{n-k} + r_k,& \text{by Proposition~\ref{prop:index}.}
\end{align*} \end{proof}

\subsection{Duality and Lagrangian Spanning Surfaces}\label{ssec:filling-dual}

We will now prove Theorem~\ref{thm:duality-filling}, which shows that the duality map $\phi$ of Theorem~\ref{thm:duality} for the Legendrian $\Lambda_+$ corresponds to a well-known duality for the Lagrangian filling $(L, \Lambda_+)$.  We first work with a long exact sequence that will serve as an intermediary between the top and bottom sequences in the theorem.
 	
\begin{lem} \label{lem:les-pair} Assume $(\leg_-, f_-) \prec_{(\clag, F)} (\leg_+, f_+)$.
Consider $\Delta$ constructed with $H \in \mathcal H(\clag)$, and let $\epsilon, \sigma, \mu$, and $v_-$ satisfy Inequalities~(\ref{ineq:epsilon-infin}), (\ref{ineq:wgh}), (\ref{ineq:vpm}), and (\ref{ineq:sigma}). There exists a long exact sequence:
$$ 
\cdots \to H^{k-1}\left(\delta_+^\epsilon, \delta_+^{-\epsilon}\right) \to 
H^{k}\left(\Delta_{[v_-, u_+]}^\sigma, \Delta_{[v_-, u_+]}^{-\mu}\right) \to 
H^{k}\left(\Delta_{[v_-, t_+]}^\sigma, \Delta_{[v_-, t_+]}^{-\mu}\right)   \to \cdots.
$$
\end{lem}

\begin{proof} The long exact sequence in the lemma is simply that of the triple $\left( \Delta_{[v_-, u_+]}^\sigma, \Delta_{[v_-, t_+]}^\sigma \cup \Delta_{[t_+, u_+]}^{-\mu}, \Delta_{[v_-, u_+]}^{-\mu} \right)$ with the first and last terms identified as follows.  The first term is isomorphic to 
    $H^{k} ( \Delta_{[t_+, u_+]}^\sigma, \Delta_{\{t_+\}}^\sigma \cup \Delta_{[t_+, u_+]}^{-\mu})$ by excision, which in turn is isomorphic to 
  $H^{k} ( S(\delta_+^\epsilon, \delta_+^{-\epsilon} ))$ by Lemma~\ref{lem:def-retr-cone}, Remark~\ref{rem:sigmat+u+}, and Lemma~\ref{lem:-mu}.  Finally, Lemma~\ref{lem:rel-susp} gives us 
    $H^{k-1} \left( \delta_+^\epsilon, \delta_+^{-\epsilon}  \right)$, as desired.
The last term may be identified using excision.
\end{proof}

We are now ready to prove the theorem.	

\begin{proof}[Proof of Theorem~\ref{thm:duality-filling}] 
  Fix $H \in \mathcal H(\clag)$, $\epsilon, \infin, \sigma, \mu > 0$ according to Inequalities~(\ref{ineq:epsilon-infin}), (\ref{ineq:sigma}), (\ref{ineq:wgh}), and $v_-< t_-$ so it satisfies Inequality~(\ref{ineq:vpm}).  Consider the following diagram of long exact sequences, where the top row is given by the long exact sequence of the triple $(\delta_+^\infin, \delta_+^\epsilon, \delta_+^{-\epsilon})$, the middle row is given by the long exact sequence of Lemma~\ref{lem:les-pair}, and the third row is the long exact sequence of the pair $(\lag, \partial \lag)$, which is diffeomorphic to $(L,\Lambda_+)$.  To make the diagram more readable, we define the notation $T=[v_-, t_+]$ and $U=[v_-,u_+]$.
$$
\xymatrix@C=10pt{
  \cdots \ar[r] &  H^{k+N}( \delta_+^\epsilon, \delta_+^{-\epsilon} ) \ar[r] 
  & H^{k+N+1}(  \delta_+^\infin, \delta_+^\epsilon ) \ar[r] & H^{k+N+1}( \delta_+^\infin, \delta_+^{-\epsilon} ) \ar[r]& \cdots \\   
  \cdots \ar[r] & H^{k+N}( \delta_+^\epsilon, \delta_+^{-\epsilon} ) \ar[r] \ar@{=}[u]
  & H^{k+N+1}(\Delta_U^\sigma, \Delta_U^{-\mu}) \ar[r] \ar[u]^{\phi_F^*} 
  & H^{k+N+1}(\Delta_T^\sigma, \Delta_T^{-\mu}) \ar[r] \ar[u]^{\widetilde{\phi}_F^*}  & \cdots \\
  \cdots \ar[r] & H^k(\Lambda_+) \ar[r] \ar[u]^\simeq & H^{k+1}(L, \Lambda_+) \ar[u]^\simeq \ar[r] &H^{k+1}(L) \ar[u]^\simeq \ar[r] &\cdots
}
$$
The vertical maps from the third to the second row are given by the Thom isomorphism; see
Proposition~\ref{prop:leg-crit-point} and the proof of Lemma~\ref{lem:filling-cobord-cohom}.  
The only place where commutativity is not obvious is in the upper right square, where we may apply Lemma~\ref{lem:cone-maps-commute}.  

To complete the proof, we identify the rightmost terms.  Along the top, the proof of Theorem~\ref{thm:duality} yields the desired sequence.  Along the bottom, Poincar\'e-Lefschetz duality implies that $H^{k+1}(L) \simeq H_{n-k}(L, \Lambda_+)$.
\end{proof}

\section{The Long Exact Sequence of a Cobordism}
\label{sec:cobord-les}

In the previous two sections, we considered cobordisms with $\leg_- = \emptyset$.  We now consider the more general situation, culminating in the proof of Theorem~\ref{thm:cobord-les} and Corollary~\ref{cor:chantraine}.  As always, after the identification of $\rr \times J^1M$ and $T^*(\rp \times M)$, we will prove Theorem~\ref{thm:cobord-les} with the cobordism $\clag$ of $T^*(\rp \times M)$ in place of $\sclag \subset \rr \times J^1M$.

We use the following assumptions and notation throughout this section: 
$(\leg_-, f_-) \prec_{(\clag, F)} (\leg_+, f_+)$,  $\clag$ is orientable, 
$\lag = \clag \cap \{ t \in [t_-, t_+] \}$, $\partial \lag_\pm = \partial \lag \cap \{ t = t_\pm \}$, 
 $\Infin, \mu>0$ satisfy Inequalities~(\ref{ineq:wgh}), and $v_\pm$ satisfies Inequalities~(\ref{ineq:vpm}).

The long exact sequence in Theorem~\ref{thm:cobord-les} will be constructed in two steps.  We first realize the pair $\left(\Delta_{[v_-, v_+]}^\Infin, \Delta_{[v_-, v_+]}^{\mu}\right)$, as a relative mapping cone by examining the regions of $\rp \times M \times \rr^{2N}$ over $[v_-, t_+]$ and $[t_+, v_+]$.  This results in the following ``horizontal'' long exact sequence:

\begin{prop} \label{prop:horiz-les} There exists a long exact sequence:
  \begin{equation*}
    \xymatrix@1{
      \cdots \ar[r] & \rgh{k}{[f_-]} \ar[r]^{\Psi_F^*} & \rgh{k}{[f_+]} \ar[r] & 
      \rwgh{k+2}{F}  
      \ar[r]  & 
   \cdots.}
 \end{equation*}
\end{prop}
 
Second, the term $WGH^{k+2}(F)$ can be identified using the following ``vertical'' long exact sequence, which arises from examining the long exact sequence of the triple $(\Delta_{[v_-, v_+]}^\Infin, \Delta_{[v_-, v_+]}^{\mu}, \Delta_{[v_-, v_+]}^{-\mu})$:

\begin{prop} \label{prop:vert-les} There exists
  a  long exact sequence:
  \begin{equation*}
    \xymatrix@1{
      \cdots \ar[r] &
      \wgh{k+1}{F} \ar[r] & H^{k+1}(\lag, \partial \lag_+) \ar[r]^-{d^*} 
      & \rwgh{k+2}{F}
      \ar[r] &
      \cdots.
    }
  \end{equation*}
\end{prop}

From these two propositions, which we will prove below, we can easily obtain:

\begin{proof}[Proof of Theorem~\ref{thm:cobord-les}]
  By Proposition~\ref{prop:wgh-vanish}, $\wgh{*}{F}$ vanishes, and thus Proposition~\ref{prop:vert-les} yields:
 $$H^{k+1}(\lag, \partial\lag_+) \simeq WGH^{k+2}(F) .$$
Substituting this into the long exact sequence of Proposition~\ref{prop:horiz-les} yields Theorem~\ref{thm:cobord-les}.
\end{proof}

\begin{proof}[Proof of Corollary~\ref{cor:chantraine}]
  The only mysterious element in the statement is the sign.  This comes from two sources: first, \cite[\S3]{ees:high-d-geometry} and Proposition~\ref{prop:index} show that up to the sign $(-1)^{\frac{1}{2}(n-1)(n-2)}$, one may calculate the Thurston-Bennequin invariant to be the Euler characteristic of the generating family homology. Second, Poincar\'e duality immediately implies that $\chi(\slag, \Lambda_+) = (-1)^n\chi(\slag)$.
\end{proof}

\subsection{The ``Horizontal'' Long Exact Sequence}
\label{ssec:horiz-les}

In parallel to Subsection~\ref{ssec:cones}, the following lemmas analyze the pair $\left(\Delta_{[v_-, v_+]}^\Infin, \Delta_{[v_-, v_+]}^{\mu}\right)$. In parallel to Lemma~\ref{lem:filling-cone}, on $\{t \geq t_+\}$, we find:

\begin{lem} \label{lem:retract-cone}   
  The pair $\left( \Delta^\Infin_{\{t_+ \}} , \Delta^\mu_{\{t_+\}} \right)$ is diffeomorphic to $\left( \delta_+^\infin, \delta_+^\epsilon \right),$ where $\epsilon, \infin$ satisfy Inequalities~(\ref{ineq:epsilon-infin}).  Moreover, for any $v_+ > t_+$ satisfying Inequality~(\ref{ineq:vpm}), after applying a fiberwise homotopy equivalence, there is a deformation retraction
$$\rho_+: \left(\Delta^\Infin_{[t_+, v_+]}, \Delta^{\mu}_{[t_+, v_+]}\right) \to
C\left(\delta_{+}^{\infin}, \delta_{+}^{\epsilon}\right), \quad\text{ with }\rho_+|_{\Delta^\Infin_{\{t_+\}}}  = \id.$$
\end{lem}

The proof of this lemma is entirely similar to that of Lemma~\ref{lem:filling-cone} with Lemma~\ref{lem:mu} in place of Lemma~\ref{lem:-mu}.

In parallel to Lemma~\ref{lem:filling-retract}, on $\{t \leq t_+\}$, we make essential use of the embeddedness of $\clag$ to find:

\begin{lem} \label{lem:retract-t-} The pair
  $\left( \Delta^\Infin_{\{t_- \}} , \Delta^\mu_{\{t_-\}} \right)$ is diffeomorphic to $\left( \delta_-^\infin, \delta_-^\epsilon \right),$ where $\epsilon, \infin$ satisfy Inequalities~(\ref{ineq:epsilon-infin}). Moreover, for all $v_-$ satisfying Inequality~\ref{ineq:vpm}), after applying a fiberwise homotopy equivalence, there exists a deformation retraction, 
$$\rho_-: \left(\Delta^\Infin_{[v_-, t_+]}, \Delta^{\mu}_{[v_-, t_+]}\right)
\to \left(\Delta^\Infin_{\{t_-\}} \cup \Delta_{[t_-, t_+]}^{\mu},  \Delta_{[t_-, t_+]}^{\mu}\right).$$
\end{lem} 

\begin{proof}
  We will prove this in two steps.  First, we show that there exists a retraction from $\Delta^\Infin_{[t_-, t_+]}$ to $\Delta_{\{t_-\}}^\Infin \cup \Delta^{\mu}_{[t_-, t_+]}$ constructed by flowing along the negative gradient vector field of $\Delta$.  This part of the proof is essentially the same as in the proof of Lemma~\ref{lem:filling-retract}. 

  Second, we apply Corollary~\ref{cor:left-map-cone} to show that there exists a retraction from $\left( \Delta^\Infin_{[v_-, t_-]}, \Delta^{ \mu}_{[v_-, t_-]}\right)$ to $\left( \Delta^\Infin_{ \{t_-\}}, \Delta^{\mu}_{\{t_-\}}\right)$.  Consider 
  $a(t) = \lambda_\mu(t)$ and $b(t) = \lambda_\Infin(t)$ for $t \in [v_-, t_-]$.
  By Lemmas~\ref{lem:mu} and \ref{lem:vary-a-b}, we may assume that $a(t)$ is strictly increasing and $a(t_-) \in (0,\lminm)$.  By Lemma~\ref{lem:Infin} and \ref{lem:vary-a-b}, we may assume that $b(t)$ is strictly increasing
  and $b(t_-) > \lmaxm$.  Thus, our desired retraction follows from Corollary~\ref{cor:left-map-cone}.
\end{proof}

In parallel to Corollary~\ref{cor:2cones}, we obtain:

\begin{cor} \label{cor:cobord-cone} 
The pair $\left( \Delta_{[v_-, v_+]}^\Infin, \Delta_{[v_-, v_+]}^{\mu} \right)$ is
the mapping cone $C(\psi_{F})$, where 
$$\psi_{F}: \left(\Delta^\Infin_{\{t_+\}}, \Delta^{\mu}_{\{t_+\}}\right)
\to \left(\Delta^\Infin_{\{t_-\}} \cup \Delta_{[t_-, t_+]}^{\mu},  \Delta_{[t_-, t_+]}^{\mu}\right)
$$
 is the restriction of the map $\rho_-$ in Lemma~\ref{lem:retract-t-} to 
  $\Delta_{\{t_+\}}^\Infin$.
\end{cor}	

We now arrive at the horizontal long exact sequence:

\begin{proof}[Proof of Proposition~\ref{prop:horiz-les}]  By Corollary~\ref{cor:cobord-cone} and Lemma~\ref{lem:cone-exact}, there is a long exact sequence:
\begin{multline} \label{eqn:pre-horiz-les}
\cdots \to H^{k+N+1}\left( \Delta_{\{t_-\}}^\Infin \cup \Delta_{[t_-, t_+]}^{\mu}, \Delta_{[t_-, t_+]}^{\mu}  \right) 
\stackrel{\psi_F^*}{\to} H^{k+N+1}\left( \Delta_{\{t_+\}}^\Infin, \Delta_{\{t_+\}}^{\mu} \right) \\ 
\to  H^{k+N+2}\left(\Delta_{[v_-, v_+]}^\Infin, \Delta_{[v_-, v_+]}^{\mu} \right)  \to \cdots.
\end{multline}

We finish the proof of the proposition with the following identification of terms:
\begin{enumerate}
\item For the first term, we have
   \begin{align*}
  H^{k+N+1}\left( \Delta_{\{t_-\}}^\Infin \cup \Delta_{[t_-, t_+]}^{\mu},  \Delta_{[t_-, t_+]}^{ \mu}  \right) &\simeq
H^{k+N+1}\left( \Delta_{\{t_-\}}^\Infin, \Delta_{\{t_-\}}^{\mu}    \right) &\text{by excision}\\ &\simeq
H^{k+N+1}\left( \delta_{-}^\infin, \delta_{-}^{\epsilon}   \right) &\text{by Lemma ~\ref{lem:retract-t-} }\\ 
& = \rgh{k}{[f_-]}.
\end{align*}

\item For the second term, we have:
\begin{align*}
H^{k+N+1}\left( \Delta_{\{t_+\}}^\Infin, \Delta_{\{t_+\}}^{\mu} \right) &\simeq 
H^{k+ N + 1} \left( \delta_{+}^\infin, \delta_{+}^{\epsilon} \right) & \text{by Lemma ~\ref{lem:retract-cone} }\\
&= \rgh{k}{[f_+]}.
\end{align*}

\item By Lemma~\ref{lem:cobord-cpt-base}, we have:
$$H^{k+N+2}\left(\Delta_{[v_-, v_+]}^\Infin, \Delta_{[v_-, v_+]}^{\mu}\right) = WGH^{k+2}(F).$$
\end{enumerate}
\end{proof}

At this point, we can officially define the cobordism map $\Psi_F: \rgh{k}{f_-} \to \rgh{k}{f_+}$ via the following diagram, where $\psi_F$ is defined by following the negative gradient flow of $\Delta$ until the first point at which the flowline intersects $\Delta^\mu$ or $\{t_-\} \times M \times \rr^{2N}$.   The map $ex$ is excision.
\begin{equation} \label{eqn:cobord-map}
  \xymatrix{
    \rgh{*}{f_-} \ar@{<->}[d]^\simeq \ar@{-->}[r]^-{\Psi_{F}} & \rgh{*}{f_+} \ar@{<->}[d]^\simeq \\
    H^*(\Delta^\Infin_{\{t_-\}}, \Delta^\mu_{\{t_-\}}) \ar@{<->}[d]^{ex} &
    H^*(\Delta^\Infin_{\{t_+\}}, \Delta^\mu_{\{t_+\}}) \\
    H^*(\Delta^\Infin_{\{t_-\}} \cup \Delta^\mu_{[t_-, t_+]}, \Delta^\mu_{[t_-, t_+]}) \ar[ur]^{\psi_{F}^*} &
}
\end{equation}

The proof of the following lemma is a straightforward exercise in relating long exact sequences of mapping cones.

\begin{lem} \label{lem:cobord-indep-equiv}
  Up to the isomorphisms in Lemma~\ref{lem:cohom-equiv}, the cobordism map descends to equivalence classes of the generating family $F$.
\end{lem}

More interestingly, we have:

\begin{lem} \label{lem:cobord-indep-H}
  The cobordism map is independent of $H \in \mathcal H(\clag)$.
\end{lem}
 
\begin{proof} 
  Since the cobordism map depends only on $\Delta$ in the region where $t \in [t_-, t_+]$ and any $H \in \mathcal{H}(\clag)$ is zero on this region, the only dependence on $H$ that the cobordism map could have is on the choice of $t_\pm$.  We restrict attention to the case where $t^0_+ < t^1_+$ but $t^0_-=t^1_-$; the proof at the negative end is entirely similar.

To notational clarity, let us rename the map $\psi_F$ with
  $$\psi_{[t_-,t_+]}: (\Delta^\Infin_{\{t_+\}}, \Delta^\mu_{\{t_+\}}) \to
  (\Delta^\Infin_{\{t_-\}} \cup \Delta^\mu_{[t_-, t_+]}, \Delta^\mu_{[t_-, t_+]}).$$
If we let $\tilde{\psi}_{[t_-,t^0_+]}$ be the map $\psi_{[t_-,t^0_+]}$ whose domain has been expanded to be the pair $(\Delta^\Infin_{\{t^0_+\}} \cup \Delta^\mu_{[t^0_+, t^1_+]}, \Delta^\mu_{[t^0_+, t^1_+]})$, then the description of the $\psi_{[t_-,t_+]}$ maps above easily implies that
\begin{equation} \label{eqn:indep-comp}
  \psi_{[t_-,t^1_+]} = \tilde{\psi}_{[t_-,t^0_+]} \circ \psi_{[t^0_+,t^1_+]}.
\end{equation}

Combining the diagram (\ref{eqn:cobord-map}) with Equation~(\ref{eqn:indep-comp}) yields the following commutative diagram, where for clarity, we suppress degrees and use the notation $[i]$ to represent the interval $[t_-, t^i_+]$ and $[01]$ to represent $[t^0_+, t^1_+]$. 

$$
\xymatrix{
\rgh{*}{f_-} \ar@{<->}[d]^\simeq \ar@{-->}[r]_-{\Psi_{[0]}} \ar@(ur, ul)@{-->}[rr]^{\Psi_{[1]}} & \rgh{*}{f_+} \ar@{<->}[d]^\simeq \ar@{-->}[r]_-{\Psi_{[01]}} & \rgh{*}{f_+} \ar@{<->}[d]^\simeq \\
H^*(\Delta^\Infin_{\{t_-\}}, \Delta^\mu_{\{t_-\}}) \ar@{<->}[d]^{ex} &
H^*(\Delta^\Infin_{\{t^0_+\}}, \Delta^\mu_{\{t^0_+\}}) \ar@{<->}[d]^{ex} &
H^*(\Delta^\Infin_{\{t^1_+\}}, \Delta^\mu_{\{t^1_+\}}) \\
H^*(\Delta^\Infin_{\{t_-\}} \cup \Delta^\mu_{[0]}, \Delta^\mu_{[0]}) \ar@{<->}[d]^{ex} \ar[ur]^{\psi_{[0]}^*} &
H^*(\Delta^\Infin_{\{t^1_+\}} \cup \Delta^\mu_{[01]}, \Delta^\mu_{[01]}) \ar[ur]^{\psi_{[01]}^*} & \\
H^*(\Delta^\Infin_{\{t_-\}} \cup \Delta^\mu_{[1]}, \Delta^\mu_{[1]}) \ar[ur]^{\tilde{\psi}_{[0]}^*} \ar@/_40pt/[uurr]_{\psi^*_{[1]}}& &
}
$$

A direct computation of the gradient of $\Delta$ over the interval $[t^0_+,t^1_+]$ (where $\Delta = tf_+$) shows that the $M \times \rr^{2N}$ component of $\psi_{[01]}$ simply follows the negative gradient flow of $\delta$, and hence we may identify the two relative generating family cohomologies of $f_+$ as in Remark~\ref{rem:e-l-indep}.  This completes the proof.
\end{proof}

\subsection{The ``Vertical'' Long Exact Sequence}
\label{ssec:vert-les}

We now prove Proposition~\ref{prop:vert-les}, which boils down to the following lemma, whose proof is similar to that of Lemma~\ref{lem:filling-cobord-cohom}.
 
\begin{lem}  \label{lem:sigma-mu}   
  There exists an isomorphism between 
$H^{k+N}\left(\Delta_{[v_-, v_+]}^\mu, \Delta_{[v_-, v_+]}^{-\mu}\right)$ and $H^k(\lag, \partial \lag_+)$.
\end{lem}

\begin{proof} 
  We will show that  the cohomology groups of  $\left(\Delta^\mu_{[v_-,v_+]}, \Delta^{-\mu}_{[v_-,v_+]}\right)$ agree with those of   $$\left(\Delta^{\mu}_{[t_-, t_+]} , 
    \Delta_{[t_-, t_+]}^{-\mu} \cup \Delta_{\{t_+\}}^{\mu}\right).$$ The Morse-Bott argument in the proof of Lemma~\ref{lem:filling-cobord-cohom} then implies that the cohomology groups of this pair can be identified with those of $(\lag, \partial \lag_+)$.
   
  First, we consider the pair $(\Delta^\mu_{J}, \Delta^{-\mu}_{J})$, where $J = [v_-, t_-]$ and $J = [t_+, v_+]$.  Consider $a(t) = \lambda_{-\mu}(t)$ and $b(t) = \lambda_\mu(t)$. When $t \in [v_-, t_-]$, Lemmas~\ref{lem:mu}, \ref{lem:-mu}, \ref{lem:vary-a-b}, and Corollary~\ref{cor:left-map-cone} imply that $(\Delta^\mu_{[t_-, v_-]}, \Delta^{-\mu}_{[t_-, v_-]})$, retracts to $\left(\Delta^\mu_{\{t_-\}},\Delta^{-\mu}_{\{t_-\}}\right)$.

  When $J = [t_+, v_+]$, we will apply a Mayer-Vietoris argument.  First observe that Lemma~\ref{lem:mu} implies $0< b(t) < \lminp$ for all $t \in [t_+, u_+]$.  By Lemma~\ref{lem:-mu}, we can assume $a(u_+) >0$. Thus, after applying a fiberwise homotopy equivalence, we can assume, for a sufficiently small $\epsilon > 0$, that the pair $(\Delta^\mu_{u_+-\epsilon,u_++\epsilon}, (\Delta^{-\mu}_{u_+-\epsilon,u_++\epsilon})$ is acyclic.
A Mayer-Vietoris argument then shows that:
 $$\begin{aligned}
 H^*&\left( \Delta_{[t_+, v_+]}^\mu, \Delta_{[t_+, v_+]}^{-\mu} \right) 
 \simeq \\
 &H^* \left( \Delta_{[t_+, u_+ + \epsilon]}^\mu, \Delta_{[t_+, u_+ + \epsilon]}^{-\mu} \right) \oplus
 H^* \left( \Delta_{[u_+ - \epsilon, v_+]}^\mu, \Delta_{[u_+ - \epsilon, v_+]}^{-\mu} \right).
 \end{aligned}$$
 By Lemma~\ref{lem:mu}, Lemma~\ref{lem:-mu}, and Corollary \ref{cor:left-map-cone}, and the choice of
 $v_+$, we find
  $H^* \left( \Delta_{[u_+ - \epsilon, v_+]}^\mu, \Delta_{[u_+ - \epsilon, v_+]}^{-\mu} \right) = 0$.  
  Lemma~\ref{lem:def-retr-cone} implies that
 $\left( \Delta_{[t_+, u_+ + \epsilon]}^\mu, \Delta_{[t_+, u_+ + \epsilon]}^{-\mu} \right) $ deformation retracts
  to 
 $$\left( \Delta_{ \{t_+\}}^\mu \times [t_+, u_+], \left( \Delta_{\{t_+ \}}^{-\mu} \times [t_+, u_+]  \right) \cup \Delta_{\{ u_+ \}}^\mu \right).$$
The rest of the proof proceeds as in that of Lemma~\ref{lem:filling-cobord-cohom}.
\end{proof}

\begin{proof}[Proof of Proposition~\ref{prop:vert-les}] The triple
 $\left(\Delta_{[v_-, v_+]}^\Infin, \Delta_{[v_-, v_+]}^\mu, \Delta_{[v_-, v_+]}^{-\mu}\right)$ leads to the long exact sequence:
$$\begin{aligned}
\cdots &\to H^{k+N}\left(\Delta_{[v_-, v_+]}^\Infin, \Delta_{[v_-, v_+]}^\mu\right) \to 
H^{k+N}\left(\Delta_{[v_-, v_+]}^\Infin, \Delta_{[v_-, v_+]}^{-\mu}\right)  \\
&\text{\hskip 1in} \to  
H^{k+N}\left(\Delta_{[v_-, v_+]}^\mu, \Delta_{[v_-, v_+]}^{-\mu}\right) \to \cdots.
\end{aligned}
$$
   Applying Definition~\ref{defn:wgh}, Corollary~\ref{cor:wgh-cpct}, and Lemma~\ref{lem:sigma-mu}, we obtain the following identifications in the terms of the long exact sequence above.
 $$ \begin{aligned} H^{k+N}\left(\Delta_{[v_-, v_+]}^\Infin, \Delta_{[v_-, v_+]}^{\mu}\right)   &\simeq  \rwgh{k}{F}, \\
  H^{k+N}\left(\Delta_{[v_-, v_+]}^\Infin, \Delta_{[v_-, v_+]}^{-\mu}\right)  
  &\simeq \wgh{k}{F}, \\
   H^{k+N}\left(\Delta_{[v_-, v_+]}^\mu, \Delta_{[v_-, v_+]}^{-\mu}\right) &\simeq H^{k}\left( \lag, \partial \lag_+ \right).
   \end{aligned}
   $$
 \end{proof}

\section{Generating Family Cohomology as a TQFT}
\label{sec:tqft}

In this section, we establish several fundamental properties of the cobordism map defined in Section~\ref{sec:cobord-les}. First, it is important to check that this cobordism map is not always trivial; in the axioms of TQFT, this is often referred to as a ``normalization'' condition. In addition, if a symplectic isotopy (of a special form at the ends) is applied to the Lagrangian, we would naturally hope that we get a cobordism map that only differs by pre- and post-compositions of isomorphisms.  Lastly, we will show that if two Lagrangians (with matching end behavior) are glued, we get a cobordism map which is a composition of the cobordism maps of the pieces; this is often referred to as a ``functoriality'' axiom in TQFT.

We begin with the non-triviality of the cobordism map.

\begin{prop}[Non-Triviality] \label{prop:non-trivial} Given 
a Lagrangian cobordism of $\rr \times J^1M$ of the form 
$(\leg, f) \prec_{(\theta(Z_{\leg}), F)} (\leg, f)$, 
the cobordism map, $\psi_{[F]}$ is the identity.  \end{prop}

The proof of this proposition rests on two lemmas.  First, we show that if the generating family $F$ has the particularly simple form $F(t, x, \e) =tf(x, \e)$,
 then $\Psi_{[F]}$ is the identity. Second, we show that any generating family $F$ of $Z_\leg$ is equivalent to $tf$.

\begin{lem} \label{lem:straight-gf} Given a Lagrangian cobordism of the form $(\leg, f) \prec_{(\clag, tf)} (\leg, f)$, 
the cobordism map, $\Psi_{[tf]}$ is the identity.  
  \end{lem}

\begin{proof}  By Lemma~\ref{lem:cobord-indep-H}, $\Psi_F$ does not depend on the
choice of $H \in \mathcal H(\clag)$.  Since $F(t, x, \e) = tf(x, \e)$, for all $t$, we can take $t_- = t_+$ in the choice of
$H \in \mathcal H(\clag)$.  Then, by construction, $\psi_F$ in Corollary~\ref{cor:cobord-cone} will be the identity, and hence
$\Psi_F$ will be the identity.   
\end{proof}

\begin{lem} \label{lem:equiv-to-straight} Given a Lagrangian cobordism of the form $(\leg, f) \prec_{(\clag, F)} (\leg, f)$, if $\theta^{-1}(\clag) = Z_\leg \subset \rr \times J^1M$, then $F$ is equivalent to $t f$.  \end{lem}

\begin{proof}  Suppose that $F$ generates 
$$\clag = \theta(Z_\leg) = \{ (t, x, z, ty) : (x, y, z) \in \leg \},$$
and that outside of a compact interval of $\rp$, $F  = tf$.  The key points of the proof will be to find a path of generating families $F_s$ for $\clag$ that interpolate between $F$ and $tf$, and then to apply an argument of Th\'eret to $F_s$ to produce the desired equivalence.

We first show that after applying a fiber-preserving diffeomorphism, we may assume that the fiber-critical sets of $F$ and $tf$ agree, i.e.
 $\Sigma_F = \rp \times \Sigma_f$,
and that $F|_{\Sigma} = tf|_{\Sigma}.$
We use the notation $F(t, x, \e) = t f_t(x, \e)$, and hence we may write
$$\Sigma_F = \{ (t, x, \e) : \pd{f_t}{\e}(x, \e) = 0 \} = \bigcup_t  \left( \{t \} \times \Sigma_{f_t} \right).$$
Since for each $t$, $\Sigma_{f_t}$ is isotopic to $\Sigma_{f_{t_-}} = \Sigma_f$, the Isotopy Extension Theorem provides a compactly-supported fiber-preserving diffeomorphism that yields $\Sigma_{f_t} = \Sigma_f$ for all $t$.  A continuity argument shows that the embedding $j_{f_t}$ is also independent of $t$.

Since $tf_t(x, \e)$ generates
$$\{ \left(t, x, f_t(x, \e) + {t}\pd{f_t}{t}(x, \e), t \pd{f_t}{x}(x, \e) \right) :  (x, \e) \in \Sigma_f  \},$$
we have $f_t(x, \e) + {t}\pd{f_t}{t}(x, \e) = z$, where for fixed $x$ and $\e$, we think of $z$ as a fixed constant.  This equation is an ODE for $f_t(x,\e)$ as a function of $t$ with initial values $f_{t_-}(x,\e) = f(x,\e)$.  Since the constant solution $f_t(x,\e) = f(x,\e)$ solves the ODE, the uniqueness of the solution yields the equation $\pd{f_t}{t} (x, \e) = 0$.

Next, we show that there exists a $1$-parameter family of generating families $F_s(t, x, \e)$ between $F$ and $tf$ so that 
$F_s$ generates $\clag$ for all $s$.  Consider the path
$$F_s(t,x,\e) = \frac{t}{(1-s)t + st_-} F( (1-s)t + st_-, x, \e).$$
It is easy to verify that $F_0 = F$, $F_1 = tf$, $\Sigma_{F_s} = \rp \times \Sigma_f$, and that $F_s$ is a generating family for each $s$.  Then, since $F(t, x, \e) = tf(x, \e)$ on $\Sigma_{F_s}$, we see that $F_s$ generates $\clag$ for each $s$.

It now follows from an argument of Th\'eret \cite[Theorem 5.1]{theret:viterbo} that $F_1$ and $F_0$ are equivalent by a fiber-preserving diffeomorphism $\Phi_1$. For the reader's convenience, we sketch Th\'eret's argument.  The goal is to show that there exists a fiber-preserving isotopy $\Phi_s$ so that $F_s \circ \Phi_s= F_0$, for all $s \in [0, 1]$.  By differentiating this equation with respect to $s$, we get an equation for a vector field $X_s$ that generates this isotopy.  It is easy to find the solution for this $X_s$ outside the fiber critical set $\Sigma_{F_s} = \Sigma$ of $F_s$.  We then apply Hadamard's lemma to find a solution $X_s$ near $\Sigma$.  These two solutions are then glued together to produce a globally defined $X_s$ by the choice of an appropriate bump function.  The taming conditions on $F$ guarantee that the vector field $X_s$ will be integrable.
\end{proof}
 
Next, we consider the naturality of the cobordism map. As usual, we begin with a gf-compatible Lagrangian  cobordism $(\leg_-, f_-) \prec_{(\sclag, F)} (\leg_+, f_+)$ in the symplectization $\rr \times J^1M$.  First suppose that $\kappa^s_\pm$ are compactly supported contact isotopies of $J^1M$.  These isotopies may be extended to symplectic isotopies on the symplectization $\rr \times J^1M$ by the formula
$$S\kappa^s_\pm(t,x) = (a^s_\pm(x)t, \kappa^s_\pm(x)),$$
where $a^s_\pm$ are the scaling functions given by $(\kappa^s_\pm)^*\alpha = a^s_\pm \alpha$. 

Next, let $\phi^s$, $s \in [0, 1]$, be a symplectic isotopy of $\rr \times J^1M$ so that,
  for compact sets $I \subset \rr$ and $X \subset J^1M$,  we have:
\begin{itemize}
\item $\phi^s = \id$ on the complement of $\rr \times X$ and
\item  $\phi^s = S\kappa^s_\pm$ for (compactly supported) 
   contact isotopies $\kappa^s_\pm$ of $J^1M$ on the complement of
   $I \times J^1M$.
\end{itemize}

\begin{prop}[Naturality] \label{prop:cobord-naturality} 
  Given the conditions above, there exists a smooth, $1$-parameter family 
  of cobordisms $(\kappa^{s}(\leg_-), f_-^{s}) \prec_{(\phi^{s}(\theta(\clag), F^{s})} (\kappa^{s}(\leg_+), f_+^{s})$ so that:
 \begin{enumerate}
  \item $(F^0, f_-^0, f_+^0) \sim (F, f_-, f_+)$, and
  \item The  following diagram commutes for all $s$:
    \begin{equation} \label{eqn:nat-diag}
      \xymatrix{
        GH^k\left([f_-^0]\right) \ar[r]^{\Psi_{[F^0]}} & GH^*\left([f_+^0]\right)\\
        GH^k\left([f^s_-]\right) \ar[r]^{\Psi_{[F^s]}} \ar[u]^{(\kappa^s_-)^\#}_{\simeq} & GH^*\left([f^s_+]\right) \ar[u]^{(\kappa^s_+)^\#}_{\simeq}.
      }
    \end{equation}
  \end{enumerate}
\end{prop}

\begin{proof}[Proof of Existence of $(F^s, f_-^s, f_+^s)$] The construction of the generating families $(F^s, f_-^s, f_+^s)$ essentially follows from a construction of Chekanov \cite[\S5]{chv:quasi-fns}; we will discuss the necessary modifications.  As Chekanov proves, we may assume that $M = \rr^k$.  We denote the symplectic isotopy $\theta \circ \phi^s \circ \theta^{-1}$ of $T^*(\rr_+ \times \rr^k)$ by $\Phi$.  Since $\phi^s$ is well-behaved outside of a compact set, we may write it as a composition of $C^2$-small symplectomorphisms, and hence may simply assume that $\phi^s$ --- and hence $\Phi$ --- is $C^2$-close to the identity.  The key idea underlying Chekanov's proof of persistence is to repeatedly apply the fact that a generating family $F$ for $\clag$ and a generating family for a sufficiently small $\Phi$ (meaning a generating family for the image of the graph of $\Phi$) can be ``composed'' to obtain a generating family for $\Phi(\clag)$.

We now outline Chekanov's composition formula, partly for the reader's convenience, and partly to note the changes necessary to adjust his proof --- in which the generating family $F(t,x,\e)$ was assumed to be of the form $t f(x,\e)$ --- to our more general situation.  Let $\Gamma_\Phi \subset \overline{T^*(\rp \times \rr^k)} \times T^*(\rp \times \rr^k)$ denote the graph of $\Phi$, and 
 let $\widetilde{\Gamma}_\Phi = \sigma(\Gamma_\Phi)$ where $\sigma$ is the symplectic embedding given by:
$$\begin{aligned}
\sigma: \overline{T^*(\rp \times \rr^k)} \times T^*(\rp \times \rr^k) & \to T^*(\rr^{k+1} \times \rp \times \rr^{k}) \\
(t,q,u, p, T, Q, U, P) &\mapsto (u,p,T,Q,t,q,U,P)
\end{aligned}
$$
If $\Phi$ is sufficiently close to the identity, then
$\widetilde{\Gamma}_\Phi$ is the graph of the exact $1$-form $dG$ for a function
$G: \rr^{k+1} \times \rp \times \rr^{k} \to \rr$.

We can now combine $F$ and $G$ to obtain a generating family 
\begin{align*}
K: \rr_+ \times \rr^k \times (\rr_+ \times \rr^k \times \rr \times \rr^k \times \rr^N) &\to \rr \\
(s,y; r,x,t,q,\e) &\mapsto F(t,q,\e) + G(r,x,s,y)\\
& \quad \quad -rt-xq
\end{align*}
for $\Phi(\clag)$. We emphasize that the fiber variables are now $r, x, t, q$, and $\eta$.  It is straightforward to check that $K$ does, indeed, generate $\Phi(\clag)$.  In order to prove that $K$ is part of a compatible triple, however, we need to make some adjustments in the spirit of \cite[Lemma 5.6]{chv:quasi-fns}.

The first adjustment begins by defining constants $a$ and $b$ such that 
\begin{equation} \label{eqn:ab}
  a < \frac{\partial_r G}{s} < b;
\end{equation}
the proof of existence of such $a$ and $b$ is entirely similar to that in \cite[Lemma 5.7]{chv:quasi-fns}.  We now define a function $\tau: \rr \to \rr$ to be a smooth, non-decreasing function with the property that:
$$ \tau(x) = \begin{cases} a/2 & x <a/2, \\ x & a \leq x \leq b, \\ 2b & x > 2b, \end{cases}$$
which then can be used to define
$$K'(s,y; r,x,t,q,\e) = F( s \tau(t/s), q, \e) + G(r,x,s,y) -rt-xq.$$
To see that $K'$ generates the same Lagrangian as $K$, notice that on the fiber critical set of $K'$, we have $\partial_r G = t$.  Thus, in a neighborhood of the fiber critical set, we have $a < t/s < b$, so $s \tau (t/s) = t$, and hence $K=K'$.

The second adjustment uses the fiber-preserving diffeomorphism
$$\beta(s,y;r,x,t,q,\e) = (s,y; r, s x, st, q, \e/\tau(t)),$$
with $K'' = K' \circ \beta$.  The function $K''$ clearly still generates $\clag$.  Outside of a compact interval in the $\rr_+$ coordinate $s$, where $F(t,q,\e) = t f_\pm(q,\e)$ and, as noted in \cite[Proposition 5.5]{chv:quasi-fns}, $G(r,sx,s,y)$ is of the form  $s g_\pm(r,x,y)$, we see that 
$$K''(s,y;r,x,t,q,\e) = s \bigl( \tau(t) f_\pm(q, \e/\tau(t)) + g_\pm(r,x,y) -rt - xq\bigr).$$
Letting $k_\pm$ be the functions in the parentheses, we see that $(K'', k_-, k_+)$ is a compatible triple of generating families of $\Phi(\clag)$.  Note that outside of a compact set, the linear term in $\e$ of $k_\pm$ agrees with that of $f_\pm$.

It remains to show that the triple $(K'', k_-, k_+)$ is tame, i.e.\ that $K''$ is slicewise linear at infinity.  Suppose that for each $s \in \rp$, $F$ is equal to the linear function $A_s(\e)$.  If we can show that the quantity
\begin{equation} \label{eqn:K''}
\begin{split}
  B(s) &= \|K'' - \frac{1}{\tau(t)}A_{s\tau(t)}(\e) + rt+xq\|_\infty \\
  &\leq \|F(s \tau(t), q, \e/\tau(t)) - \frac{1}{\tau(t)}A_{s \tau(t)}(\e)\|_\infty + \|G(r,sx,s,y)\|_\infty
\end{split}
\end{equation}
is bounded above for each fixed $s$, then a slicewise application of Fuchs and Rutherford's proof of Lemma~\ref{lem:lin-diff}, together with Lemma~\ref{lem:lq-l}, shows that $K''$ is equivalent to a slicewise linear-at-infinity generating family.

To show that the quantity $B(s)$ in (\ref{eqn:K''}) is bounded for each fixed $s$, we begin by noting that since $G$ generates a slicewise compactly supported symplectomorphism, it must be constant outside of a compact set in each slice and hence the second term in (\ref{eqn:K''}) is bounded. For each fixed $t$, the fact that $F$ is slicewise linear-at-infinity shows that if $q$ or $\e$ grow large, then $F(s \tau(t), q, \e/\tau(t))$ agrees with $\frac{1}{\tau(t)}A_{s\tau(t)}(\e)$.  Finally, if $t$ lies outside a compact set in $\rr_+$, then $\tau(t)$ becomes constant, and hence $\|F(s \tau(t), q, \e/\tau(t)) - \frac{1}{\tau(t)}A_{s \tau(t)}(\e)\|_\infty$ is uniformly bounded for all $t$.  \end{proof}

\begin{proof}[Proof of the Commutativity of (\ref{eqn:nat-diag})] Let $\clag^s$ denote the image of $\phi^s(\overline L)$ in $T^*(\rp \times M)$.  We may assume that there is a $1$-parameter family of tame, compatible generating families $(F^s, f_-^s, f_+^s)$ for $\left(\clag^s, \kappa_-^s(\leg_-), \kappa_+^s(\leg_+)\right)$. Construct $\Delta^s$ from $F^s$ and a $1$-parameter family of shearing functions $H^s$ in $\mathcal H\left(\clag^s\right)$ so that the values $t_\pm$ are fixed for all $s$. Finally, choose $1$-parameter families $\Infin(s)$ and $\mu(s)$ satisfying Inequalities~(\ref{ineq:wgh}).  As in Inequalities~(\ref{ineq:vpm}), choose $v_\pm$ so that 
\begin{align*}
    \lambda_{\Infin(s)}(t) &< -\lmaxm^s & \text{for all }t \leq v_-,\\
    \lambda_{\mu(s)}(t) &> \lmaxp^s & \text{for all }t \geq v_+.  
\end{align*}

As noted after Corollary~\ref{cor:end-naturality}, the isomorphisms $(\kappa^s_\pm)^\#$ are constructed by applying the Critical Non-Crossing Lemma~\ref{lem:crit-non-crossing} to the difference functions $\delta^s_\pm$.  In particular, the maps $k^s_\pm$ underlying $(\kappa^s_\pm)^\#$ are simply compositions of positive and negative gradient flows of finitely many $\delta^s_\pm$ functions.  In this proof, we think of the gradients $\nabla \delta^s_\pm$ as projections of the gradients $\nabla \Delta^s$ restricted to $t=t_\pm$.  This allows us to extend $k^s_-$ to a map $\tilde{k}^s_-$ on $(\Delta^s)_{t_-}^\Infin \cup (\Delta^s)^\mu_{[t_-,t_+]}$ by looking at flows of the projections of $\nabla \Delta^s$ to each constant $t$ slice.  We can therefore form the map 
$$(\tilde{k}^s_-)^{-1} \circ \psi_{F^s} \circ k^s_+,$$
which is clearly homotopic to $\psi_F$.  The commutativity of the diagram (\ref{eqn:nat-diag}) follows.
\end{proof}

The final TQFT-like property that we will explore is functoriality, i.e.\ the behavior of the cobordism map under the gluing of cobordisms. In order to state the functoriality property, we begin by specifying what it means to glue together two generating families with matching end behavior.

\begin{defn} \label{defn:rho} Two  compatible triples of generating families
$(F^1, f_-^1, f_+^1)$ and $(F^2, f_-^2, f_-^2)$ are  \dfn{composable} if $f^1_+ \sim f^2_-$. 
\end{defn}

Given two composable triples, after applying stabilizations and fiber-preserving diffeomorphisms to $F^1$ and $F^2$, we can assume that $f^1_+ = f^2_-$. Let $t^1_\pm$ and $t^2_\pm$ be as in Definition~\ref{defn:shear-H}, and choose any $\rho > 0$ so that $t^2_- + \rho \geq t^1_+$.  Then define:
$$F^1 \#_\rho F^2 (t, x, \e) = \begin{cases}
  F^1(t,x,\e), & t \leq  t^2_- + \rho \\
  F^2(t-\rho, x, \e), & t \geq t^2_- + \rho.
\end{cases}$$ 
It is not hard to see that if, for $i = 1,2$, we have Lagrangian cobordisms $(\leg_-^i, f_-^i) \prec_{(\clag^i, F^i)} (\leg_+^i, f_+^i)$, and if $(F^1, f_-^1, f_+^1)$ and $(F^2, f_-^2, f_+^2)$ are composable, then $(F^1 \#_\rho F^2, f_-^1, f_+^2)$ is a tame, compatible triple of generating families
for the cobordism $ \leg_-^1 \prec_{\clag^1 \#_\rho \clag^2} \leg_+^2$, where $\clag^1 \#_\rho \clag^2$ denotes the glued Lagrangian which agrees with the image of $\leg_-^1$ when $t \leq t_-^1$ and with the image of $\leg_+^2$ when $t \geq t_2^+ + \rho$.  For $H \in \mathcal H(F^1 \#_\rho F^2)$, let $\bar{\Delta}$ be the associated sheared difference function.

\begin{prop}[Functoriality] \label{prop:funct} If $(F^1, f_-^1, v_+^1)$ and $(F^2, f_-^2, f_+^2)$ are
  composable, tame generating families, then:
    $$\Psi_{[F^1 \#_\rho F^2]} = \Psi_{[F^2]} \circ \Psi_{[F^1]}:
     \rgh{k}{[f^1_-]} \to \rgh{k}{[f^2_+]}$$
   In particular, the cobordism map for the glued Lagrangian does not
   depend on $\rho$.
\end{prop}

\begin{proof}
  As usual, choose $\Infin$ and $\mu$ that satisfy (\ref{ineq:wgh}), $v_\pm$ that satisfy (\ref{ineq:vpm}), and, for simplicity, $\rho$ so that $t_-^2 + \rho = t_+^1$.  As in the proof of Lemma~\ref{lem:cobord-indep-H}, we concentrate on the interval $[t^1_-, t^2_++\rho]$.  Recall that the map
  $$\psi_{F}: (\Delta^\Infin_{\{t_+\}}, \Delta^\mu_{\{t_+\}}) \to
  (\Delta^\Infin_{\{t_-\}} \cup \Delta^\mu_{[t_-, t_+]}, \Delta^\mu_{[t_-, t_+]})$$
  is defined by following the negative gradient flow of $\Delta$ until the first point at which the flowline intersects $\Delta^\mu$ or $\{t_-\} \times M \times \rr^{2N}$.  If we let $\tilde{\psi}_{F^1}$ be the map $\psi_{F^1}$ whose domain has been expanded to be the pair $(\Delta^\Infin_{\{t^1_+\}} \cup \Delta^\mu_{[t^1_+, t^2_+ + \rho]}, \Delta^\mu_{[t^1_+, t^2_+ + \rho]})$, then the description of the $\psi_F$ maps above easily implies that
\begin{equation} \label{eqn:comp-comp}
  \psi_{F^1 \# F^2} = \tilde{\psi}_{F^1} \circ \psi_{F^2}.
\end{equation}

The rest of the proof is completely analogous to that of Lemma~\ref{lem:cobord-indep-H}.
\end{proof}

\section{Examples and Open Questions}
\label{sec:ex}

 In this section, we give a number of applications of some of the main theorems of this paper.  Some of these results are
 stated in Theorem~\ref{thm:ex}.  The following applications fall in the general categories of negative twist knots, higher dimensional
 Legendrians, and  Legendrians with non-equivalent generating families.

Throughout this section, we use $\zz_2$ coefficients in this section both for ease of computation and so as to be able to use Fuchs and Rutherford's aforementioned connection between generating family homology and linearized contact homology in $\rr^3$.

\subsection{Negative Twist Knots}

Etnyre, Ng, and Vertesi classified Legendrian negative twist knots in $\rr^3$ up to Legendrian isotopy in \cite{env:twist}. They showed that any negative twist knot with maximal Thurston-Bennequin
  number  is isotopic to one of the forms pictured in Figure~\ref{fig:twist-knot}.  For a fixed twist knot, let $z^\pm$ denote the number of crossings of the form $Z^\pm$. 

\begin{figure}
  \centerline{\includegraphics[width=4in]{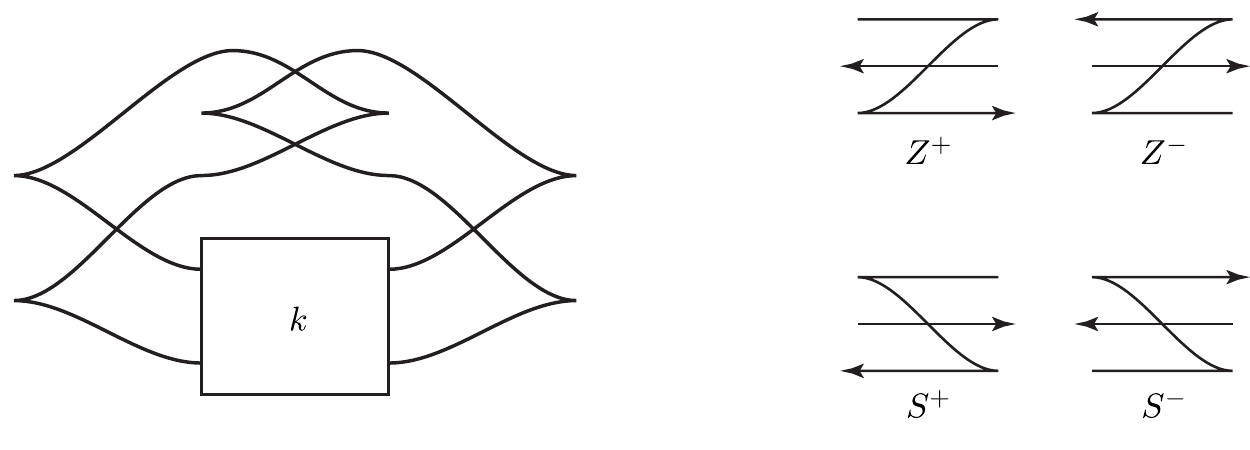}}
   \caption{Any Legendrian knot that is topologically a negative twist knot, $K_m$ with $m \leq 2$ ,
    is isotopic to one of the pictured knots, where the rectangle contains $k = |m +2|$ negative 
    half twists each of which is of type $Z^\pm$ or $S^\pm$.}
  \label{fig:twist-knot}
\end{figure} 

  Etnyre, Ng, and Vertesi show that
a Legendrian knot that is topologically an odd, negative twist knot $K_{-2n-1}$ is isotopic to one in the form of Figure~\ref{fig:twist-knot}
with $z^+ = n$ and that the Legendrian isotopy class of the knot is determined by $z^-$, where $0 \leq z^- < n$. 

  All of these knots possess graded rulings, and hence generating families. A straightforward computation of the linearized Legendrian contact homology, followed by an application of Fuchs and Rutherford's theorem \cite{f-r}, shows that there is a unique generating family cohomology whose Poincar\'e polynomial is $t^{-2z^--1} + t + t^{2z^--1}$.  
  Thus Corollary~\ref{cor:cyl-iso} shows that no two topologically equivalent odd, negative twist knots are gf-compatibly Lagrangian cobordant; further, Theorem~\ref{thm:filling-iso} shows that none of them possess a gf-compatible Lagrangian filling.  Note that only the latter result can be derived from classical invariants via Corollary~\ref{cor:chantraine}.

For a Legendrian knot of maximal Thurston-Bennequin invariant that is topologically $K_{-2n}$, 
the classification of \cite{env:twist} is somewhat different.  These knots are determined by the pair $(z_+,z_-)$, subject to the unique relation that $K_{(z^+,z^-)} \simeq K_{(n-1-z^+,n-1-z^-)}$.  In this case, arguments as above yield a unique generating family cohomology whose Poincar\'e polynomial is: $$t^{-2(z^++z^-+1-n)} + t + t^{2(z^++z^-+1-n)}.$$  Thus, Corollary~\ref{cor:cyl-iso} shows that if $z_0^++z_0^- \neq z_1^+ + z_1^-$, then there is no gf-compatible Lagrangian cobordism between $K_{(z_0^+,z_0^-)}$ and $K_{(z_1^+, z_1^-)}$.  Further, Theorem~\ref{thm:filling-iso} (but not the classical information in Corollary~\ref{cor:chantraine}) implies that a negative even twist knot can only have a gf-compatible Lagrangian filling if $z^++z^-=n-1$.  

This argument does leave an open question, however:

\begin{oques}
  If $z_0^++z_0^- = z_1^+ + z_1^-$, is there a gf-compatible Lagrangian concordance from one of these knots to the other? In particular,
  there are two non-equivalent Legendrian representatives of $K_{-6} = m(7_2)$ with maximal Thurston-Bennequin
  invariant that both have generating family polynomial $2 + t$; see, for example, \cite{ng:atlas}.    Are they Lagrangian cobordant?
\end{oques}

We suspect that the answer is ``no'' since Etnyre, Ng, and Vertesi distinguished these knots using the contact element in Heegaard-Floer knot homology, which itself should be an invariant of Lagrangian cobordism. See also \cite{sivek:monopole}.  The open question above may be generalized to:

\begin{oques}
  Do there exist Legendrian twist knots that are Lagrangian concordant but not Legendrian isotopic?  That is, is the relation of Lagrangian concordance among twist knots completely determined by Legendrian isotopy?  Even smooth concordances between twist knots are a
  topic of current research; see, for example, \cite{grs:concordance}.  
\end{oques}
 
\subsection{Higher Dimensional Legendrians}
Next, we move to higher dimensions.  We revisit Examples 3.1 and 4.9 of \cite{ees:high-d-geometry} from a generating family perspective in $J^1\rr^2$ (though the techniques here apply to $n>2$ as well).  Let $\leg_0$ be Legendrian sphere whose front diagram is the ``flying saucer'' of Figure~\ref{fig:high-d-ex}(a).  One may construct a linear-at-infinity generating family for $\leg_0$ by carefully spinning a generating family for the standard unknot in $J^1\rr$.  This Legendrian knot has a single Reeb chord of index $2$, and hence we have, for any generating family $f$, that  its Poincar\'e polynomial is $P_f(t) = t^2$.

Create another surface $\Lambda_1$ as follows; see Figure~\ref{fig:high-d-ex}. 
Squeeze the front of $\leg_0$ along a plane through the origin, producing a dumbbell shape as
shown in Figure~\ref{fig:high-d-ex}.  The region between the tubes can be stretched into a tube.   Finally, make the tube into a helical shape
so that in the resulting front, the dumbbell ends are overlapping.  It is not hard to explicitly construct a generating family $f:\rr^2 \times \rr \to \rr$ for $\leg_1$.  
  Further, \cite[Prop 4.10]{ees:high-d-geometry} shows that $\leg_0$ and $\leg_1$ have the same classical invariants.

In \cite{ees:high-d-geometry}, the authors compute that $\leg_1$ is  Legendrian isotopic to a  surface with seven Reeb chords
  with the following gradings:
\begin{align*}
  |a| = |b| = |c| &= 2, \\
  |d| &= 1, \\
  |e| &= 0, \\
  |f| = |g| &= -1.  
\end{align*} 
Proposition~\ref{prop:index} shows that these gradings are the same as those for the generating family cohomology.
Working over a field, we immediately see that $\dim \rgh{-1}{f} > 0$.  We conclude from Theorem~\ref{thm:cobord-les} 
that 
there is no tame, generating family compatible Lagrangian cobordism between $\Lambda_0$ and $\Lambda_1$, in either order.
Further, by 
Theorem~\ref{thm:filling-iso}, $\leg_1$ cannot have a generating family compatible Lagrangian filling.  
 
\subsection{Legendrians with Non-Equivalent Generating Families}
Finally, we return to $\rr^3$ to show how subtle the question of the existence of a compatible cobordism can be.

  Let $\leg$ be the Legendrian knot pictured in Figure~\ref{fig:ms}.  Melvin and Shrestha found that the Legendrian contact homology of $\leg$ (over $\zz_2$) has augmentations $\epsilon_1$ and $\epsilon_2$ with different linearized homologies \cite{melvin-shrestha}.  Assume, for the moment, that there exist generating families $f_1$ and $f_2$ for $\leg$ so that the Fuchs-Rutherford isomorphism yields $LCH^*_{\epsilon_i}(\leg) \simeq \rgh{*}{f_i}$, and hence that:
  $$P_{f_1}(t) = t^{-1} + 4 + 2t \quad \text{and} \quad P_{f_2}(t) = 2+t.$$
Corollary~\ref{cor:cyl-iso} implies that there is no gf-compatible concordance between $\leg$ and itself that interpolates between the generating families $f_1$ and $f_2$, while Theorem~\ref{thm:filling-iso} shows that only the second generating family can support a gf-compatible Lagrangian filling (in fact, a punctured torus).  This demonstrates the subtlety of the question of existence of gf-compatible cobordisms.

\begin{figure}
  \centerline{\includegraphics{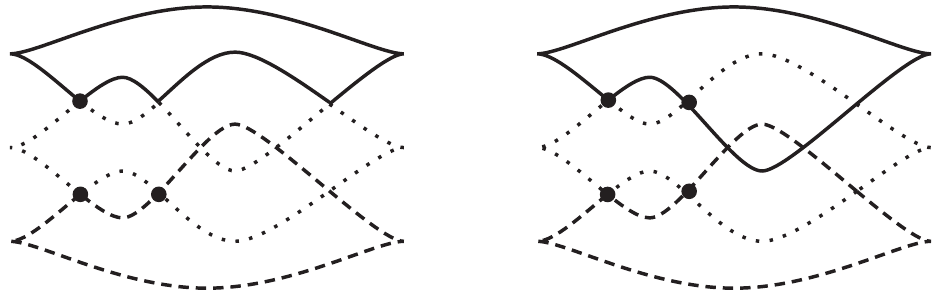}}
  \caption{Two graded normal rulings of a Legendrian $m(8_{21})$ knot and the corresponding augmentations of the Chekanov-Eliashberg DGA.  These rulings may be used to construct generating families that cannot be related by a gf-compatible concordance. }
  \label{fig:ms}
\end{figure} 

We now sketch a proof that the generating families $f_1$ and $f_2$ do, indeed, exist.  Beginning with the augmentations $\epsilon_1$ and $\epsilon_2$, the algorithm in \cite{rulings} produces both the graded normal rulings in Figure~\ref{fig:ms} and augmentations $\epsilon_1'$ and $\epsilon_2'$ on a diagram of $\leg$ that has a set of ``dips'' between every pair of adjacent crossings or cusps. On one hand, Fuchs and Rutherford show how to construct generating families $f_1$ and $f_2$ out of these rulings \cite[Section 3]{f-r}.  On the other, the new diagram does not quite yield the DGA used in \cite{f-r} to interpolate between the generating family and linearized contact homologies; it is necessary to add in a new dip to the left of the dips surrounding each crossing of the original diagram of $\leg$ that is augmented by $\epsilon_i'$.  The augmentations on the newly dipped diagrams are obtained by, in the language of \cite[Section 5.3.2]{henry:mcs}, extending the augmentation by $\mathcal{H}_{j+1,j}$, where the associated crossing occurs between the strands $j$ and $j+1$ of the front diagram.  

The resulting augmentations yield the linearized chain complexes that Fuchs and Rutherford define whose homology computes both the linearized contact homology associated to the $\epsilon_i$ and the generating family homology associated to the $f_i$.
 
It is reasonable to conjecture using the constructions of Melvin and Shrestha \cite{melvin-shrestha}, or alternatively using Sivek's Whitehead double construction \cite{sivek:bordered-dga}, that this example is but one of an infinite family of Legendrian knots with two --- and probably arbitrarily many --- non-compatibly-concordant generating families.


\bibliographystyle{amsplain} 
\bibliography{main}

\end{document}